\documentclass{amsart}
\usepackage{latexsym,amsxtra,amscd,ifthen}
\usepackage{amsfonts}
\usepackage{verbatim}
\usepackage{amsmath}
\usepackage{amsthm}
\usepackage{amssymb}
\usepackage{url,color}
\usepackage[arrow,matrix]{xy}

\usepackage{anysize}\marginsize{25mm}{25mm}{30mm}{30mm}
\allowdisplaybreaks[4]
\allowdisplaybreaks[4]

\numberwithin{equation}{section}
\numberwithin{table}{section}
\numberwithin{equation}{section}
\newtheorem{theorem}{Theorem}[section]
\newtheorem{lemma}[theorem]{Lemma}

\newtheorem{proposition}[theorem]{Proposition}

\newtheorem{corollary}[theorem]{Corollary}
\theoremstyle{definition}
\newtheorem{example}[theorem]{Example}
\newtheorem{remark}[theorem]{Remark}

\theoremstyle{definition}
\newtheorem{definition}[theorem]{Definition}

\usepackage{bbm}

\def \wv{\wp\varsigma}
\def\la{\lambda}
\def \ot{\otimes}

\def \As{\mathcal{A}ss}

\def \Com{\mathcal{C}om}
\def \Lie{\mathcal{L}ie}
\def \Pois{\mathcal{P}ois}
\def \O{\mathbb{O}}

\def\vs{\varsigma}
 \def\e{\epsilon}

\def \idf{\operatorname{idf}}

\def \fdf{\operatorname{fdf}}

\def \ui{\ucr{i}}

\def \ip{{\mathcal{P}}}

\def \iu{\!{\it{\Upsilon}}}
\def \ias{\!\mathcal{A}ss}
\def \dim{\operatorname{dim}}
\def\GKdim{\operatorname{GKdim}}

\def \Hom{\operatorname{Hom}}
\def \Ext{\operatorname{Ext}}
\def \der{\operatorname{der}}
\def \sf{\operatorname{sf}}
\def \ssf{\operatorname{Sf}}
\def \ad{\operatorname{ad}}
\def \aad{\operatorname{Ad}}
\def \ider{\operatorname{ider}}
\def \End{\operatorname{End}}

\def \HH{\operatorname{HH}}
\def \Aut{\operatorname{Aut}}
\def \IAut{\operatorname{IAut}}
\def \OAut{\operatorname{OAut}}
\def \id{\operatorname{id}}
\def \ker{\operatorname{Ker}}

\def \Vect{\operatorname{Vect}}

\def \ep{\epsilon}

\def \1{\mathbbm 1}

\def \k{\Bbbk}

\def \S{\mathbb{S}}

\def \deg{\operatorname{deg}}

\def \n{[n]}

\newcommand{\ucr}[1]{\underset{#1}{\circ}}
\def \fb{\mathbbm{b}}

\begin{document}
\title{Cohomological invariants of algebraic operads, I}

\subjclass[2010]{}

\keywords{}

\author{Y.-H. Bao}
\address{(Bao) School of Mathematical Sciences, Anhui University,
Hefei 230601, China}
\email{baoyh@ahu.edu.cn, yhbao@ustc.edu.cn}

\author{Y.-H. Wang}
\address{(Wang) School of Mathematics, Shanghai Key Laboratory
of Financial Information Technology, Shanghai University of
Finance and Economics, Shanghai 200433, China}
\email{yhw@mail.shufe.edu.cn}

\author{X.-W. Xu}
\address{(Xu) College of Mathematics, Jilin University,
Changchun 130012, China}
\email{xuxw@jlu.edu.cn}

\author{Y. Ye}
\address{(Ye) School of Mathematical Sciences,
University of Sciences and Technology of China, Hefei 230026, China}
\email{yeyu@ustc.edu.cn}

\author{J.J. Zhang}
\address{(Zhang) Department of Mathematics, Box 354350,
University of Washington, Seattle, Washington 98195, USA}
\email{zhang@math.washington.edu}

\author{Z.-B. Zhao}
\address{(Zhao) School of Mathematical Sciences, Anhui University,
Hefei 230601, China}
\email{zbzhao@ahu.edu.cn}

\begin{abstract}
We study various invariants, such as cohomology groups, derivations,
automorphisms and infinitesimal deformations, of algebraic operads
and show that $\As$, $\Com$, $\Lie$ and $\Pois$ are rigid or semirigid.
\end{abstract}

\subjclass[2000]{Primary 18D50, 55P48, 16E40, 53D55}



\keywords{Operad, cohomology group, rigidity, automorphism,
endomorphism, derivation, infinitesimal deformation}

\maketitle


\setcounter{section}{-1}

\section{Introduction}
\label{yysec0}

Several cohomology theories of operads have been introduced since 
1990s. The first was in Rezk's 1996 Ph.D. Dissertation \cite{Re}. 
Other theories via deformation complexes were introduced by 
Ne{\v c}as-Niessner in his Master Thesis \cite{NN} in 2010 and 
by Paljug in his Ph.D. Dissertation \cite{Pa} in 2015. Similar 
ideas appeared in many interesting papers including 
\cite{DK, FP, FW, KS, Ma1, Ma2}. Like the Hochschild 
cohomology of associative algebras, cohomologies of operads are 
extremely important to understand structures of algebraic operads. 
On the other hand, these cohomology theories have not been 
calculated for any explicit algebraic operads (as far as we 
know).

This paper concerns cohomology groups $H^i(\ip)$, for $i=0, 1,2$,
and related invariants such as derivations, automorphisms, and
infinitesimal deformations of an algebraic operad $\ip$. These 
basic invariants are already very interesting and reflect 
different aspects of algebraic operads. Algebraic operads encoding 
unital associative algebras (denoted by $\As$), unital commutative 
algebras (denoted by $\Com$), Lie algebras (denoted by $\Lie$), and 
unital Poisson algebras (denoted by $\Pois$) have been studied 
extensively by many authors for many years \cite{LV}. One of 
our main goals is to systematically calculate above mentioned 
invariants for $\As$, $\Com$, $\Lie$, $\Pois$ and other 
related operads.

Let $\Bbbk$ be a base field throughout the introduction and we 
consider symmetric operads over $\Bbbk$ though some ideas apply to 
plain (or non-symmetric) or set operads.

Below we recall the partial definition of an operad since we will 
use it as our basic definition. For the classical definition of a 
(non-symmetric or symmetric) operad, we refer to \cite[Chapter 5]{LV}.
To simplify notation we work with operads over $\Bbbk$-vector spaces.

\begin{definition}\cite[Section 5.3.4]{LV}
\label{yydef0.1}
An \emph{operad} consists of the following data:
\begin{enumerate}
\item[(i)]
a sequence $(\ip(n))_{n\geq 0}$ of right $\Bbbk \S_n$-modules,
where the right $\S_n$ actions are denoted by $\ast$,
\item[(ii)]
an element $\1\in \ip(1)$ called the \emph{identity},
\item[(iii)]
for all integers $m\ge 1$, $n \ge0$, and $1\le i\le m$,
a \emph{partial composition map}
\[-\ucr{i}-\colon \ip(m) \otimes \ip(n) \to \ip(m+n-1),\]
\end{enumerate}
satisfying the following axioms:
\begin{enumerate}
\item[(OP1$'$)] (Identity)
for $\theta\in \ip(n)$ and $1\leq i\leq n$,
\[\theta\ucr{i} \1 = \theta =\1\underset{1}{\circ} \theta;\]

\item[(OP2$'$)] (Associativity)
for $\lambda \in \ip(l)$, $\mu\in \ip(m)$ and $\nu\in \ip(n)$,
\[
\begin{cases}
(\la  \ucr{i} \mu) \underset{i-1+j}{\circ} \nu
=\la \ucr{i} (\mu \underset{j}{\circ} \nu),
& 1\le i\le l, 1\le j\le m,\\
(\la  \ucr{i} \mu) \underset{k-1+m}{\circ} \nu
=(\la \underset{k}{\circ}\nu) \ucr{i} \mu,
& 1\le i<k\le l;
\end{cases}
\]

\item[(OP3$'$)] (Equivariance)
for $\mu\in \ip(m)$, $\phi\in \S_m$, $\nu\in \ip(n)$ and
$\sigma\in \S_{n}$,
\[
\begin{cases}
\mu  \ucr{i} (\nu \ast \sigma)=
(\mu  \ucr{i} \nu)\ast \sigma',\\
(\mu\ast \phi)  \ucr{i} \nu=
(\mu  \underset{\phi(i)}{\circ} \nu)\ast \phi'',
\end{cases}
\]
where
\begin{align}
\label{E0.1.1}\tag{E0.1.1}
\sigma'= 1_m\ucr{i}\sigma, \ \ \textrm{and}\ \
\phi''= \phi\ucr{i} 1_n
\end{align}
are given by the partial composition in the associative algebra 
operad $\As$. Here $1_m$ denotes the identity permutation in 
$\S_m$ for all $m\geq 0$.
\end{enumerate}
\end{definition}


The approach of this paper is to consider an operad as a 
version of a noncommutative algebra and to use elementary 
invariants to define the first few cohomology groups. Following 
the classical definition of the Hochschild cohomology of 
associative algebras (respectively, Chevalley-Eilenberg 
cohomology of Lie algebras and Harrison cohomology of 
commutative algebras), the first cohomology group of an 
operad is defined in terms of derivations and the second 
cohomology group of an operad is defined in terms of 
infinitesimal deformations. The concepts of derivations 
and infinitesimal deformations are known for operads,
see for example \cite{Re, NN, Pa}, but it is a good idea 
to give a brief review here. In the introduction, we only 
recall some relevant definitions for the first cohomology 
group $H^1(\ip)$.

\begin{definition}
\label{yydef0.2}
Let $\ip$ be an operad.
\begin{enumerate}
\item[(1)] \cite[Section 6.1]{DR} \cite{Ja, Sc}
A {\it derivation} of $\ip$ is a $\Bbbk$-linear map $\partial$ of $\ip$
preserving the $\S_n$-module structure of $\ip$ such that
\[\partial (\mu \ucr{i} \nu)=
\partial (\mu) \ucr{i} \nu+
\mu \ucr{i} \partial (\nu)\]
for all $\mu \in \ip(m)$, $\nu\in \ip(n)$ and $1\leq i\leq m$.
The set of derivations of $\ip$ is denoted by $\der(\ip)$.

\item[(2)]
A derivation $\partial$ of $\ip$ is called {\it inner}
if there is an element $\la\in \ip(1)$ such that
\[\partial(\theta)=\la \underset{1}{\circ} \theta
 -\sum_{i=1}^{n} \theta\ucr{i} \la\]
for all $\theta\in \ip(n)$.
Such a derivation is denoted by $\ad_{\la}$. The set of 
inner derivations of $\ip$ is denoted by $\ider(\ip)$.

\item[(3)]
The \emph{first cohomology group} of $\ip$ is defined to 
be the $\Bbbk$-linear space
\[H^1(\ip):=\der(\ip)/\ider(\ip).\]

\item[(4)]
We say $\ip$ is {\it $\der$-rigid} (\emph{resp}., 
{\it $\der$-semirigid}) if $\dim_{\Bbbk} \der(\ip)\leq 1$ 
(\emph{resp}., $\dim_{\Bbbk} \der(\ip)\leq 2$).
\end{enumerate}
\end{definition}

The definitions of $H^0(\ip)$ and $H^2(\ip)$ can be found 
in Section \ref{yysec2}. One reason that we refer to 
Section \ref{yysec2} for the definition of $H^2(\ip)$
(and $H^2_{\ast}(\ip)$) is that it is quite involved.
Note that $H^2_{\ast}(\ip)$ corresponds to infinitesimal 
deformations that do not necessarily preserve the $\S$-action.
Therefore our $H^2_{\ast}(\ip)$ does not agree with most of 
the cohomology theories developed for operads. In general, 
$H^2_{\ast}(\ip)$ can not be computed by the deformation 
complex introduced in \cite{NN}. When $i\geq 3$, the 
definition of $H^i(\ip)$ is much more complicated, see \cite{BQZ}.
One advantage of our approach is that the first few 
cohomology groups are computable in some special cases 
due to the explicit formulation in the definition.

Going back to the derivations, it is easy to see that 
$\der(\ip)$ is a Lie algebra [Lemma \ref{yylem1.4}(3)].
If $\ip$ is locally finite and finitely generated,
then $\der(\ip)$ is finite dimensional [Lemma \ref{yylem3.1}(2)].
Recall that $\As$ (\emph{resp}., $\Com$, $\Lie$, $\Pois$)
denotes the operad encoding unital associative algebras
(\emph{resp}., unital commutative algebras, Lie algebras, 
unital Poisson algebras). These four are 
famous and commonly-used algebraic operads. The calculation 
of the derivations for these operads is quite easy.

\begin{theorem}
\label{yythm0.3}
\begin{enumerate}
\item[(1)]
$\der(\As)\cong \Bbbk$ and $H^1(\As)=0$.
\item[(2)]
$\der(\Com)\cong \Bbbk$ and $H^1(\Com)=0$.
\item[(3)]
$\der(\Lie)\cong \Bbbk$ and $H^1(\Lie)=0$.
\item[(4)]
$\der(\Pois)\cong \Bbbk\oplus \Bbbk$, the 2-dimensional abelian
Lie algebra, and $H^1(\Pois)\cong \Bbbk$.
\end{enumerate}
\end{theorem}

Theorem \ref{yythm0.3} indicates the (semi-)rigidity of algebraic operads
$\As$, $\Com$, $\Lie$ and $\Pois$ with respect to derivations.
There are similar statements concerning the automorphisms
and endomorphisms of these operads, see Section \ref{yysec3}.

For a general operad $\ip$, it is difficult to describe all 
derivations of $\ip$, however, we are able to work them out 
for operads introduced in \cite[Example 2.3]{BYZ}. Let $A$ 
be an associative algebra (unital or non-unital). The set 
of derivations of $A$ is denoted by $\der(A)$ and the first 
Hochschild cohomology of $A$ is denoted by $\HH^1(A)$
(see the beginning of Section \ref{yysec4}).
The Gelfand-Kirillov dimension of an operad is defined in 
\cite[(E0.0.3)]{BYZ}. The definition of 2-unitary operads
will be reviewed in Definition \ref{yydef1.1}(3).

\begin{theorem}
\label{yythm0.4}
Let $\ip$ be a 2-unitary operad of GKdimension two. Let
$\overline{\ip(1)}$ be the augmentation idea of the associative
algebra $\ip(1)$. Then there is a short exact sequence of Lie
algebras
$$0\to \Bbbk\to \der(\ip)\to \der(\overline{\ip(1)})\to 0.$$
As a consequence, $\der(\ip)\cong \Bbbk \rtimes \der(\overline{\ip(1)})$
and $H^1(\ip)\cong \HH^1(\overline{\ip(1)})$.
\end{theorem}

The above theorem fails for operads of GKdimension larger than
$2$ [Example \ref{yyex4.4}]. Similar to Theorem \ref{yythm0.4},
we have

\begin{theorem}
\label{yythm0.5}
Let $\ip$ be a 2-unitary operad that is left and right 
artinian and semiprime. Let $\overline{\ip(1)}$ be the 
augmentation ideal of the associative algebra $\ip(1)$.
Then there is a short exact sequence of Lie algebras
\[0\to \Bbbk\to \der(\ip)\to \der(\overline{\ip(1)})\to 0.\]
As a consequence, $\der(\ip)\cong \Bbbk \rtimes \der(\overline{\ip(1)})$
and $H^1(\ip)\cong \HH^1(\overline{\ip(1)})$.
\end{theorem}

Deformation theory of algebraic structures is controlled by their
cohomology theory, see \cite{NN, Pa, Re}. The second cohomology
$H^2_{\ast}(\ip)$ of an operad $\ip$ is closely related to infinitesimal
deformations of $\ip$. The definition of $H^0(\ip)$ and
$H^1(\ip)$ is definite, while there are two slightly different
diversions when we consider the second cohomology group.
One is $H^2_{\ast}(\ip)$ that measures all infinitesimal
deformations of $\ip$, and the other is $H^2(\ip)$
that measures those infinitesimal deformations of $\ip$
preserving the $\S$-action. In general, $H^2_{\ast}(\ip)
\not\cong H^2(\ip)$ [Example \ref{yyex2.8}]. Fortunately,
when ${\rm{char}}\; \Bbbk=0$, these two are the same
[Theorem \ref{yythm2.10}(3)]. One of the main results concerning
the second cohomology group is the following.

\begin{theorem}
\label{yythm0.6}
Suppose ${\rm{char}}\; \Bbbk=0$ in part {\rm{(3)}}.
\begin{enumerate}
\item[(1)]
$H^2_{\ast}(\As)=H^2(\As)=0$.
\item[(2)]
$H^2_{\ast}(\Com)=H^2(\Com)=0$.
\item[(3)]
$H^2_{\ast}(\Lie)=H^2(\Lie)=0$.
\item[(4)]
$H^2_{\ast}(\Pois)=H^2(\Pois)\cong \Bbbk$.
\end{enumerate}
\end{theorem}

See Theorem \ref{yythm6.8} and Remark \ref{yyrem7.7}
for other results about the second cohomology groups.
The proof of Theorem \ref{yythm0.6} is very much involved
and dependent on careful analysis of $\Pois$.

It is well-known from the work of Livernet-Loday \cite{LL} and
Markl-Remm \cite{MaR} that $\As$ is a deformation of $\Pois$.
(See \cite{Br} for deformations of a few other operads.)
Theorem \ref{yythm0.6}(4) suggests that a similar statement 
holds at the infinitesimal level. Theorem \ref{yythm0.6} also 
indicates that these four commonly-used operads are rigid or 
semirigid with respect to deformations, which proves a claim 
made by Kontsevich-Soibelman in \cite{KS}.

\begin{remark}
\label{yyrem0.7}
In positive characteristic, the second cohomology $H^2_{\ast}(\ip)$ 
of an operad $\ip$ is slightly different from other cohomology 
theories of operads in Rezk's Thesis \cite{Re}, or 
Ne{\v c}as-Niessner's Thesis \cite{NN}, or Paljug's Thesis \cite{Pa} 
or \cite{BQZ}. But there are some strong connections between our 
approach and others.
\end{remark}

In addition to the above results for specific operads, we have
a general result concerning $H^i(\ip)$ for $i\leq 2$.

\begin{theorem}
\label{yythm0.8}
Let $\ip$ be a locally finite operad.
Suppose that $\ip$ is generated by finitely many elements and
subject to a finite set of relations. Then the following hold.
\begin{enumerate}
\item[(1)]
The automorphism group $\Aut(\ip)$ is an algebraic group.
\item[(2)]
$H^i(\ip)$ is finite dimensional over $\Bbbk$ for $i=0,1,2$.
\end{enumerate}
\end{theorem}

\begin{remark}
\label{yyrem0.9}
In the sequel \cite{BWX}, we will continue to study other 
interesting cohomological invariants of operads related to 
universal deformation formulas and universal cohomology 
classes of $\ip$-algebras. 
\end{remark}

The paper is organized as follows. In Section \ref{yysec1} 
we review some basic definitions. The first couple of 
cohomology groups $H^i(\ip)$ are defined in Section 
\ref{yysec2}. Section \ref{yysec3} is devoted to computation 
of derivations and automorphisms of operads and Theorem
\ref{yythm0.3} is proved there. Theorems \ref{yythm0.4} and
\ref{yythm0.5} are proved in Section \ref{yysec4}. Several 
useful lemmas for $H^2_{\ast}(\ip)$ computation, as well as 
a part of proof of Theorem \ref{yythm0.8}, are given in 
Section \ref{yysec5}. In Sections \ref{yysec6} and 
\ref{yysec7}, we calculate $H^2_{\ast}(\ip)$ for several 
operads and prove Theorem \ref{yythm0.6}. In Section 
\ref{yysec8} we briefly recall the notion of a formal 
deformation of an operad.

We denote by $1_n$  the identity element in $\S_n$, and 
sometimes write a general permutation $\sigma\in \S_n$ as 
a product of disjoint cycles, namely, 
$\sigma=(i_1i_2\cdots i_{r})\cdots(j_1j_2\cdots j_s)_n$,
or $\sigma=(i_1i_2\cdots i_{r})\cdots(j_1j_2\cdots j_s)$ 
when no confusion arise. Note that the convention is 
different from \cite[Appendix]{BYZ}.

\section{Preliminaries}
\label{yysec1}

The operad theory is originated from the work of Boardman-Vogt \cite{BV}
and May \cite{May} in 1970s in homotopy theory. Operadic structures have 
been used in category theory, combinatorics, homological algebra, 
mathematical physics and topology. Many developments have been recorded 
in recent books \cite{Fr, LV, MSS}. We refer to \cite{LV} for basics of 
algebraic operads. In \cite{BYZ}, several new families of operads were 
constructed; in particular, 2-unitary operads of Gelfand-Kirillov dimension 
at most two are classified in \cite[Theorem 0.6]{BYZ}. The new examples 
given in \cite{BYZ} can be used to test various theories and questions
as we will do in this paper.

Throughout the rest of the paper we fix a base commutative ring $\Bbbk$.
In most of examples in this paper, we assume that $\Bbbk$ is a field.
All unadorned $\ot$ is $\ot_\k$.
The base category we
use is $\Vect$, the category of $\Bbbk$-modules (or the category
of vector spaces over $\Bbbk$ when $\Bbbk$ is a field). In this
section, we review definitions and basic facts about some special 
operads from \cite[Section 1]{BYZ}.


\begin{definition}
\label{yydef1.1}
\begin{enumerate}
\item[(1)] \cite[Section 2.2]{Fr}
An operad $\ip$ over $\Bbbk$ is called {\it unitary}
if $\ip(0)=\Bbbk \1_0\cong \Bbbk$.
Here $\1_0$ is a basis for $\ip(0)$ and is called a \emph{$0$-unit}
of $\ip$.
\item[(2)] An operad is called \emph{connected} if $\ip(1)=\Bbbk \1$.
\item[(3)] \cite[Definition 1.1(5)]{BYZ}
Let $\ip$ be a unitary operad with a fixed $0$-unit $\1_0
\in \ip(0)$. A unitary operad $\ip$ is called
{\it 2-unitary} if there is an element $\1_2\in \ip(2)$  (called a \emph{2-unit})
such that
\[\1_2 \underset{1}{\circ} \1_0 =\1=\1_2 \underset{2}{\circ} \1_0.\]
\end{enumerate}
\end{definition}

We remark that our notion of a 2-unitary operad is closely 
related to but different from the one of an operad with 
multiplication as introduced by Menich in \cite{Me}. Recall 
that an {\it operad with multiplication} is an operad 
equipped with an element $\mu\in \ip(2)$ (called the 
multiplication) and an element $e\in \ip(0)$ such that 
$$\mu\ucr{1}\mu=\mu\ucr{2}\mu, \quad {\text{and}} \quad
\mu\ucr{1}e=\1=\mu\ucr{2}e,$$ 
which is also called a \emph{strict} unital comp algebra 
in \cite{GS}. In fact, for a 2-unitary operad, we assume 
that $\ip(0)$ has $\Bbbk$-dimension 1, while  for an 
operad with multiplication, one needs the associativity 
of $\mu$.

We continue to give definitions, remarks
and comments, and to prove basic results, concerning derivations,
endomorphisms and automorphisms of operads. Recall that the
definition of a derivation is given in Definition \ref{yydef0.2}.
The following lemma is known and its proof is omitted.

\begin{lemma}
\label{yylem1.2}
Let $\ip$ be an operad.
\begin{enumerate}
\item[(1)] 
Let $c\in \Bbbk$. Then $\sf_c: \ip\to \ip$ determined by
\[\sf_c(\theta)=(n-1)c \theta,
\quad {\text{for all $\theta\in \ip(n)$}}\]
is a derivation.
\item[(2)] \cite[Proposition 6.1]{DR}
Let $\lambda\in \ip(1)$. Then
$\ad_\lambda: \ip\to \ip$ determined by
\[\ad_\lambda (\theta)= \lambda \underset{1}{\circ} \theta
-\sum_{i=1}^n \theta
\ucr{i} \lambda,
\quad {\text{for all $\theta\in \ip(n)$}}\]
is a derivation.
\end{enumerate}
\end{lemma}

It is easy to see that $\sf_{-c}=\ad_{c \1}$.

\begin{remark}
\label{yyrem1.3}
Let $\ip$ be an operad.
\begin{enumerate}
\item[(1)]
Let $f:\ip\to\ip$ be a $\Bbbk$-linear map of
$\ip$ (always preserving the degree).
Then $f$ is said to be {\it superfluous} if
there are a sequence of scalars $\{a_i\}_{i\geq 0}$ in
$\Bbbk$ such that
$$f(\mu)=a_m \mu$$
for all $\mu\in \ip(m)$.
\item[(2)]
The derivation $\sf_{c}$ given in Lemma
\ref{yylem1.2}(1) is superfluous.  The set of {\it superfluous
derivations} of the form $\sf_{c}$ is denoted by $\sf(\ip)$.

We mention that for the operads we are interested in this paper, 
such as $\ias, \Com, \Lie, \Pois$, a superfluous derivation is 
always of the form $\sf_c$. In fact, the statement is true if 
for all $m\ge 1, n\ge 0$, there exist $\mu\in \ip(m)$, 
$\nu\in \ip(n)$ and $1\le i\le m$, such that $\mu\ucr{i} \nu
\neq 0$. Let $\partial$ be a superfluous derivation of $\ip$ 
associated to $\{a_i\}_{i\ge 0}$. Then by our assumption, we 
have $a_{m+n-1}=a_m+a_n$ for all $m\ge 1$ and $m\ge 0$. 
Clearly, $a_1=0$ and $a_2=-a_0=c$. Taking $m=2$, we have 
$a_{n+1}=a_n+c$. It follows that $a_n=(n-1)c$ for all $n\ge 0$.
\item[(3)]
An inner derivation is also called {\it intrinsic} in
\cite[Section 6.1]{DR}.
\end{enumerate}
\end{remark}

\begin{lemma}
\label{yylem1.4}
Let $\ip$ be an operad.
\begin{enumerate}
\item[(1)]
$\der(\ip)$ is a Lie algebra over $\Bbbk$.
\item[(2)] \cite[Exercise 6.2]{DR}
If $\mu,\nu\in \ip(1)$, then
$[\ad_{\mu},\ad_{\nu}]=\ad_{[\mu,\nu]}$.
\item[(3)]
If $\partial\in \der(\ip)$ and $\ad_{\la}\in \ider(\ip)$
with $\la\in \ip(1)$ {\rm{[Definition \ref{yydef0.2}(2)]}},
then $[\partial, \ad_{\la}]=\ad_{\partial(\la)}$.
As a consequence, $\ider(\ip)$ is a Lie ideal of $\der(\ip)$
and $\der(\ip)/\ider(\ip)$ is a Lie algebra.
\item[(4)]
Every superfluous derivation of the form $\sf_{c}$ is in the
center of $\der(\ip)$.
\end{enumerate}
\end{lemma}

\begin{proof}
(1) Let $\partial, \partial'\in \der(\ip)$. Define
\[[\partial, \partial']:=\partial \circ \partial'-
\partial'\circ \partial.\]
We only need to prove $[\partial, \partial']\in \der(\ip)$. 
For $\mu \in \ip(m)$, $\nu\in \ip(n)$ and
$1\leq i \leq m$, we have
\begin{align*}
 [\partial, \partial'](\mu \ucr{i} \nu)
&=\partial \circ \partial' (\mu \ucr{i} \nu)
-\partial' \circ \partial (\mu \ucr{i} \nu)\\
&=\partial [\partial'(\mu) \ucr{i} \nu+\mu\ucr{i} \partial'(\nu)]
-\partial' [\partial(\mu) \ucr{i} \nu+\mu\ucr{i} \partial(\nu)]\\
&=\partial(\partial'(\mu)) \ucr{i} \nu+\partial'(\mu)\ucr{i} \partial(\nu)
+[\partial(\mu) \ucr{i} \partial'(\nu)+\mu\ucr{i} \partial(\partial'(\nu))]\\
&\quad
-[\partial'(\partial(\mu)) \ucr{i} \nu
+\partial(\mu)\ucr{i} \partial'(\nu)]
-[\partial'(\mu)\ucr{i} \partial(\nu)+\mu\ucr{i} \partial'(\partial(\nu))]\\
&=[\partial, \partial'](\mu) \ucr{i} \nu+\mu \ucr{i} [\partial, \partial'](\nu).
\end{align*}
Therefore $[\partial, \partial']$ is a derivation.

(3) For every $\theta\in \ip(n)$, we have
$$\begin{aligned}
\quad [\partial, \ad_{\la}](\theta)
&=\partial(\ad_{\la} (\theta))-\ad_{\la} (\partial(\theta))\\
&=\partial (\la\ucr{1} \theta -\sum_{i=1}^n \theta \ucr{i}
\la)-\la\ucr{1} \partial(\theta)+\sum_{i=1}^n
\partial(\theta) \underset{i}{\circ} \la\\
&=\partial(\la)\ucr{1} \theta+\la\ucr{1} \partial(\theta)
-\sum_{i=1}^n \partial(\theta) \ucr{i}
\la-\sum_{i=1}^n \theta \ucr{i}\partial(\la)
-\la\ucr{1} \partial(\theta)+\sum_{i=1}^n
\partial(\theta) \ucr{i} \la\\
&=\partial(\la)\ucr{1} \theta-\sum_{i=1}^n \theta \ucr{i}\partial(\la)
\\
&=\ad_{\partial(\la)}(\theta).
\end{aligned}
$$
The assertion follows.

(4) By part (3),
\[[\partial, \sf_{c}]=[\partial, \ad_{-c\1}]=
\ad_{\partial(-c\1)}=\ad_{0}=0.\]
The assertion follows.
\end{proof}

If $A$ is an associative algebra, then the set of
invertible elements in $A$ is denoted by $A^{\times}$.

\begin{definition}
\label{yydef1.5}
Let $\ip$ be an operad and $\Phi\in \Hom(\ip, \ip)$
be a $\Bbbk$-linear map.
\begin{enumerate}
\item[(1)]
We say $\Phi$ is an {\it endomorphism}
of $\ip$ if
\begin{enumerate}
\item[(1a)]
$\Phi(\1)=\1$,
\item[(1b)]
$\Phi(\mu \ast \sigma)=\Phi(\mu)\ast \sigma$ for all
$\mu\in \ip(m)$ and $\sigma \in \S_m$,
\item[(1c)]
$\Phi(\mu \ucr{i} \nu)=
\Phi(\mu)\ucr{i} \Phi(\nu)$
for all $\mu\in \ip(m)$, $\nu\in \ip(n)$ and $1\leq i\leq m$.
\end{enumerate}
The set of endomorphisms of $\ip$ is denoted by
$\End(\ip)$.
\item[(2)]
We say $\Phi$ is an {\it automorphism}
of $\ip$ if it is an endomorphism and is
invertible.
All automorphisms of $\ip$ form a group, denoted by
$\Aut(\ip)$.
\item[(3)]
Suppose $\Aut(\ip)$ is an algebraic group as in
Theorem \ref{yythm1.7} below. We say $\ip$ is
{\it $\Aut$-rigid}
(respectively, {\it $\Aut$-semirigid}) if
$\dim \Aut(\ip)\leq 1$ (respectively,
$\dim \Aut(\ip)\leq 2$).
\item[(4)]
If, there is a scalar $c\in \Bbbk^\times$, $\Phi$ satisfies
that
$$\Phi(\theta)= c^{n-1} \theta$$
for all $\theta\in \ip(n)$, then $\Phi$ is denoted by $\ssf_{c}$.
It is clear that $\ssf_{c}\in \Aut(\ip)$. The set of superfluous
automorphisms of the form $\ssf_{c}$ is denoted by $\ssf(\ip)$,
which is a subgroup of $\Aut(\ip)$.
\item[(5)]
An automorphism $\Phi$ is called {\it inner} if, there is an
invertible element $\la\in \ip(1)$ such that
$$\Phi(\theta)= (\la^{-1} \circ \theta) \circ (\la, \la,\cdots, \la)$$
for all $\theta\in \ip(n)$. In this case, $\Phi$ is denoted by
$\aad_{\la}$.
All inner automorphisms form a normal subgroup of $\Aut(\ip)$,
denoted by $\IAut(\ip)$.
The {\it outer automorphism} group of $\ip$ is defined to be
$$\OAut(\ip)=\Aut(\ip)/\IAut(\ip).$$
\end{enumerate}
\end{definition}

One can check the following lemma.

\begin{lemma}
\label{yylem1.6}
Let $\ip$ be an operad.
\begin{enumerate}
\item[(1)]
Let $G$ be a subgroup of $\Aut(\ip)$. Then
the fixed subspace
$$\ip^G:=\{x\in \ip\mid g(x)=x, \; \forall \;
g\in G\}$$
is a suboperad of $\ip$. It is called the {\it
fixed suboperad} under the $G$-action.
\item[(2)]
Let $L$ be a Lie subalgebra (or Lie subring) of $\der(\ip)$.
Then the fixed subspace
$$\ip^L:=\{x\in \ip\mid \partial(x)=0, \; \forall \;
\partial\in L\}$$
is a suboperad of $\ip$. It is called the {\it
fixed suboperad} under the $L$-action.
\end{enumerate}
\end{lemma}

To conclude this section we show that under some reasonable
hypotheses, $\Aut(\ip)$ is an algebraic group. Let $\O$ be the
category of operads. Let $\O_u$ (respectively, $\O_{2u}$) be the
category of unitary (respectively, 2-unitary) operads. For
finitely generated locally finite operads, we have the following
theorem. Starting with a set of generators $X$, an expression in $X$ is
a result of finite $\S$-actions and partial compositions in terms
of $X$. This expression is similar to an element in the free operad
generated by $X$. We use $e(X)$ to denote an expression in $X$.

\begin{theorem}
\label{yythm1.7}
Suppose $\Bbbk$ is a field. Let $\ip$ be a finitely generated and
locally finite operad. The following hold.
\begin{enumerate}
\item[(1)]
$\Aut_{\O}(\ip)$ is an algebraic group.
\item[(2)]
If $\ip$ is a unitary operad, then $\Aut_{\O_u}(\ip)$ is an
algebraic group and
$$\Aut_{\O_u}(\ip) \subseteq \Aut_{\O}(\ip).$$
\item[(3)]
If $\ip$ is a 2-unitary operad, then $\Aut_{\O_{2u}}(\ip)$ is
an algebraic group and
$$\Aut_{\O_{2u}}(\ip)\subseteq \Aut_{\O_{u}}(\ip) \subseteq
\Aut_{\O}(\ip).$$
\end{enumerate}
\end{theorem}

\begin{proof} (1) Let $\{x_{\alpha}\}_{\alpha\in I}$ be a $\Bbbk$-linear
basis of $\ip$ where $I$ is an index set that starts with $1,2,3,\cdots$.
Since $\ip$ is finitely generated and locally finite, there are two
positive integers $M$ and $N$ such that $X:=\{x_{\alpha}\}_{\alpha=1}^N$
(as a subset of $\{x_{\alpha}\}_{\alpha\in I}$) is a $\Bbbk$-linear
basis of $\bigoplus_{i\leq M} \ip(i)$ that generates $\ip$.  Since $X$
generates $\ip$, for each $\alpha>N$ (always in $I$), $x_{\alpha}$ can
be written as a fixed expression $e_{\alpha}(X)$ that involves
$\S_n$-actions and partial compositions and elements in $X$ and write
$$e_{\alpha}(X)=e_{\alpha}(x_1,\cdots, x_N).$$
If $\alpha\leq N$, we just set $e_{\alpha}(X)=x_{\alpha}$.

Let $\phi$ be an element in $\Aut_{\O}(\ip)$. Then, for each $\alpha
\leq N$, $\phi(x_{\alpha})=\sum_{j\leq N} m_{\alpha j} x_{j}$. Hence
we can define naturally a map $\phi\to \left(m_{ij}\right)_{N\times N}$
that is a group monomorphism from $\Aut_{\O}(\ip)$ to the general
linear group $GL(\Bbbk X)$. It remains to show that $\Aut_{\O}(\ip)$
is a closed variety of $GL(\Bbbk X)$. We will do next is to translate
the conditions of $\phi$ being an operad automorphism into a set of
polynomial identities on the matrix $\left(m_{ij}\right)_{N\times N}$
with fixed coefficients in $\Bbbk$.

Now write, for all $\alpha, \beta$ in $I$ and $\sigma$ in $\S$,
\begin{equation}
\notag
x_{\alpha}\ast \sigma =\sum_{\gamma} c_{\alpha, \sigma}^{\gamma}
x_{\gamma}
\end{equation}
and
\begin{equation}
\label{E1.7.1}\tag{E1.7.1}
x_{\alpha}\ucr{i} x_{\beta}=\sum_{\gamma}
c_{\alpha,i,\beta}^{\gamma}x_{\gamma}
\end{equation}
be the $\S_n$-action and partial composition rules with a set of
fixed scalars $c_{\alpha, \sigma}^{\gamma}$ and
$c_{\alpha,i,\beta}^{\gamma}$. And for each expression $e_{\alpha}(X)$
defined in the first paragraph and a collection $\{x_{i_1},\cdots,x_{i_N}\}$
with repetition where $X_{i_s}\in X$, we fix a linear combination
\begin{equation}
\notag
e_{\alpha}(x_{i_1},\cdots,x_{i_N})=
\begin{cases}
0, & \deg x_{i_s}\neq\deg x_s, \; {\text{for some $s$}},\\
\sum_{\gamma} c_{\alpha}^{\gamma}((i_s)) x_{\gamma},
& \deg x_{i_s}=\deg x_s, \; {\text{for all $s$}},
\end{cases}
\end{equation}
where $\{c_{\alpha}^{\gamma}((i_s))\}$ is a fixed set of scalars only
dependent on $\alpha$, $\gamma$ and $\{i_s\}_{s=1}^N$.

Applying $\phi$ to $e_{\alpha}(X)$ we have
$$\begin{aligned}
\phi(e_{\alpha}(X))&=e_{\alpha}(\phi(X))=e_{\alpha}(\sum_{i_1} m_{1 i_1}x_{i_1},\cdots,
\sum_{i_N} m_{N i_N}x_{i_N})\\
&=\sum c_{1}(m_{ij})_{\alpha, (i_s)} e_{\alpha}(x_{i_1},\cdots, x_{i_N})\\
&=\sum c(m_{ij})_{\alpha}^{\gamma} x_{\gamma}
\end{aligned}
$$
where $c_{1}(m_{ij})_{\alpha, (i_s)}$ and $c(m_{ij})_{\alpha}^{\gamma}$ are
polynomials in $\{m_{ij} \mid 1\leq i,j\leq N\}$ over $\Bbbk$ for
each $\alpha$ and $\gamma$. Note that the coefficients of
$c_{1}(m_{ij})_{\alpha, (i_s)}$ and $c(m_{ij})_{\alpha}^{\gamma}$ are
independent of the matrix $(m_{ij})_{N\times N}$. From the partial definition
[Definition \ref{yydef0.1}], we also have, for all $\alpha, \alpha',
\alpha''$ in $I$,
\begin{align}
\label{E1.7.2}\tag{E1.7.2}
x_{\alpha}\ucr{i} \1 &= x_{\alpha} =\1\underset{1}{\circ} x_{\alpha},\\
\label{E1.7.3}\tag{E1.7.3}
(x_{\alpha}  \ucr{i} x_{\alpha'}) \underset{i-1+j}{\circ} x_{\alpha''}
& =x_{\alpha} \ucr{i} (x_{\alpha'} \underset{j}{\circ} x_{\alpha''}),
 1\le i\le l, 1\le j\le m,\\
\label{E1.7.4}\tag{E1.7.4}
(x_{\alpha}  \ucr{i} x_{\alpha'}) \underset{k-1+m}{\circ} x_{\alpha''}
& =(x_{\alpha} \underset{k}{\circ}x_{\alpha''}) \ucr{i} x_{\alpha'},
 1\le i<k\le l,\\
\label{E1.7.5}\tag{E1.7.5}
x_{\alpha}  \ucr{i} (x_{\alpha'} \ast \sigma) &=
(x_{\alpha}  \ucr{i} x_{\alpha'})\ast \sigma',\\
\label{E1.7.6}\tag{E1.7.6}
(x_{\alpha}\ast \sigma)  \ucr{i} x_{\alpha'} & =
(x_{\alpha}  \underset{\sigma(i)}{\circ} x_{\alpha'})\ast \sigma'',
\end{align}
where
$\sigma'$ and $\sigma''$ are given in \eqref{E0.1.1}.

Note that $\phi: \Bbbk X\to \Bbbk X$ induces an endomorphism of $\ip$ precisely
when $\phi$ preserves the relations \eqref{E1.7.2} -- \eqref{E1.7.6}. These
give rise to a set of equations on the matrix $(m_{ij})_{N\times N}$. Just to
give one example, applying $\phi$ to \eqref{E1.7.3} and using the fact that
$$\phi(x_{\alpha})=
\phi(e_{\alpha}(X))=e_{\alpha}(\phi(X))=\sum c(m_{ij})_{\alpha}^{\gamma} x_{\gamma},$$
we have an equation
$$\begin{aligned}
\quad [(\sum c(m_{ij})_{\alpha}^{\gamma} x_{\gamma})
\ucr{i} &(\sum c(m_{ij})_{\alpha'}^{\gamma} x_{\gamma})]
\underset{i-1+j}{\circ} (\sum c(m_{ij})_{\alpha''}^{\gamma} x_{\gamma})\\
&=(\sum c(m_{ij})_{\alpha}^{\gamma} x_{\gamma})\ucr{i}
[(\sum c(m_{ij})_{\alpha'}^{\gamma} x_{\gamma})
\underset{j}{\circ} (\sum c(m_{ij})_{\alpha''}^{\gamma} x_{\gamma})].
\end{aligned}
$$
By using \eqref{E1.7.1} we obtain  a set of polynomial equations in $(m_{ij})_{N\times N}$.
The same argument applies to the other relations in
\eqref{E1.7.2} -- \eqref{E1.7.6}. Therefore $\Aut_{\O}(\ip)$
is a closed variety of $GL(\Bbbk X)$. The assertion follows.

(2) Since $\Aut_{\O_u}(\ip)=\{f\in \Aut_{\O}(\ip)\mid f(\1_0)=\1_0\}$,
$\Aut_{\O_u}(\ip)$ is a closed subgroup of $\Aut_{\O}(\ip)$. The assertion
follows.

(3) Since $\Aut_{\O_{2u}}(\ip)=\{f\in \Aut_{\O}(\ip)\mid f(\1_0)=\1_0, f(\1_2)=\1_2\}$,
$\Aut_{\O_{2u}}(\ip)$ is a closed subgroup of $\Aut_{\O}(\ip)$. The assertion
follows.
\end{proof}

\section{Definition of the first three cohomologies}
\label{yysec2}

In this section we define the first few cohomology
groups of an operad. Recall that, for $\lambda\in
\ip(1)$, $\ad_{\lambda}$ is defined to be the
derivation given in Lemma \ref{yylem1.2}(2).

\begin{definition}
\label{yydef2.1}
Let $\ip$ be an operad.
The \emph{0th cohomology group} of $\ip$ is defined to be
$$H^0(\ip):=\{\lambda\in \ip(1)\mid \ad_{\lambda}=0\}.$$
\end{definition}

It is easy to see that $H^0(\ip)$ is a $\Bbbk$-submodule
of the center of the associative algebra $\ip(1)$. This
implies that $H^0(\ip)$ has a natural (and trivial)
abelian Lie algebra structure. Calculation of $H^0(\ip)$
is relatively easy when $\ip$ is explicitly given, as the next
lemma indicates.

\begin{lemma}
\label{yylem2.2}
Let P be a connected operad over a field $\Bbbk$. 
Then $H^0(\ip)=0$ provided that one of the following holds: 
\begin{enumerate}
\item[(1)] 
$\operatorname{char}\Bbbk =0$ and 
$\ip(n)\neq 0$ for some $n\neq 1$; 
\item[(2)]
$\operatorname{char}\Bbbk =p>0$ and 
$\ip(n)\neq 0$ for some $n\not\equiv 1\ (\!\!\!\mod p)$.
\end{enumerate}
\end{lemma}

The proof is easy and is omitted. The above lemma applies
to $\As$, $\Com$, $\Lie$, $\Pois$ and many others.

\begin{definition}[Definition \ref{yydef0.2}(3)]
\label{yydef2.3}
Let $\ip$ be an operad.
The \emph{1st cohomology group} of $\ip$ is defined to be
$$H^1(\ip):=\der(\ip)/\ider(\ip)$$
the quotient $\Bbbk$-module of derivations modulo that of the
inner derivations.
\end{definition}

By Lemma \ref{yylem1.4}(3), $H^1(\ip)$ is a Lie algebra.
We will work out many examples of $H^1(\ip)$ in Sections \ref{yysec3}
and \ref{yysec4}. The 2nd cohomology is quite complicated and we need
several steps.

\begin{definition}
\label{yydef2.4}
Let $\ip$ be an operad.
\begin{enumerate}
\item[(1)]
A {\it $\wv$-collection} of $\ip$ means
a collection of $\Bbbk$-linear maps, for all $m\geq 1,n\geq 0$,
$$\{ \wp_i:=\wp_{m,n,i}: \ip(m)\otimes \ip(n)\to \ip(m+n-1), \;
\forall\;  1\leq i\leq m\}$$
together with a collection of $\Bbbk$-linear maps, for all $m\geq 0$,
$$\{\varsigma:=\varsigma_{m}: \ip(m)\otimes \Bbbk \S_m\to
\ip(m)\}.$$
The $\Bbbk$-module of all $\wv$-collections is denoted by $\wv(\ip)$.
\item[(2)]
A $\wv$-collection of $\ip$
is called a {\it 2-cocycle} if the
following hold.
\begin{enumerate}
\item[(2a)]
for $\lambda \in \ip(l)$, $\mu\in \ip(m)$ and $\nu\in \ip(n)$,
\begin{equation}
\label{E2.4.1}\tag{E2.4.1}
\wp_{i-1+j}(\la  \ucr{i} \mu, \nu)
+\wp_{i}(\la,\mu) \underset{i-1+j}{\circ} \nu
=\wp_i(\la, \mu \underset{j}{\circ} \nu)
+ \la \ucr{i} \wp_{j}(\mu, \nu),
\end{equation}
if $1\le i\le l, 1\le j\le m$, and
\begin{equation}
\label{E2.4.2}\tag{E2.4.2}
\wp_{k-1+m}(\la  \ucr{i} \mu,\nu)
+ \wp_{i}(\la, \mu) \underset{k-1+m}{\circ} \nu
=\wp_{i}(\la \underset{k}{\circ}\nu, \mu)
+\wp_{k}(\la,\nu) \ucr{i} \mu,
\end{equation}
if $1\le i<k\le l$.
\item[(2b)]
for $\mu\in \ip(m)$, $\phi\in \S_m$, $\nu\in \ip(n)$ and $\sigma\in \S_{n}$,
\begin{equation}
\label{E2.4.3}\tag{E2.4.3}
\wp_i(\mu, \nu \ast \sigma)+\mu\ui \varsigma(\nu,\sigma)=
\wp_i(\mu, \nu)\ast \sigma'+\varsigma(\mu\ui\nu, \sigma')
\end{equation}
\begin{equation}
\label{E2.4.4}\tag{E2.4.4}
\wp_i(\mu\ast \phi, \nu)+\varsigma(\mu, \phi)\ui \nu=
\wp_{\phi(i)}(\mu, \nu)\ast \phi''+\varsigma(\mu\underset{\phi(i)}{\circ}
\nu, \phi''),
\end{equation}
where $\sigma'=1_m\ucr{i} \sigma$ and $\phi=\phi\ucr{i} 1_n$
are given by the partial composition maps in $\As$.

\item[(2c)]
for $\mu\in \ip(m)$ and $\sigma, \tau \in \S_m$,
\begin{equation}
\label{E2.4.5}\tag{E2.4.5}
 \vs(\mu, \sigma\tau)=\vs(\mu\ast \sigma, \tau)+\vs(\mu, \sigma)\ast \tau.
\end{equation}
\end{enumerate}
We continue to use $\wp\varsigma$ to denote a 2-cocycle.
The set of 2-cocycles of $\ip$, which is a $\Bbbk$-module, is
denoted by $Z^2_{\ast}(\ip)$.
\item[(3)]
A $\wv$-collection (or a 2-cocycle) $\wp\varsigma\in \wv(\ip)$ is called
{\it superfluous} if
\begin{enumerate}
\item[(3a)]
there are scalars $a_{m,n,i}$
such that
\[\wp_{i}(\mu, \nu)=a_{m,n,i} \; \mu\ui\nu\]
for all $\mu\in \ip(m)$, $\nu\in \ip(n)$; and
\item[(3b)]
there are scalars $b_m$ such that
\[\varsigma(\mu, \sigma)=b_m \; \mu\ast \sigma\]
for all $\mu\in \ip(m)$ and $\sigma\in \S_m$.
\end{enumerate}
\item[(4)]
A 2-cocycle $\wv$ is called an {\it $\S$-2-cocycle} if $\varsigma=0$.
The set of $\S$-2-cocycles of $\ip$, which is a $\Bbbk$-module, is
denoted by $Z^2(\ip)$.
\item[(5)]
A 2-cocycle (or $\S$-2-cocycle) $\wv$ is called \emph{normalized}, 
if for all $\mu\in \ip(m)$ and $1\leq i\leq m$,
\[\wp_i(\mu,\1)=0=\wp_1(\1,\mu).\]
\end{enumerate}
\end{definition}

From \eqref{E2.4.5}, it is easily seen that $\varsigma(\mu, 1_m)=0$ 
for any $\mu\in \ip(m)$. Consequently, $b_m=0$ in {\rm (3b)} for 
all $m$ for a superfluous 2-cocycle $\wv$.

\begin{definition}
\label{yydef2.5}
Let $\wv$ and $\widetilde{\wv}$ be two 2-cocycles of $\ip$.
\begin{enumerate}
\item[(1)]
We say $\wv$ is a {\it 2-coboundary} if there
is a $\Bbbk$-linear map $\partial: \ip\to \ip$ such that,
for all $1\leq i\leq m$, $\mu\in \ip(m)$, $\nu\in \ip(n)$
and $\sigma\in \S_m$,
\begin{enumerate}
\item[(1a)]
\begin{equation}
\label{E2.5.1}\tag{E2.5.1}
\wp_i(\mu,\nu)=\partial(\mu\ui \nu)
-[\partial(\mu)\ui \nu+\mu \ui \partial(\nu)],
\end{equation}
\item[(1b)]
\begin{equation}
\label{E2.5.2}\tag{E2.5.2}
\varsigma(\mu, \sigma)=
\partial(\mu\ast \sigma)-\partial(\mu)\ast \sigma.
\end{equation}
\end{enumerate}
In this case we write $\wv=(\wp_i, \vs)$ as 
$\wv(\partial)=(\wp(\partial)_i, \vs(\partial))$.
The set of 2-coboundaries of $\ip$, which is a $\Bbbk$-module, is
denoted by $B^2_{\ast}(\ip)$.
\item[(2)]
A 2-coboundary $\wv$ is called an \emph{$\S$-2-coboundary} if $\varsigma=0$.
The set of $\S$-2-coboundaries of $\ip$, which is a $\Bbbk$-module, is
denoted by $B^2(\ip)$.
\item[(3)]
We say $\wv$ and $\widetilde{\wv}$ are {\it equivalent} if
$\wv-\widetilde{\wv}$ is a $2$-coboundary.
\end{enumerate}
\end{definition}

One can check that every 2-coboundary is a 2-cocycle (also
see Theorem \ref{yythm2.10}(2)). Hence $B^2_{\ast}(\ip)
\subseteq Z^2_{\ast}(\ip)$ and $B^2(\ip)
\subseteq Z^2(\ip)$.

\begin{lemma}
\label{yylem2.6}
Let $\ip$ be an operad. Then any 2-cocycle $($or $\S$-2-cocycle$)$
$\wv$ is equivalent to a normalized one.
\end{lemma}

\begin{proof}
Suppose that $\{\1\}\cup \{ x_{\alpha}\}_{\alpha\in I}$
is a $\Bbbk$-linear basis of $\ip$. Let 
$\partial\colon \ip\to \ip$ be a $\Bbbk$-linear map defined as
$\partial(\1)=-\wp_1(\1, \1)$ and $\partial(x_{\alpha})=0$ otherwise.
It is easy to see that $\partial(\mu\ast \sigma)=\partial(\mu)
\ast \sigma$ for all $\mu\in \ip(m)$ and $\sigma\in \S_m$.
Taking $\mu=\1=\nu$ in \eqref{E2.4.1}, we get
\[\wp_i(\la, \1)=\la\ucr{i}\wp_1(\1, \1),\]
and then
\[\wp_i(\la, \1)=\partial(\la\ucr{i}\1)-
(\partial(\la)\ucr{i}\1+\la\ucr{i}\partial(\1))\]
for all $\la\in \ip(l), 1\le i\le l$.
Taking $\la=\1=\mu$ in \eqref{E2.4.1}, we get
\[\wp_1(\1, \nu)=\wp_1(\1, \1)\ucr{1} \nu,\]
and then
\[\wp_1(\1, \nu)=\partial(\1\ucr{1}\nu)-
(\partial(\1)\ucr{1}\nu+\1\ucr{1}\partial(\nu))\]
for all $\nu\in \ip(n)$.
Set
\begin{align*}
\widetilde{\wp}_i(\mu, \nu)=&\wp_i(\mu, \nu)-
[\partial(\mu\ucr{i}\nu)-\partial(\mu)\ucr{i}\nu-\mu\ucr{i}\partial(\nu)]\\
\widetilde{\varsigma}(\mu, \sigma)
=& \varsigma(\mu, \sigma)-[\partial(\mu\ast\sigma)-\partial(\mu)\ast \sigma]
\end{align*}
for all $\mu\in \ip(m), \nu\in \ip(n)$ and $1\le i\le m$. Then 
$\widetilde{\wv}$ is equivalent to $\wv$ and satisfies
\[\widetilde{\wp}_i(\mu, \1)=0=\widetilde{\wp}_1(\1, \mu)\]
for all $\mu\in \ip(m)$ and $1\le i\le m$.
\end{proof}

\begin{definition}
\label{yydef2.7}
Let $\ip$ be an operad.
\begin{enumerate}
\item[(1)]
The \emph{2nd cohomology group} of $\ip$ is defined to be the $\Bbbk$-module
\[H^2_{\ast}(\ip):=Z^2_{\ast}(\ip)/B^2_{\ast}(\ip).\]
\item[(2)]
The \emph{2nd $\S$-cohomology group} of $\ip$ is defined to be the $\Bbbk$-module
\[H^2(\ip):=Z^2(\ip)/B^2(\ip).\]
\end{enumerate}
\end{definition}

In fact, in all computation, we can use normalized 
2-cocycles and normalized 2-coboundaries to
compute the 2nd cohomology group by Lemma \ref{yylem2.6}.

There is a natural $\Bbbk$-linear injective morphism from
$H^2(\ip)\to H^2_{\ast}(\ip)$. But, in general, $H^2(\ip)
\not\cong H^2_{\ast}(\ip)$, see the next example.

\begin{example}
\label{yyex2.8}
Suppose ${\rm{char}}\; \Bbbk=2$. Let $\ip$ be the
operad $\Bbbk \1\oplus \Bbbk\1_2$ where $\1$ is the
identity with composition law given by
\[\1_2 \underset{1}{\circ} \1_2=\1_2 \underset{2}{\circ}\1_2=0.\]
The $\S_2$-action on $\Bbbk \1_2$ is trivial.
It is easy to check that $\ip$ is an operad.

Now we define a 2-cocycle of $\ip$ by
\begin{align*}
\wp_i&=0 \qquad {\text{ for all $i$,}}\\
\varsigma_{m}&=0\qquad {\text{ for all $m\neq 2$,}}
\end{align*}
and
\[\varsigma_{2}(\1_2, \1_2)=0, \quad \varsigma_{2}(\1_2, (12))=\1_2.\]
By using the fact ${\rm{char}}\; \Bbbk=2$, it is straightforward
to show that the above defines a 2-cocycle of $\ip$, denoted
by $\wv$. By Definition \ref{yydef2.5}(1) and the fact
that $\S_2$-action on $\ip(2)$ is trivial, $\varsigma=0$
for every 2-coboundary. Therefore $\wv$ is not a
2-coboundary. In fact, with a few more lines of details, this shows that
\[H^2_{\ast}(\ip)=\Bbbk.\]
By Definition \ref{yydef2.4}(4), $\varsigma=0$ for every
$\S$-2-cocycle. Using Definition \ref{yydef2.4}(2a), one
sees that every $\S$-2-cocycle of $\ip$ is zero. This implies
that
\[H^2(\ip)=0.\]
Therefore $H^2(\ip)\not\cong H^2_{\ast}(\ip)$.

In view of Theorem \ref{yythm2.10} below, not every infinitesimal
deformation preserves $\S$-module structure.
\end{example}

Following ideas of Gerstenhaber \cite{Ge1, Ge2, Ge3}, the 2nd
cohomology group of $\ip$ should control infinitesimal
deformations of $\ip$. Let $\Bbbk$ be a base commutative ring
and let $\Bbbk[\epsilon]$ be the ring $\Bbbk[t]/(t^2)$.  Let
$\ip[\epsilon]$ denote $\ip\otimes \Bbbk[\epsilon]$, namely, for
every $n\geq 0$,
\[\ip[\epsilon](n):=\ip(n)\otimes \Bbbk[\epsilon].\]

In this paper we only consider infinitesimal deformations that
preserve the identity $\1\in \ip$, see Remark \ref{yyrem2.11}.

\begin{definition}
\label{yydef2.9}
Let $\ip$ be an operad over $\Bbbk$.
\begin{enumerate}
\item[(1)]
An {\it infinitesimal deformation} of $\ip$ is a $\Bbbk[\epsilon]$-linear
operadic structure on the $\Bbbk[\epsilon]$-module $\ip[\epsilon]$ such
that its partial composition, denoted by $-\ucr{i}^\ep -$,
and its $\S$-action, denoted by $\ast^\epsilon$, satisfy the following
conditions:
\begin{enumerate}
\item[(1a)]
for all $1\leq i\leq m$,
\[-\ucr{i}^\ep -:
\ip[\epsilon](m) \otimes_{\Bbbk[\epsilon]}\ip[\epsilon](n)
\to \ip[\epsilon](m+n-1)\]
\begin{equation}
\label{E2.9.1}\tag{E2.9.1}
\mu \ucr{i}^\ep \nu=
\mu \ucr{i} \nu +\wp_i(\mu,\nu) t.
\end{equation}
where $\wp_i:\ip(m) \otimes \ip(n) \to \ip(m+n-1)$
is $\Bbbk$-linear for all $m, n$ and $1\le i\le m$.

\item[(1b)]
for each $m\geq 0$,
\[\ast^\epsilon: \ip[\epsilon](m) \otimes_{\Bbbk[\epsilon]} 
\Bbbk[\epsilon] \S_m \to \ip[\epsilon](m)\]
\begin{equation}
\label{E2.9.2}\tag{E2.9.2}
\mu \ast^\epsilon \sigma =\mu \ast \sigma +\varsigma(\mu, \sigma) t.
\end{equation}
where $\varsigma: \ip(m)\otimes \Bbbk \S_m \to \ip(m)$
is $\Bbbk$-linear  for all $m\ge 0$.
\end{enumerate}
Note that every infinitesimal deformation produces a $\wv$-collection
by \eqref{E2.9.1} and \eqref{E2.9.2}.
\item[(2)]
Two infinitesimal deformations
$(\ip[\epsilon], \ucr{i}^{\epsilon}, \ast^{\epsilon})$ and 
$(\ip[\epsilon], \underset{i}{\widetilde{\circ}}^{\epsilon},
\widetilde{\ast}^{\epsilon})$ are called {\it equivalent} if 
there is an automorphism
$\Phi$ of $\Bbbk[\epsilon]$-modules $\ip[\epsilon]$ such that
\begin{enumerate}
\item[(2a)]
$\Phi$ is an isomorphism of operads from
$(\ip[\epsilon], \ucr{i}^{\epsilon}, \ast^{\epsilon})$ to 
$(\ip[\epsilon], \underset{i}{\widetilde{\circ}}^{\epsilon},
\widetilde{\ast}^{\epsilon})$, and
\item[(2b)]
there is a $\Bbbk$-linear map $\partial:\ip\to\ip$ such that
$\Phi(\mu)=\mu+\partial(\mu) t$ for all $\mu\in \ip(m)$.
\end{enumerate}
\item[(3)]
An infinitesimal deformation is called {\it trivial} if it is equivalent to
the one defined by zero $\wv$-collection.
\item[(4)]
The set of infinitesimal deformations of $\ip$ modulo the equivalent relation
defined in part (2) is denoted by $\idf (\ip)$.
\end{enumerate}
\end{definition}

\begin{theorem}
\label{yythm2.10}
The following hold.
\begin{enumerate}
\item[(1)]
There is a one-to-one correspondence between $Z^2_{\ast}(\ip)$ and
the set of infinitesimal deformations of $\ip$ via
\eqref{E2.9.1} and \eqref{E2.9.2}.
\item[(2)]
Under the correspondence in part {\rm{(1)}}, a 2-cocycle is a
2-coboundary if and only if the corresponding infinitesimal deformation
is trivial.
\item[(3)]
There is a natural isomorphism between $\idf (\ip)$ and
$H^2_{\ast}(\ip)$.
\end{enumerate}
\end{theorem}

\begin{proof}
(1) For any 2-cocycle $\wv$ of $\ip$, one can define the
$\Bbbk[\epsilon]$-linear operadic structure on $\ip[\epsilon]$,
with the partial composition $-\ucr{i}^\ep -$
given by \eqref{E2.9.1} and $\S$-actions $\ast^\epsilon$ given
by \eqref{E2.9.2}. In fact, using \eqref{E2.4.1} and \eqref{E2.4.2},
we can check that the partial composition
$-\ucr{i}^\ep -$ satisfy the equations in (OP2$'$),
and using \eqref{E2.4.3}-\eqref{E2.4.5}, we can check that the
equations in (OP3$'$) holds. The equations in (OP1$'$)
follows from Definition \ref{yydef2.4}(2a). Therefore, each
2-cocycle $\wv$ determines an infinitesimal deformation of $\ip$.

Conversely, every infinitesimal deformation of $\ip$ produces
a natural $\wv$-collection by \eqref{E2.9.1} and \eqref{E2.9.2}.
Since $\ast^\epsilon$ is an $\S$-action, \eqref{E2.4.5}
holds. By (OP$2'$) and (OP$3'$) for the $\Bbbk[\epsilon]$-linear
operad $\ip[\epsilon]$, we immediately obtain the equations
\eqref{E2.4.1}, \eqref{E2.4.2}, \eqref{E2.4.3} and \eqref{E2.4.4},
which means that this $\wv$-collection is a 2-cocycle.

(2) Let $\wv$ be a 2-coboundary. Then there exists a $\Bbbk$-linear
map $\partial\colon \ip \to \ip$ such that
\begin{equation}
\label{E2.10.1}\tag{E2.10.1}
\wp_i(\mu, \nu) =\partial(\mu)\ucr{i} \nu
+\mu\ucr{i} \partial(\nu)-\partial(\mu\ucr{i} \nu)
\end{equation}
and
\[\varsigma(\mu, \sigma)=
\partial(\mu)\ast \sigma-\partial(\mu\ast \sigma),\]
for any $\mu\in \ip(m), \nu\in \ip(n), \sigma\in \S_m$, which
determines an infinitesimal deformation
$(\ip[\epsilon], -\ucr{i}^\epsilon-, \ast^\epsilon)$
of $\ip$ by the part (1).

Consider the $\Bbbk[\epsilon]$-linear map
$\Phi\colon \ip[\epsilon] \to \ip[\epsilon]$ given by
\[\Phi(\mu)=\mu+\partial(\mu) t\]
for any $\mu\in \ip(m)$. Clearly, $\Phi$ is a $\Bbbk[\e]$-automorphism
of $\ip[\epsilon]$. The zero $\wv$-collection yields the trivial
infinitesimal deformation, whose partial composition is still
denoted by $-\ucr{i}-$. By \eqref{E2.9.1}, we have
\begin{align*}
\Phi(\mu)\ucr{i} \Phi(\nu)&=(\mu+\partial(\mu)t)
\ucr{i} (\nu+\partial(\nu)t) \\
&=\mu\ucr{i} \nu+(\partial(\mu) \ucr{i}
\nu + \mu \ucr{i} \partial(\nu))t\\
&= \mu\ucr{i} \nu+(\partial(\mu \ucr{i}\nu)
+\wp_i(\mu, \nu))t\\
& =\Phi(\mu \ucr{i}^\epsilon \nu)
\end{align*}
and
\begin{align*}
\Phi(\mu\ast^\epsilon \sigma)&= \Phi(\mu\ast \sigma+\varsigma(\mu, \sigma)t)\\
&= \mu\ast \sigma+(\partial(\mu\ast\sigma)+\varsigma(\mu, \sigma))t\\
&= \mu\ast \sigma+(\partial(\mu)\ast \sigma)t\\
&= \Phi(\mu)\ast \sigma
\end{align*}
It follows that $\Phi$ is an isomorphism of operads, and
$(\ip[\epsilon], \ucr{i}^\epsilon, \ast^\epsilon)$
is a trivial infinitesimal deformation.

Conversely, if the infinitesimal deformation
$(\ip[\epsilon], \ucr{i}^\epsilon, \ast^\epsilon)$
defined by the 2-cocycle $\wv=(\wp_i, \varsigma)$ is trivial, 
then there exists an isomorphism $\Phi$ from
$(\ip[\epsilon], \ucr{i}^\epsilon, \ast^\epsilon)$
to the one deformed by zero $\wv$-collection, such that 
$\Phi(\mu)=\mu+\partial(\mu)t$ for all $\mu\in \ip(m)$,
where $\partial\colon \ip \to \ip$ is a $\Bbbk$-linear 
map. By easy calculation,
we know that the equations \eqref{E2.5.1}-\eqref{E2.5.2} hold, and
therefore $\wv=(\wp_i, \varsigma)$ is a 2-coboundary.

(3) It is obvious by the parts (1) and (2).
\end{proof}

\begin{remark}
\label{yyrem2.11}
From Lemma \ref{yylem2.6}, it is easily seen that each 
infinitesimal deformation that do not preserve the identity
$\1\in \ip(1)$ is equivalent to one preserving the 
identity. 
\end{remark}

\begin{theorem}
\label{yythm2.12}
Let $\ip$ be an operad over a field $\Bbbk$ and let
$\wv$ be a 2-cocycle of $\ip$.
\begin{enumerate}
\item[(1)]
If $\Ext^1_{\Bbbk \S_n}(\ip(n),\ip(n))=0$ for a fixed integer
$n$, then $\wv$ is equivalent to a 2-cocycle with $\varsigma_n=0$.
\item[(2)]
Suppose that $\Ext^1_{\Bbbk \S_m}(\ip(m),\ip(m))=0$ for all $m$.
Then $\wv$ is equivalent to an $\S$-2-cocycle.
\item[(3)]
Suppose that $\Ext^1_{\Bbbk \S_m}(\ip(m),\ip(m))=0$ for all $m$.
{\rm{(}}This is automatic if $\Bbbk$ is a field of
characteristic zero.{\rm{)}} Then
$H^2_{\ast}(\ip)\cong H^2(\ip)$.
\end{enumerate}
\end{theorem}

\begin{proof}
(1) For each fixed $n$, the equation \eqref{E2.4.5} implies that
$\varsigma_n$ is a 1-cocycle in the complex
\[X^\bullet :=\Hom_{\Bbbk \S_n}(\ip(n) \otimes_{\Bbbk \S_n} 
C^\bullet(\Bbbk \S_n), \ip(n)),\]
where $C^\bullet(\Bbbk \S_n)$ 
is the Hochschild cochain complex of the group algebra 
$\Bbbk \S_n$.

By \cite[Lemma 9.1.9]{We} and our assumption,
the first Hochschild cohomology group
\[\HH^1(\Bbbk\S_n, \End_\Bbbk(\ip(n))=\Ext_{\Bbbk\S_n}^1(\ip(n), \ip(n))=0.\]
Therefore, the $\Bbbk\S_n$-module $\ip(n)$ is rigid, and
$\varsigma_n$ being a 1-coboundary is equivalent to that
there is an $\partial_n:\ip(n)\to \ip(n)$ such that
$\varsigma_n(\mu, \sigma)=\partial_n(\mu\ast \sigma)
-\partial_n(\mu)\ast \sigma$ for all $\mu\in \ip(n)$ and
$\sigma\in \S_{n}$. 
Let $\partial$ be the collection $(\partial_m)_{m\geq 0}$ where
\[\partial_m=
\begin{cases}
0,          & m\neq n,\\
\partial_n, & m=n.
\end{cases}\]
Define $(\widetilde{\wp}, \widetilde{\varsigma})$ by
\begin{align*}
\widetilde{\wp}(\mu,\nu) &= \wp(\mu, \nu) -(\partial(\mu\ui \nu)
-\partial(\mu)\ui \nu-\mu \ui \partial(\nu))\\
\widetilde{\varsigma}(\mu, \sigma) &=\varsigma(\mu, \sigma)
-(\partial(\mu\ast \sigma)-\partial(\mu)\ast \sigma).
\end{align*}
Then $(\widetilde{\wp}, \widetilde{\varsigma})$ is equivalent
to $(\wp, \varsigma)$ by Definition \ref{yydef2.5}.
By the choice of $\partial$, $\widetilde{\varsigma}_n=0$ as
desired.

(2) By the proof of part (1), since
$\Ext^1_{\Bbbk \S_m}(\ip(m),\ip(m))=0$, there is an
$\partial_m$ such that there is an $\partial_m:\ip(m)\to \ip(m)$
such that $\varsigma_m(\mu, \sigma)=\partial_m(\mu\ast \sigma)
-\partial_m(\mu)\ast \sigma$ for all $\mu\in \ip(m)$ and
$\sigma\in \S_{m}$. Let $\partial$ be the collection
$(\partial_m)_{m\geq 0}$. Define
$(\widetilde{\wp}, \widetilde{\varsigma})$ by
$$\begin{aligned}
\widetilde{\wp}(\mu,\nu) &= \wp(\mu, \nu) -(\partial(\mu\ui \nu)
-\partial(\mu)\ui \nu-\mu \ui \partial(\nu))\\
\widetilde{\varsigma}(\mu, \sigma) &=\varsigma(\mu, \sigma)
-(\partial(\mu\ast \sigma)-\partial(\mu)\ast \sigma).
\end{aligned}
$$
Then $(\widetilde{\wp}, \widetilde{\varsigma})$ is equivalent
to $(\wp, \varsigma)$ by Definition \ref{yydef2.5}.
By the choice of $\partial$, $\widetilde{\varsigma}_m=0$ for
all $m$. So $(\widetilde{\wp}, \widetilde{\varsigma})$ is
an $\S$-2-cocycle.

(3) When $\Bbbk$ is a field of characteristic zero,
$\Bbbk \S_m$ is semisimple for each $m$. As a consequence, for every
$m$, $\Ext^1_{\Bbbk \S_m}(\ip(m),\ip(m))=0$.

By part (2), every 2-cocycle is equivalent to an $\S$-2-cocycle.
Consider the map $Z^2(\ip)\to Z^2_{\ast}(\ip)\to H^2_{\ast}(\ip)$.
Then this map is surjective. Therefore it induces an isomorphism
$$H^2(\ip)=Z^2(\ip)/B^2(\ip)=
Z^2(\ip)/(Z^2(\ip)\cap B^2_{\ast}(\ip))
\xrightarrow{\quad \cong \quad} Z^2_{\ast}(\ip)/B^2_{\ast}(\ip)=H^2_{\ast}(\ip).$$
\end{proof}

If $(\wp, \varsigma)$ is an $\S$-2-cocycle, then the deformation
equations \eqref{E2.4.1}-\eqref{E2.4.4} can be rewritten as
follows (and \eqref{E2.4.5} is automatic):
for $\lambda \in \ip(l)$, $\mu\in \ip(m)$ and $\nu\in \ip(n)$,
\begin{equation}
\label{E2.11.1}\tag{E2.11.1}
\wp_{i-1+j}(\la  \ucr{i} \mu, \nu)
+\wp_{i}(\la,\mu) \underset{i-1+j}{\circ} \nu
=\wp_i(\la, \mu \underset{j}{\circ} \nu)
+ \la \ucr{i} \wp_{j}(\mu, \nu),
\end{equation}
if $1\le i\le l, 1\le j\le m$, and
\begin{equation}
\label{E2.11.2}\tag{E2.11.2}
\wp_{k-1+m}(\la  \ucr{i} \mu,\nu)
+ \wp_{i}(\la, \mu) \underset{k-1+m}{\circ} \nu
=\wp_{i}(\la \underset{k}{\circ}\nu, \mu)
+\wp_{k}(\la,\nu) \ucr{i} \mu,
\end{equation}
if $1\le i<k\le l$, for $\mu\in \ip(m)$, $\phi\in \S_m$,
$\nu\in \ip(n)$ and $\sigma\in \S_{n}$,
\begin{equation}
\label{E2.11.3}\tag{E2.11.3}
\wp_i(\mu, \nu \ast \sigma)=\wp_i(\mu, \nu)\ast \sigma'
\end{equation}
\begin{equation}
\label{E2.11.4}\tag{E2.11.4}
\wp_i(\mu\ast \phi, \nu)=\wp_{\phi(i)}(\mu, \nu)\ast \phi'',
\end{equation}
where $\sigma'$ and $\phi''$ are given in \eqref{E0.1.1}.

In view of Theorem \ref{yythm2.10}(3), we make the following
definition.

\begin{definition}
\label{yydef2.13}
Let $\ip$ be an operad. We say $\ip$ is {\it $\idf$-rigid}
(resp., {\it $\idf$-semirigid}) if $H^2_{\ast}(\ip)=0$
(resp., $H^2_{\ast}(\ip)=\Bbbk$).
\end{definition}

\section{Calculation of derivations, automorphisms and $H^1$}
\label{yysec3}

First we recall the definition of ${^k \iu}$
for a unitary operad $\ip$ from \cite{BYZ}. Let $\n$
be the set $\{1, \cdots, n\}$ and $I$ be a subset of $\n$.
Let $\chi_{I}$ be the characteristic function of $I$, i.e.
$\chi_{I}(i)=1$ for $i\in I$ and $\chi_{I}(i)=0$ otherwise.
Then one defines the {\it restriction operator}
$\pi^I: \ip(n)\to \ip(s)$, where $s=|I|$, by
\[ \pi^I(\theta)=
\theta\circ(\1_{\chi_{I}(1)}, \cdots, \1_{\chi_{I}(n)})
\]
for all $\theta\in \ip(n)$, see \cite[Section 2.2.1]{Fr} or 
\cite[Section 2.2]{BYZ}. For $k\geq 1$, the
{\it $k$-th truncation ideal} of $\ip$, denoted by $^k \iu$, is
defined by
\begin{equation}
\label{E3.0.1}\tag{E3.0.1}
^k\iu(n) =\begin{cases}
\bigcap\limits_{I\subset \n,\, |I|= k-1} \ker\pi^I, & \text{if }n\geq k;\\
\quad\ \ 0, & \text{if } n<k.
\end{cases}
\end{equation}
By convention, ${^0 \iu}=\ip$. See \cite{BYZ} for several
applications of the truncation ideals ${^k \iu}$.

The aim of this section is to calculate invariants such as
derivations, automorphisms and $H^1$ for the following
operads
$$\Com, \quad \As, \quad \As/{^k \iu}, \quad \Lie,
\quad \Pois, \quad \Pois/{^k \iu}$$
for $k\geq 1$.

A derivation $\partial$ of $\ip$ is called
{\it locally nilpotent} if for every $a\in \ip$, there is
an $n$ such that $\partial^n(a)=0$. We start with two easy
lemmas.

\begin{lemma}
\label{yylem3.1}
Let $\ip$ be an operad and let $\partial$ be a derivation of
$\ip$.
\begin{enumerate}
\item[(1)]
Let $T$ be a generating set of $\ip$. If $\partial$ is zero
when restricted to $T$, then $\partial$ is zero.
\item[(2)]
Suppose $\Bbbk$ is a field. If $\ip$ is locally finite and
finitely generated, then $\der(\ip)$ is a finite dimensional
Lie algebra.
\item[(3)]
Suppose $\ip$ is $2$-unitary with a $2$-unit $\1_2$.
If $\partial(\1_0)=-c\1_0$ for some $c\in \Bbbk$, then
$\partial(\1_2)-c\1_2\in {^2\iu}(2)$.
\item[(4)]
Suppose ${\mathbb Q}\subseteq \Bbbk$.
If $\partial$ is a locally nilpotent derivation of $\ip$, then
$\Phi:=\exp(\partial)$ is an automorphism of $\ip$, 
where $\exp(\partial)(\theta)=
\sum\limits_{i=0}^\infty \dfrac{1}{i!}\partial^n(\theta)$ 
for all $\theta$ in $\ip$.
\end{enumerate}
\end{lemma}

\begin{proof}
(1) This is easy.

(2) Suppose $\ip$ is generated by $T:=\bigoplus_{s=0}^d \ip(s)$
which is finite dimensional over $\Bbbk$ as $\ip$ is locally finite.
Then the map $\partial\mapsto \partial\mid_{T}$ is injective
by part (1). Note that $\partial\mid_{T}$ is a linear map
of $T$, or $\partial\mid_{T}\in \Hom(T,T)$.
This implies that $\dim_{\Bbbk} (\der(\ip)) \leq (\dim_{\Bbbk} T)^2$.

(3) By definition, we have $\1_2\ucr{i} \1_0=\1_1$ for
$i=1, 2$. Applying the derivation $\partial$, we get
$$\partial(\1_2)\ucr{i}\1_0=
\partial(\1_1)-\1_2 \ucr{i}\partial(\1_0)
=0+c\1_1=c\1_1
$$
by assumption for $i=1, 2$. Set $\mu=\partial(\1_2)-c\1_2$, and we have
$$\mu\ucr{i}\1_0=(\partial(\1_2)-c\1_2)\ucr{i}\1_0
=c\1_1-c\1_1=0.
$$
This implies $\mu\in {^2\iu}(2)$ by definition.

(4) Clearly, $\Phi=\exp(\partial)$ is well-defined since
$\partial$ is locally nilpotent, and $\exp(\partial)$ is an
automorphism of $\Bbbk \S$-module by
$\exp(\partial)\exp(-\partial)=\id_\ip$.

Let $\mu\in \ip(m), \nu\in \ip(n)$. Since the derivation $\partial$
is locally nilpotent, we have
\[\partial^{k_\mu}(\mu)=0,\ {\rm and} \ \partial^{k_\nu}(\nu)=0\]
for some $k_\mu, k_\nu\in \mathbb{N}$. Then
\begin{align*}
\exp(\partial)(\mu\ucr{i}\nu)
&=\sum\limits_{k=0}^{k_\mu+k_\nu-1}\dfrac{1}{k!}
\left(\sum\limits_{r=0}^k {k\choose r} 
\partial^r(\mu)\ucr{i} \partial^{k-r}(\nu)\right)\\
&=\sum\limits_{k=0}^{k_\mu+k_\nu-1}
\left(\sum\limits_{r+s=k}(\dfrac{1}{r!}\partial^r(\mu))
\ucr{i} (\dfrac{1}{s!}\partial^{s}(\nu))\right)\\
&=\left(\sum\limits_{r=0}^{k_\mu-1}
\dfrac{1}{r!}\partial^r(\mu)\right)
\left(\sum\limits_{s=0}^{k_\nu-1}\dfrac{1}{s!}\partial^s(\nu)\right)\\
&=\exp(\partial)(\mu) \ucr{i}\exp(\partial)(\nu)
\end{align*}
It follows that $\exp(\partial)$ is an an automorphism of the
operad $\ip$.
\end{proof}

\begin{lemma}
\label{yylem3.2}
Let $\ip$ be an operad.
\begin{enumerate}
\item[(1)]
$\sf(\ip)\subseteq \ider(\ip)\subseteq \der(\ip)$.
\item[(2)]
If, for every $0\neq c\in \Bbbk$, there is an $n$ such that
$c(n-1) \ip(n)\neq 0$, then $\sf(\ip)=\Bbbk$.
\end{enumerate}
\end{lemma}

\begin{proof}
(1) Every element in $\sf(\ip)$ is of the form
$\sf_{c}$, which is equal to $\ad_{-c\1}\in
\ider(\ip)$. Hence $\sf(\ip)\subseteq \ider(\ip)$.
It is clear that $\ider(\ip)\subseteq \der(\ip)$.

(2) Under the hypothesis, we see that $\sf_{c}\neq 0$
for every $0\neq c\in \Bbbk$. Thus the $\Bbbk$-linear
map $c\mapsto \sf_{c}$ induces an isomorphism from $\Bbbk$ to
$\sf(\ip)$.
\end{proof}

There is a version of Lemma \ref{yylem3.2} for automorphisms
of operads. Details are omitted.

Next we calculate derivations of commonly
used operads. Let $\ip=\As$ and let ${^k \iu}$ be
defined as in \eqref{E3.0.1}. By \cite[p.39]{BYZ},
${^3 \iu}(3)$ of the operad $\As$ is a free $\Bbbk$-module
of dimension two with basis elements
\begin{align}
\label{E3.2.1}\tag{E3.2.1}
\xi_1& =(1) - (12)+ (13) - (123) \in \Bbbk \S_3,\\
\label{E3.2.2}\tag{E3.2.2}
\xi_2& =(23) - (12) - (123) + (132) \in \Bbbk \S_3
\end{align}

\begin{proposition}
\label{yypro3.3} Let $\ip=\As$ or $\ip=\As/{^k \iu}$ for some
$k\geq 4$. In parts {\rm{(2, 3)}}, assume that $\Bbbk$ has no
nontrivial idempotent.
\begin{enumerate}
\item[(1)]
$\der(\ip)=\ider(\ip)\cong \Bbbk$. As a consequence, 
$H^1(\ip)=0$.
\item[(2)]
$\IAut(\ip)=\Bbbk^{\times}$
and $\Aut(\ip)=\Bbbk^{\times} \rtimes {\mathbb Z}_2$.
As a consequence, $\OAut(\ip)={\mathbb Z}_2$.
\item[(3)]
$\End(\ip)=\Aut(\ip)$.
\end{enumerate}
Consequently, $\ip$ is both $\der$-rigid and $\Aut$-rigid.
\end{proposition}

\begin{proof} We prove the assertions only for the operad
$\As$. The same proof works for operads $\As/{^k \iu}$
for all $k\geq 4$, where one of the key point is that
$\xi_2$ is a nonzero basis element in $\ip(3)$. However,
$H^1(\As/{^3 \iu})=\Bbbk$, see Theorem \ref{yythm3.8}(2).

(1) Recall that $\As(n)=\Bbbk \S_n$. The identity element
in $\Bbbk \S_n$ is denoted by $1_n$.

Let $\partial\in \der(\As)$ and let $\partial(1_0)
=-c 1_0$. Replacing $\partial$ by $\partial-\sf_{c}$,
we can assume that $c=0$ or $\partial(1_0)=0$. So
we have $\partial(1_n)=0$ for $n=0,1$.

Note that $\As$ is generated by $\{1_0,1_1,1_2\}$.
We claim that $\partial(1_2)=0$. If this holds,
by Lemma \ref{yylem3.1}(1), $\partial=0$.
Combining with the last paragraph, every
derivation of $\As$ is of the form $\sf_{c}$.
Thus the assertion follows from Lemma \ref{yylem3.2}.

Finally we prove the claim. Suppose $\partial(1_2)=
a 1_2+ b(12)$ for some $a, b\in \Bbbk$.
Since $1_2\underset{2}\circ 1_0=1_1$, after applying $\partial$,
we obtain that $\partial(1_2)\underset{2}\circ 1_0=0$.
This implies that $b=-a$. Applying $\partial$ to 
$1_2\underset{1}{\circ} 1_2=1_2\underset{2}{\circ} 1_2$,
we obtain
\[a[(1_2-(12))\underset{1}{\circ}1_2 +1_2\underset{1}{\circ}(1_2-(12))]
=a[(1_2-(12))\underset{2}{\circ}1_2+1_2\underset{2}{\circ}(1_2-(12))],\]
which is equivalent to
\[a \xi_2=0.\]
Since $\xi_2$ in \eqref{E3.2.2} is a basis element, $a=0$ as required.

(2, 3) Since $\As(1)=\Bbbk$, it is clear that $\IAut(\As)=\ssf(\As) =\Bbbk^{\times}$.
So we prove the second assertion of parts (2) and (3) together.

Let $\Phi$ be an endomorphism of $\As$.
Write $\Phi(1_0)=d 1_0$ and $\Phi(1_2)=a 1_2+b (12)$.
Applying $\Phi$ to the equation $1_2\underset{2}{\circ}1_0=1_1$,
we have $d(a+b)=1$. Hence $d$ is invertible.
Write $d=c^{-1}$ for some $c\in \Bbbk^{\times}$.
Then $\Phi(1_0)=c^{-1} 1_0$ for some $c\in \Bbbk^{\times}$.
Replacing $\Phi$ by $\Phi \circ (\ssf_{c})^{-1}$,
we can assume that $c=1$ or $\Phi(1_0)=1_0$.
So we have $\Phi(1_n)=1_n$ for $n=0,1$, and $a+b=1$.

We claim that $\Phi(1_2)=1_2$ or $(12)\in \S_2$. Applying 
$\Phi$ to $1_2\underset{1}{\circ}1_2=1_2\underset{2}{\circ}1_2$,
we obtain that
\[(a1_2+(1-a)(12))\underset{1}{\circ} (a1_2+(1-a)(12))
=(a1_2+(1-a)(12)) \underset{2}{\circ} (1_1, a1_2+(1-a)(12)\]
which is equivalent to
$$a(1-a) \xi_2=0.$$
Since $\xi_2$ is a basis element, we obtain that $a^2=a$. 
Since $\Bbbk$ does not have any nontrivial idempotent,
we obtain that either $(a, b)=(1,0)$ or $(a, b)=(0,1)$. In both cases,
we obtain an automorphism of $\As$. Therefore $\Phi$ is an
automorphism of $\As$ (part (4)) and there is a short exact
sequence
\[\{1\} \to \IAut(\As)\to \Aut(\As)\to {\mathbb Z}_2\to \{1\}.\]
Now the second assertion of part (3) follows.
\end{proof}

It is interesting to see that the fixed suboperad $\As^{{\mathbb Z}_2}$
[Lemma \ref{yylem1.6}(1)] is related (but not equal) to the operad
${\mathcal J}ord$ that
encodes the category of Jordan algebras. Note that special Jordan algebras
do not form a variety defined by polynomial identities, which means
that there is no operad that encodes just special Jordan algebras.
See discussions in \cite[Section 1.3]{BBM} and references in
\cite{BBM}.

The following Proposition is easy and its proof is similar to
the proof of Proposition \ref{yypro3.3}. So we skip its proof.

\begin{proposition}
\label{yypro3.4} Let $\ip=\Com$.
\begin{enumerate}
\item[(1)]
$\der(\ip)=\ider(\ip)\cong \Bbbk$.
As a consequence, $H^1(\ip)=0$.
\item[(2)]
$\IAut(\ip)=\Aut(\ip)=\End(\ip)=\Bbbk^{\times}$.
As a consequence, $\OAut(\ip)=\{1\}$.
\end{enumerate}
Consequently, $\Com$ is both $\der$-rigid and
$\Aut$-rigid.
\end{proposition}

\begin{proposition}
\label{yypro3.5}
Let $\ip=\Lie$.
\begin{enumerate}
\item[(1)]
$\der(\ip)=\ider(\ip)\cong \Bbbk$.
As a consequence, $H^1(\ip)=0$.
\item[(2)]
$\Aut(\ip)=\IAut(\ip)\cong \Bbbk^{\times}$.
As a consequence, $\OAut(\ip)=\{1\}$.
\item[(3)]
$\End(\Lie)=\Bbbk$.
\end{enumerate}
Consequently, $\Lie$ is both $\der$-rigid and
$\Aut$-rigid.
\end{proposition}

\begin{proof}
(1) Note that $\Lie(0)=0$,
$\Lie(1)=\Bbbk \1$ and $\Lie(2)=
\Bbbk \fb$. Also $\Lie$ is generated by $\fb$
subject to the relations
\[\fb\ast (12)_2=-\fb\]
and
\[\fb\underset{2}{\circ} \fb=\fb\underset{1}{\circ} \fb
+(\fb\underset{2}{\circ} \fb)\ast (12)_3.\]

Let $\partial\in \der(\Lie)$ and
$\partial(\fb)=c\fb$. Since $\Lie$ is
generated by $\fb$, $\partial=\sf_c$
by Lemma \ref{yylem3.1}(1).

(2) Let $\Phi\in \Aut(\Lie)$ and
$\Phi(\fb)=c\fb$ for some $c\in \Bbbk^{\times}$.
Since $\Lie$ is generated by $\fb$, $\Phi=\ssf_c$. 
The assertions follow.

(3) Let $\Phi\in \End(\Lie)$ and
$\Phi(\fb)=c\fb$ for some $c\in \Bbbk$.
Since $\Lie$ is generated by $m$ and subject to
the relations in the proof of part (2), there is a
one-to-one correspondence between $\Phi\in \End(\Lie)$ and
$c\in \Bbbk$. The assertion follows.
\end{proof}

Next we consider the Poisson operad. Recall that $\Pois$ denotes the
operad encoding unital commutative Poisson algebras. It is generated by
$\{\1_0, \1_1=\1, \1_2, \fb\}$ subject to the relations (see the proof of
\cite[Lemma 7.5]{BYZ})
\begin{align}
\label{E3.5.1}\tag{E3.5.1}
\1_2 \underset{1}{\circ} \1_0& =\1_2 \underset{2}{\circ} \1_0=\1_1,\\
\label{E3.5.2}\tag{E3.5.2}
\fb \underset{1}{\circ} \1_0& =\fb \underset{2}{\circ} \1_0=0,\\
\label{E3.5.3}\tag{E3.5.3}
\1_2\ast (12)& =\1_2,\\
\label{E3.5.4}\tag{E3.5.4}
\fb\ast (12) &=-\fb,\\
\label{E3.5.5}\tag{E3.5.5}
\1_2 \underset{1}{\circ} \1_2&= \1_2 \underset{2}{\circ} \1_2,\\
\label{E3.5.6}\tag{E3.5.6}
\fb\underset{1}{\circ} \1_2&=\1_2\underset{2}{\circ} \fb
+(\1_2\underset{2}{\circ} \fb)\ast (12)_3,\\
\label{E3.5.7}\tag{E3.5.7}
\fb\underset{2}{\circ} \fb&=\fb\underset{1}{\circ} \fb
+(\fb\underset{2}{\circ} \fb)\ast (12)_3.
\end{align}

\begin{proposition}
\label{yypro3.6}
Let $\ip=\Pois$ or $\Pois/{^k \iu}$ for some $k\geq 3$.
\begin{enumerate}
\item[(1)]
$\der(\ip)\cong \Bbbk^{\oplus 2}$ and
$\ider(\ip)=\Bbbk$. As a consequence,
$H^1(\ip)=\Bbbk$.
\item[(2)]
$\Aut(\ip)=(\Bbbk^{\times})^{\otimes 2}$.
and $\IAut(\ip)=\Bbbk^{\times}$. As a consequence,
$\OAut(\ip)\cong \Bbbk^{\times}$.
\item[(3)]
$\End(\ip)=\Bbbk \times \Bbbk^{\times}$.
\end{enumerate}
Consequently, $\Pois$ is both $\der$-semirigid and
$\Aut$-semirigid.
\end{proposition}

\begin{proof}
We prove the assertions only for the operad $\Pois$.
A similar proof works for the operads $\Pois/{^k \iu}$ for all $k\geq 3$.
The key point is that $\Pois/{^k \iu}$ is generated by $\{\1_0, \1_2, \fb\}$.
To save space, we omit the proof for $\Pois/{^k \iu}$. For $k=1, 2$,
see Theorem \ref{yythm3.8}(1).

(1) Let $\partial$ be a derivation of $\Pois$.
Since $\partial$ commutes with the $\S$-action,
we have $\partial(\1_2)=c \1_2$ and $\partial(\fb)=d \fb$ 
for some $c, d\in \Bbbk$.
By using \eqref{E3.5.1}, $\partial(\1_0)=-c \1_0$.
One can easily check that $\partial$ preserves all relations 
\eqref{E3.5.1}-\eqref{E3.5.7}
only using the axioms of a derivation.
Therefore there is a one-to-one correspondence 
between $\partial\in \der(\Pois)$ and
the pairs $(c, d)\in \Bbbk^{\oplus 2}$. The first assertion follows.
The second assertion can be proved similarly.

The proof of part (2) is similar to the proof of part (3),
and we only prove part (3).

(3) Let $\Phi$ be an endomorphism of $\Pois$.
Since $\Phi$ commutes with the $\S$-action,
we have $\Phi(\1_0)=d \1_0$, $\Phi(\1_2)=c \1_2$
and $\Phi(\fb)=b \fb$ for some $b, c, d\in \Bbbk$.
By using \eqref{E3.5.1}, we obtain that $dc=1$. So $c$ is invertible, and
$b$ can be any element in $\Bbbk$. One can easily check
that endomorphism $\Phi$ preserves all relations \eqref{E3.5.1}-\eqref{E3.5.7}.
Therefore there is a one-to-one correspondence between
$\Phi\in \End(\Pois)$ and the pairs $(b, c)\in \Bbbk \times \Bbbk^{\times}$.
The assertion follows.
\end{proof}

\begin{remark}
\label{yyrem3.7}
In this remark we assume that $\Bbbk$ is a field. Recall 
that for a 2-unitary operad, we define inductively
$\1_d=\1_{d-1} \underset{1}{\circ}  \1_2$
for all $d\geq 3$. We will consider suboperads of $\Pois$. Let $\partial$
be the derivation of $\Pois$ determined by $\partial(\1_2)=\1_2$
and $\partial(\fb)=0$ (see the proof of Proposition \ref{yypro3.6}).
\begin{enumerate}
\item[(1)]
If $\text{char}\ \Bbbk=0$, then the fixed suboperad 
$\Pois^{\Bbbk \partial}$ [Lemma \ref{yylem1.6}(2)]
under the $\partial$-action is isomorphic to $\Lie$.
\item[(2)]
If $\text{char}\ \Bbbk>0$, then we have
$$\Lie\subsetneq \Pois^{\Bbbk \partial}\subsetneq \Pois.$$
\item[(3)]
For every $d\geq 3$, we define $\Pois_{d}$ be the (non-unitary) suboperad
of $\Pois$ generated by  $\1_d$ and $\fb$. We call $\Pois_{d}$ the
{\it $d$-ary Poisson operad}. It is obvious that
$$\Lie\subsetneq \Pois_{d}\subsetneq \Pois_{2}=\Pois.$$
It would be interesting to understand more about the operads
$\Pois_{d}$ for all $d\geq 3$.
\item[(4)]
If ${\rm{char}}\; \Bbbk=p>0$, then one can check that
$\Pois^{\Bbbk \partial}=\Pois_{p+1}$.
\end{enumerate}
\end{remark}

\begin{proof}[Proof of Theorem \ref{yythm0.3}]
The assertions follow from Propositions \ref{yypro3.3}(1),
\ref{yypro3.4}(1), \ref{yypro3.5}(1), and \ref{yypro3.6}(1).
\end{proof}

\begin{theorem}
\label{yythm3.8} Let $\Bbbk$ be a field.
Let ${^k\iu}$ denote the $k$-th truncation ideal
of $\As$ and $\Pois$ respectively.
\begin{enumerate}
\item[(1)]
$\As/{^1 \iu}\cong \Pois/{^1 \iu}\cong \As/{^2 \iu}
\cong \Pois/{^2 \iu}\cong \Com$.
As a consequence, $H^1(\ip)=0$.
\item[(2)]
If ${\rm{char}}\; \Bbbk\neq 2$, then
$\As/{^3 \iu}\cong \Pois/{^3 \iu}$. As a consequence,
$H^1(\ip)=\Bbbk$.
\item[(3)]
If ${\rm{char}}\; \Bbbk= 2$, then
$\As/{^3 \iu}\not\cong \Pois/{^3 \iu}$.
\item[(4)]
For every $k\geq 4$,
$\As/{^k \iu}\not\cong \Pois/{^k \iu}$.
\end{enumerate}
\end{theorem}

\begin{proof} (1) By \cite[Lemma 3.7]{BYZ},
${^1\iu}={^2\iu}$ for $\ip=\As$. By \cite[Lemma 3.5]{BYZ},
$\As/{^1 \iu}\cong \Com$. So the assertion holds for
$\ip=\As$. A similar argument shows that the assertion
holds for $\ip=\Pois$. The consequence follows from
Proposition \ref{yypro3.4}(1).

(2) Define two elements in $\As/{^3 \iu}$
$$\1'_2:=\frac{1}{2}(\1_2+\1_2\ast (12))
\quad {\text{and}} \quad
\tau' :=\frac{1}{2}(\1_2-\1_2\ast (12)).$$
Then it is easy to check that equations
\eqref{E3.5.1}-\eqref{E3.5.4} hold. For \eqref{E3.5.5},
by a computation, we have
\[\1'_2 \underset{1}{\circ} \1'_2- \1'_2 \underset{2}{\circ} \1'_2
=-\frac{1}{4} \xi_2\]
where $\xi_2$ is given in \eqref{E3.2.2}. Note that
$\xi_2=0$ in $\As/{^3 \iu}$. Therefore \eqref{E3.5.5}
holds for $(\1'_2, \tau')$. Note that it is well known that
\eqref{E3.5.6}-\eqref{E3.5.7} hold for $(\1'_2, \tau')$
even at the level of $\As$. Therefore
\eqref{E3.5.6}-\eqref{E3.5.7} hold for $(\1'_2, \tau')$
in $\As/{^3 \iu}$. Now we define a map
from $\phi: \Pois\to \As/{^3\iu}$ by sending
\[\1_0\mapsto \1_0, \quad \1\mapsto \1, \quad
\1_2\mapsto \1'_2, \quad \tau\mapsto \tau'.\]
Since $(\1_0, \1, \1'_2, \tau')$ in
$\As/{^3 \iu}$ satisfy \eqref{E3.5.1}-\eqref{E3.5.7},
$\phi$ uniquely determines an operadic morphism from
$\Pois$ to $\As/{^3\iu}$. Since $\As/{^3\iu}$ is
generated by $\1_0, \1_2=\1'_2+\tau'$, $\phi$ is a
surjective morphism. Let $K$ be the kernel of $\phi$.
Since $\GKdim \Pois/K= \GKdim \As/{^3\iu}=3$ where the
last equation is \cite[Theorem 0.2(2)]{BYZ}. By
\cite[Theorem 0.2(2)]{BYZ} again, we have that
$K\supseteq {^3 \iu}_{\Pois}$. Since $\As$ and $\Pois$
have the same Hilbert series, the proof of \cite[Theorem 0.1]{BYZ}
shows that $\As/{^k \iu}$ and $\Pois/{^k\iu}$ have the same
Hilbert series for each $k$, in particular, $\As/{^3\iu}$ and
$\Pois/{^3\iu}$ have the same Hilbert series. This forces that
$K={^3\iu}_{\Pois}$, or equivalently, $\phi$ induces an
operadic isomorphism from $\Pois/{^3 \iu}\to \As/{^3 \iu}$.

The consequence follows from
Proposition \ref{yypro3.6}(1).

(3) Let $\ip=\As/{^3\iu}$.
If ${\rm{char}}\; \Bbbk=2$, then there is no nonzero element
$\1'_2:=a \1_2+b \1_2\ast (12)\in \ip(2)$ such that both
\eqref{E3.5.1} and \eqref{E3.5.3} hold for for $\1'_2$.
Therefore $\ip$ can not be isomorphic to $\Pois/{^3 \iu}$.

(4) This follows from the fact that $H^1(\As/{^k \iu})=0$
for all $k\geq 4$ [Proposition \ref{yypro3.3}(1)] and
that $H^1(\Pois/{^k \iu})=\Bbbk$
for all $k\geq 3$ [Proposition \ref{yypro3.6}(1)].
\end{proof}

\section{Proof of Theorems \ref{yythm0.4} and \ref{yythm0.5}}
\label{yysec4}

The goal of this section is to prove Theorems \ref{yythm0.4} and
\ref{yythm0.5} which concern $H^1(\ip)$ of more complicated operads
$\ip$. For simplicity, we assume that $\Bbbk$ is a field. Then we
can use results in other papers (such as \cite[Theorem 0.6]{BYZ})
that assume that $\Bbbk$ is a field.
We need to recall the definition of the first Hochschild
cohomology of an associative algebra. Note that the notion of a
derivation (respectively, an inner derivation) of an associative
algebra can be found in \cite[Section 9.2]{We}. By
\cite[Lemma 9.2.1]{We}, the first Hochschild cohomology of an
associative algebra $A$ is equal to
\begin{equation}
\label{E4.0.1}\tag{E4.0.1}
\HH^1(A)=\der(A)/\ider(A)
\end{equation}
where $\der(A)$ is the set of derivations of $A$ and
$\ider(A)$ is the set of inner derivations of $A$. We will use
\eqref{E4.0.1} as a definition of the first Hochschild
cohomology for a non-unital associative algebra $A$ too.

We recall the construction in \cite[Example 2.3]{BYZ}, 
where the classification of 2-unitary operads of
GKdimension two is given.

\begin{example} \cite[Example 2.3]{BYZ}
\label{yyex4.1}
Let $A=\Bbbk \1_1\oplus \bar A$ be an augmented
algebra with augmentation ideal $\bar A$. Let
$\{\delta_i\mid i\in T\}$ be a $\Bbbk$-basis for $\bar A$ where
$T$ is an index set, and $\{\Omega_{kl}^v\mid k,l,v\in T\}$ the
corresponding structural constant, namely,
\begin{equation}
\notag
\delta_i\delta_j=\sum_{k\in T}\Omega_{ij}^k\delta_k
\end{equation}
for all $i,j\in T$.  We assume that $0$ is not in $T$.

We define a 2-unitary operad ${\mathcal D}_A$ as follows. Set
${\mathcal D}_A(0)=\Bbbk \1_0\cong \Bbbk$,
${\mathcal D}_A(1)=A=\Bbbk \1_1\oplus {\bar A}$,
and
\begin{equation}
\label{E4.1.1}\tag{E4.1.1}
{\mathcal D}_A(n)=\Bbbk \1_n \oplus \bigoplus
\limits_{i\in\n, j\in T}
\Bbbk \delta^n_{(i)j}
\end{equation}
for $n\ge 2$. For consistency of notations, we set $\delta^1_{(1)j}=\delta_j$
for each $j\in T$, and $\delta^n_{(i)0}=\1_n$ for all $i\in \n$.

The action of $\S_n$ on ${\mathcal D}_A(n)$ is given by
$\1_n\ast \sigma = \1_n$ and
$\delta^n_{(i)j}\ast \sigma =  \delta^n_{(\sigma^{-1}(i))j}$
for all $\sigma\in \S_n$ and all $n$.

The partial composition
\[-\ucr{i}-\colon {\mathcal D}_A(m) \otimes {\mathcal D}_A(n)
\to {\mathcal D}_A(m+n-1),\ \  1\leq i\leq m,\]
is defined by
\begin{align}
\label{E4.1.2}\tag{E4.1.2}
\begin{split}
\1_{m} \ucr{i} \1_{n}& = \1_{m+n-1},\\
\1_{m} \ucr{i} \delta_{(k)l}^n&= \delta_{(k+i-1)l}^{m+n-1},\\
\delta^m_{(s)t} \ucr{i} \1_n&=
\begin{cases} \delta^{m+n-1}_{(s)t}, & \quad 1\leq s \leq i-1,\\
              \sum\limits_{h=i}^{i+n-1} \delta_{(h)t}^{m+n-1}, & \quad s=i,\\
              \delta_{(s+n-1) t}^{m+n-1}, & \quad i<s\le m,
\end{cases}\\
\delta^m_{(s)t}\ucr{i} \delta_{(k)l}^n&=
\begin{cases}
              \sum\limits_{v\in T}\Omega^v_{tl}\delta^{m+n-1}_{(i+k-1)v}, & s=i,\\
              0, & s\neq i.
\end{cases}
\end{split}
\end{align}
Note that $-\underset{1}{\circ}-$ in $A$ is just the associative
multiplication of $A$. By the second relation on the above list,
we obtain
$$\delta_{(i)j}^n=\1_n\ucr{i} \delta_j$$
for all $i\in \n, j\in T$. By \cite[Example 2.3]{BYZ},
${\mathcal D}_A$ is a 2-unitary operad.

A $\Bbbk$-linear basis of ${\mathcal D}_A$ is explicitly given in \eqref{E4.1.1}.
When $T$ is a finite set with $d>0$ elements, the generating function of
${\mathcal D}_A$ is
$$G_{\mathcal{D}_A}(t)=\sum_{n=0}^{\infty} (1+dn)t^n=\frac{1}{1-t}+\frac{d t}{(1-t)^{2}}.$$
In this case, ${\mathcal D}_A$ has GKdimension two.
\end{example}

\begin{lemma}
\label{yylem4.2}
Retain the notation as in Example \ref{yyex4.1} and 
let $\ip={\mathcal D}_{A}$.
\begin{enumerate}
\item[(1)]
If $\theta\in \ip(2)$ satisfies
$\theta\ucr{i}\1_0=0$,
for $i=1, 2$, then $\theta=0$.
\item[(2)]
If $\theta\in \ip(2)$ satisfies
$\theta\ucr{i}\1_0=\1_1$,
for $i=1, 2$, then $\theta=\1_2$.
\item[(3)]
$H^0(\ip)=Z(A)\cap \bar{A}$.
\end{enumerate}
\end{lemma}

\begin{proof} (1)
By the construction in Example \ref{yyex4.1},
$\ip(2)$ has a $\Bbbk$-linear basis
$\{\1_2\}\cup \{\delta^2_{(1)i}\}_{i\in T} \cup \{\delta^2_{(2)i}\}_{i\in T}$.
Write $\theta=a \1_2 +\sum_{i\in T} b_i \delta^2_{(1)i}
+\sum_{i\in T} c_i \delta^2_{(2)i}$ for $a, b_i, c_i\in \Bbbk$. Then
\[0=\theta\underset{2}{\circ}\1_0=a \1_1+\sum_{i\in T} b_i \delta_{i},\]
which implies that $a=b_i=0$ for all $i$. By symmetry, $c_i=0$.
Therefore $\theta=0$.

(2) Let $\theta'=\1_2-\theta$. Then
$\theta'\ucr{i} \1_0=0$ for $i=1, 2$.
By part (1), $\theta'=0$. Hence $\theta=\1_2$.

(3) It is easy to see that $H^0(\ip)\subseteq Z(A)\cap \bar{A}$.
Conversely, let $\delta\in Z(A)\cap \bar{A}$. Consider the inner derivation
$\ad_{\delta}$, that is zero when restricted to $\ip(1)$ as $\delta\in Z(A)$.
Since $\delta\in \bar{A}$, $\ad_\delta(\1_0)=0$. By the definition of partial
composition \eqref{E4.1.2}, one sees that $\ad_\delta(\1_2)=0$. Since
$\ip$ is generated by $\{\1_0\}\cup \ip(1) \cup \{\1_2\}$, $\ad_\delta=0$
by Lemma \ref{yylem3.1}(1). Therefore $\delta\in H^0(\ip)$.
\end{proof}

\begin{theorem}
\label{yythm4.3}
Retain the above notation.
\begin{enumerate}
\item[(1)] 
The derivation space
$\der({\mathcal D}_A)$ fits into the following short
exact sequence of Lie algebras
\[0\to \Bbbk \to \der({\mathcal D}_A)\to \der(\bar{A})\to 0,\]
where $\der(\bar{A})$ is the $\Bbbk$-vector space of derivations
of nonunital algebra $\bar{A}$. As a
consequence, $\der({\mathcal D}_A)\cong \Bbbk\oplus \der(\bar{A})$.
\item[(2)] 
The first cohomology group
\[H^1(\ip)\cong \frac{\der(\bar{A})}{\ider(\bar{A})}=\HH^1(\bar{A}),\]
where $\ider(\bar{A})$ is the $\Bbbk$-vector space of inner
derivations of $\bar{A}$.
\item[(3)]
The automorphism group $\Aut(\ip)$ fits into the
following short exact sequence of groups
\[\{1\}\to \Bbbk^{\times} \to \Aut({\mathcal D}_A)\to
\Aut(\bar{A})\to \{1\},\]
where $\Aut(\bar{A})$ is the group of algebra automorphisms of $\bar{A}$.
As a consequence, $\Aut({\mathcal D}_A)\cong \Bbbk^{\times}\rtimes
\Aut(\bar{A})$.
\item[(4)]
The monoid $\End(\ip)$ fits into the
following short exact sequence of semigroups
\[\{1\}\to \Bbbk^{\times} \to \End({\mathcal D}_A)\to \End(\bar{A})\to \{1\},\]
where $\End(\bar{A})$ is the semigroup of algebra endomorphisms of $\bar{A}$.
\end{enumerate}
\end{theorem}

\begin{proof} (1) Let $\ip={\mathcal D}_A$.
For each fixed nonzero element $\delta\in \bar{A}$, let
$\delta^{n}_{(i)}$ denote the element $\1_n \ucr{i} \delta$
for all $n\geq 1$. Then partial composition of $\ip$
is determined by

\begin{align}
\label{E4.3.1}\tag{E4.3.1}
\1_{m} \ucr{i} \1_{n}& = \1_{m+n-1},\\
\label{E4.3.2}\tag{E4.3.2}
\1_{m} \ucr{i} \delta^n_{(k)}&= \delta^{m+n-1}_{(k+i-1)},\\
\label{E4.3.3}\tag{E4.3.3}
\delta^m_{(s)} \ucr{i} \1_n&=
\begin{cases} \delta^{m+n-1}_{(s)}, & \quad 1\leq s \leq i-1,\\
              \sum\limits_{h=i}^{i+n-1} \delta^{m+n-1}_{(h)}, & \quad s=i,\\
              \delta^{m+n-1}_{(s+n-1)}, & \quad i<s\le m,
\end{cases}\\
\label{E4.3.4}\tag{E4.3.4}
\delta^m_{(s)}\ucr{i} {\delta'}^n_{(k)}&=
\begin{cases}
              (\delta\delta')^{m+n-1}_{(i+k-1)}, & s=i,\\
              0 & s\neq i,
\end{cases}
\end{align}
for all $\delta, \delta'\in \bar{A}$ and all $n\geq 0, m\geq 1$.

Let $\phi$ be a derivation of the algebra $\bar{A}$. Define
$\partial_{\phi}\in \Hom(\ip,\ip)$ by
$$\partial_{\phi}(\1_m)=0,\;  \forall \; m\geq 0
\quad {\text{and}}\quad
\partial_{\phi}(\theta^n_{(i)})=(\phi(\theta))^n_{(i)}, \; \forall \;
\theta\in \bar{A}.$$
By using \eqref{E4.3.1}-\eqref{E4.3.4},
one can easily check that $\partial_{\phi}$
is a derivation of the operad $\ip$.

Let $\partial$ be a derivation of the operad $\ip$
and let $\partial(\1_0)=-c \1_0$. Replacing
$\partial$ by $\partial-\sf_{c}$, we may assume that
$\partial(\1_0)=0$. Since $\partial$ is a
derivation of $\ip$, the restriction
$\partial\mid_{\ip(1)}$ is a derivation of
$\ip(1)$ that preserves the augmentation
of $\ip(1)$. (This follows
from the equation that $\delta\circ \1_0=0$ for all
$\delta\in \bar{A}$.) Hence $\partial$ maps $\bar{A}$
to $\bar{A}$ (inside $\ip(1)$), or more precisely,
$\partial$ is a derivation of $\bar{A}$
when restricted to $\bar{A}$.
Let $\phi=\partial\mid_{\bar{A}}$.
Replacing $\partial$ by $\partial-\partial_{\phi}$,
we have that $\partial(\1_0)=0$ and
$\partial(\ip(1))=0$. Under this hypothesis, we claim
that $\partial=0$. Since $\ip$ is generated by
$\1_0$, $\ip(1)$ and $\1_2$, it suffices to
show that $\partial(\1_2)=0$ by Lemma \ref{yylem3.1}(1).
Starting with the equation,
\begin{equation}
\label{E4.3.5}\tag{E4.3.5}
\1_2\underset{1}{\circ}\1_0=\1_1=\1_2\underset{2}{\circ}\1_0,
\end{equation}
after applying $\partial$ to the above, we have
\[\partial(\1_2)\underset{1}{\circ}\1_0=0=\partial(\1_2)\underset{2}{\circ}\1_0.\]
By Lemma \ref{yylem4.2}(1), $\partial(\1_2)=0$. Thus we have proved
the claim and therefore $\partial=0$.

The above paragraph shows that every derivation $\partial$ can
uniquely be written as $\sf_{c}+\partial_{\phi}$ for some
$\phi\in \der(\bar{A})$. Hence the exact sequence follows.

It is clear that $\sf_{c}$ commutes with $\partial_{\phi}$.
Then the exact sequence splits and we have
$\der({\mathcal D}_A)\cong \Bbbk\oplus \der(\bar{A})$.

(2) The assertion follows from the fact that
$\ider(\ip)=\Bbbk\oplus \ider(\bar{A})$ which
can be proved in a way similar to the proof of
part (1).

(3) Let $g\in \Aut(\ip)$ and let $g(\1_0)=c^{-1} \1_0$ for
some $c\in \Bbbk^{\times}$. By replacing $g$ by $g (\aad_{c\1_1})^{-1}$,
we may assume that $g(\1_0)=\1_0$.
We claim that $g(\1_2)=\1_2$. Applying $g$ to Equation \eqref{E4.3.5} and
using the fact $g(\1_n)=\1_n$ for $n=0,1$, we obtain that
\[g(\1_2)\underset{1}{\circ}\1_0=\1_1=g(\1_2)\underset{2}{\circ}\1_0.\]
By Lemma \ref{yylem4.2}(2), $g(\1_2)=\1_2$, so we proved the claim.
In this case, $g$ is an automorphism of the 2-unitary operad $\ip$.
By \cite[Theorem 0.6]{BYZ}, the automorphism group of the 2-unitary
operad $\ip$, denoted by $\Aut_{2u}(\ip)$, is naturally isomorphic
to the automorphism group of the augmented algebra $A$, which is
isomorphic to $\Aut(\bar{A})$. Thus we have a short exact sequence
\[\{1\}\to \Bbbk^{\times} \to \Aut(\ip)\to \Aut_{2u}(\ip)\to \{1\}\]
and an isomorphism
\[\Aut_{2u}(\ip)\cong \Aut(\bar{A}).\]
The assertion follows.

(4) Let $g\in \End(\ip)$ and let $g(\1_0)=d\1_0$ for
some $d\in \Bbbk$. Applying $g$ to Equation \eqref{E4.3.5} and
using the fact $g(\1_1)=\1_1$, we obtain that
$$dg(\1_2)\circ (\1_1,\1_0)=\1_1=dg(\1_2)\circ (\1_0,\1_1).$$
By Lemma \ref{yylem4.2}(2), $d g(\1_2)=\1_2$. Hence $d$ is invertible.
Write $d=c^{-1}$, we have
$g(\1_0)=c^{-1} \1_0$. The rest of the proof is very similar
to the proof of part (2), so is omitted.
\end{proof}

\begin{proof}[Proof of Theorem \ref{yythm0.4}]
By \cite[Theorem 0.6]{BYZ} and its proof, every 2-unitary operad of
GKdimension two is of the form ${\mathcal D}_A$ given in Example
\ref{yyex4.1}. The assertion follows from Theorem \ref{yythm4.3}(1).
\end{proof}

\begin{proof}[Proof of Theorem \ref{yythm0.5}]
By \cite[Theorem 0.4(3)]{BYZ}, if $\ip$ is left and right
artinian and semiprime, then $\ip$ is isomorphic to ${\mathcal D}_A$
for an augmented semisimple artinian algebra $A$ as given in Example
\ref{yyex4.1}. The assertion follows from Theorem \ref{yythm4.3}(1).
\end{proof}

A version of Theorem \ref{yythm4.3}(1) fails when $\GKdim \ip\geq 3$.
We give an example when $\GKdim \ip=3$.

\begin{example}
\label{yyex4.4}
Suppose ${\rm{char}}\; \Bbbk\neq 2$.
Let $A$ be an augmented associative algebra $\Bbbk 1_A\oplus \Bbbk \delta$ with
$\delta^2=2\delta$ and $\bar{A}$ be the 1-dimensional 
nonunital subalgebra $\Bbbk \delta$.
It is easy to verify that $\der(\bar{A})=0$.

To construct an operad of GKdimension 3, we first define the sequence of
vector spaces
$$\ip(n)=\begin{cases} \Bbbk \1_0 & n=0,\\
A \cong \Bbbk \1_1 \oplus \Bbbk \delta^{1}_{(1)} & n=1,\\
\Bbbk \1_n \oplus \bigoplus_{1\leq i\leq n}\Bbbk\delta^{n}_{(i)} \oplus
\bigoplus_{1\leq k<l\leq n} \Bbbk \delta^{n}_{(kl)} & n\geq 2.
\end{cases}
$$
Then the generating series of $\ip$ is
$$G_{\ip}(t)=\frac{1}{1-t}+\frac{t}{(1-t)^2}+\frac{t^2}{(1-t)^3}$$
which implies that $\GKdim \ip=3$.

The action of $\S_n$ on $\ip(n)$ is given by $\1_n\ast \sigma=\1_n$,
$\delta_{(i)}^n\ast \sigma=\delta_{(\sigma^{-1}(i))}^n$
and
$\delta_{(ij)}^n\ast \sigma=\delta_{(\sigma^{-1}(i), \sigma^{-1}(j))}^n$
for all $1\le i<j\le n$. Here we use the convention that
$\delta_{(kl)}^n=\delta_{(lk)}^n$ for all $n$ and $1\le l<k\le n$. 

Define the partial composition of $\ip$ by the following rules,
for all $m\geq 1$, $n\geq 0$ (and $n\geq 1$ when $\delta^n_{(\cdot)}$ appears),
$1\leq i\leq m$, $1\leq s< t\leq m$, $1\leq k<l\leq n$,
\begin{align}
\label{E4.4.1}\tag{E4.4.1}
\1_{m} \ucr{i} \1_{n}& = \1_{m+n-1}, \\
\label{E4.4.2}\tag{E4.4.2}
\1_{m} \ucr{i} \delta^n_{(k)}&= \delta^{m+n-1}_{(k+i-1)}, \; \forall \; 1\leq k\leq n, \\
\label{E4.4.3}\tag{E4.4.3}
\1_{m} \ucr{i} \delta^n_{(kl)}&= \delta^{m+n-1}_{(k+i-1, l+i-1)}, \; \forall \; 1\leq k<l\leq n, \\
\label{E4.4.4}\tag{E4.4.4}
\delta^m_{(s)} \ucr{i} \1_n&=
\begin{cases} \delta^{m+n-1}_{(s)}, & \quad 1\leq s \leq i-1,\\
              \sum\limits_{i\leq h\leq i+n-1} \delta^{m+n-1}_{(h)}
							-\sum\limits_{i\leq k<l\leq i+n-1} \delta^{m+n-1}_{(kl)}, & \quad s=i,\\
              \delta^{m+n-1}_{(s+n-1)}, & \quad i<s\le m,
\end{cases}\\
\label{E4.4.5}\tag{E4.4.5}
\delta^m_{(st)} \ucr{i} \1_n&=
\begin{cases}
\delta^{m+n-1}_{(st)},                                 & \quad 1\leq s <t\leq i-1,\\
\sum\limits_{t\leq h\leq t+n-1} \delta^{m+n-1}_{(sh)}, & \quad t=i,\\
\delta^{m+n-1}_{(s,t+n-1)},                            & \quad s <i <t,\\
\sum\limits_{s\leq h\leq s+n-1} \delta^{m+n-1}_{(h,t+n-1)},
							 & \quad s=i,\\
\delta^{m+n-1}_{(s+n-1, t+n-1)}, & \quad i<s\le m,
\end{cases}\\
\label{E4.4.6}\tag{E4.4.6}
\delta^m_{(s)}\ucr{i} \delta^n_{(k)}&=
\begin{cases}
\delta^{m+n-1}_{(s,k+i-1)}, & s<i,\\
             2\delta^{m+n-1}_{(i+k-1)}
						-\sum\limits_{{i\le h\le i+n-1\atop h\neq i+k-1}}
						\delta^{m+n-1}_{(h, i+k-1)}, & s=i,
						\\
              \delta^{m+n-1}_{(k+i-1,s+n-1)}, & s> i,
\end{cases}\\
\label{E4.4.7}\tag{E4.4.7}
\delta^m_{(st)}\ucr{i} \delta^n_{(k)}&=
\begin{cases}
              2 \delta^{m+n-1}_{(s,t+k-1)}, & i=t,\\
							2 \delta^{m+n-1}_{(s+k-1,t+n-1)}, & i=s,\\
              0, & i\neq s,t,
\end{cases}\\
\label{E4.4.8}\tag{E4.4.8}
\delta^m_{(s)}\ucr{i} \delta^n_{(kl)}&=0,\\
\label{E4.4.9}\tag{E4.4.9}
\delta^m_{(st)}\ucr{i} \delta^n_{(kl)}&=0.
\end{align}

Now one can check that $\ip$ is a 2-unitary operad.
(Note that all checking are straightforward, though tedious.
We apologize for leaving out details).  Also one can
check that $H^0(\ip)=0$.

Note that $\ad_{\delta^1_{(1)}}$ is a
nonzero derivation as
$$\begin{aligned}
\ad_{\delta^1_{(1)}}(\1_2)&= \delta^1_{(1)}\circ \1_2-\1_2\underset{1}{\circ} \delta^1_{(1)}
-\1_2\underset{2}{\circ} \delta^1_{(1)}\\
&=\delta^1_{(1)}\circ \1_2-\delta^2_{(1)}-\delta^2_{(2)}\\
&=-\delta^2_{(12)}
\end{aligned}
$$
where the last equality is \eqref{E4.4.4}. Since $\ad_{\delta^1_{(1)}}$
is not of the form $\sf_{c}$, by Lemma \ref{yylem3.2}(2),
$\ider(\ip)$ has $\Bbbk$-dimension at least two. Since $\der(\overline{\ip(1)})=0$,
$\der(\ip)$ does not fit into an exact sequence
$$0\to \Bbbk\to \der(\ip)\to \der(\overline{\ip(1)})\to 0.$$

In fact, we can calculate $\der(\ip)$ as follows.
Let $\partial$ be an arbitrary derivation of $\ip$ and let
$\partial(\1_0)=-c\1_0$. Replacing $\partial$ by
$\partial-\sf_{c}$, we may assume that $\partial(\1_0)=0$.
Since $\der(\overline{\ip(1)})=0$, we obtain that
$\partial=0$ when restricted to $\ip(1)$.
Suppose $\partial(\1_2)=a \1_2+ b \delta^2_{(1)}+c \delta^2_{(2)}
+d \delta^2_{(12)}$. Applying $\partial$ to the identity
$$\1_2\underset{1}{\circ}\1_0=\1_1=\1_2\underset{2}{\circ}\1_0,$$
we have
$$\partial(\1_2)\underset{1}{\circ}\1_0=0=\partial(\1_2)\underset{2}{\circ}\1_0.$$
The above equation is equivalent to
$$a \1_1+b \delta^1_{(1)}=0=a \1_1+ c\delta^1_{(1)}$$
which implies that $a=b=c=0$. Thus $\partial=(-d) \ad_{\delta^1_{(1)}}$.
Therefore
\begin{equation}
\notag
\der(\ip)=\sf(\ip)\oplus \Bbbk \ad_{\delta^1_{(1)}}\cong \Bbbk^{\oplus 2}.
\end{equation}
Similarly,
$$\ider(\ip)=\Bbbk^{\oplus 2}$$
and
$$H^1(\ip)=0=\HH^1(\overline{\ip(1)}).$$

We can also calculate the automorphism group of $\ip$. If $\Bbbk$ is
${\mathbb R}$ or ${\mathbb C}$, we can use the derivation $\ad_{\delta^1_{(1)}}$
to define an automorphism
\begin{equation}
\notag
\exp(d \ad_{\delta^1_{(1)}}):=\sum_{s=0}^{\infty} \frac{d^n}{n!} \;
(\ad_{\delta^1_{(1)}})^n: \ip\to \ip.
\end{equation}
(Although $\ad_{\delta^1_{(1)}}$ is not nilpotent,
$\exp(d \ad_{\delta^1_{(1)}})$ is still well-defined.) For general
$\Bbbk$, we can use the inner automorphism $\aad_{\la}$ where $\la=\1 +x \delta^1_{(1)}$
for some $x\in \Bbbk$ such that $(1+2x)$ is invertible
(and $\la^{-1}=\1+y \delta^1_{(1)}$ where $y=\frac{-x}{1+2x}$). Now
\[\begin{aligned}
\aad_{\la}(\1_2)&= (\la^{-1} \circ \1_2) \circ (\la, \la)\\
&=\1_2\circ (\1+x \delta^1_{(1)}, \1+x \delta^1_{(1)})\\
&\qquad +y[\1_2\circ (\delta^1_{(1)}, \1)+\1_2\circ (\1,\delta^1_{(1)})
-\1_2\circ (\delta^1_{(1)}, \delta^1_{(1)})]\circ (\1+x \delta^1_{(1)}, \1+x \delta^1_{(1)})\\
&=\1_2+x \delta^2_{(1)}+x \delta^2_{(2)}+x^2 \delta^2_{(12)}\\
&\qquad +y[(1+2x)(\delta^2_{(1)}+\delta^2_{(2)})-(1+2x) \delta^2_{(12)}]\\
&=\1_2+x \delta^2_{(1)}+x \delta^2_{(2)}+x^2 \delta^2_{(12)}\\
&\qquad +(-x)(\delta^2_{(1)}+\delta^2_{(2)})+x \delta^2_{(12)}\\
&=\1_2+(x+x^2) \delta^2_{(12)}.
\end{aligned}
\]

In general, one can check directly that $g_d: \1_0\mapsto \1_0, \1\mapsto \1, \delta^1_{(1)}\mapsto
\delta^1_{(1)}, \1_2\mapsto \1_2+d \delta^2_{(12)}$ extends an automorphism of the operad $\ip$
for every $d\in \Bbbk$. Using this fact, one can show that
\begin{equation}
\notag
\Aut(\ip)=\ssf(\ip) \times \{g_d\mid d\in \Bbbk\}\cong \Bbbk^{\times} \times \Bbbk.
\end{equation}
\end{example}



\section{Preliminaries for $H^2_{\ast}$ and $H^2$ computation}
\label{yysec5}

This and the next two sections are devoted to the computation of $H^2_{\ast}(\ip)$
and $H^2(\ip)$ for various operads $\ip$ (note that $H^2_{\ast}(\ip)$
is a lot more difficult to understand than $H^1(\ip)$). The  main result in
this section is Theorem \ref{yythm5.3}.

The following lemma will be used several times in computation of
$H^2_{\ast}(\ip)$.

\begin{lemma}
\label{yylem5.1}
Let $\Bbbk$ be a field and $\ip$ be an operad over $\Bbbk$
generated by a set of elements, say $X$, and subject to a
set of relations, say $R$. Let $(\ip[\epsilon],
\ucr{i}^{\epsilon}, \ast^{\epsilon})$ be an
infinitesimal deformation of $\ip$. Suppose that the set
$X$ in $\ip\subseteq \ip[\epsilon]$ satisfies all relations
in $R$ with respect to operations
$\ucr{i}^{\epsilon}$ and $\ast^{\epsilon}$. Then
$(\ip[\epsilon], \ucr{i}^{\epsilon},
\ast^{\epsilon})$ is equivalent to the trivial infinitesimal
deformation of $\ip$.
\end{lemma}

\begin{proof} We may consider $\ip[\epsilon]$ as an operad over $\Bbbk$
as well as over $\Bbbk[\epsilon]$.
There is always an operadic morphism from the free operad $\mathcal{F}(X)$ over $\Bbbk$
generated by $X$ to $\ip[\epsilon]$ sending $x\mapsto x$ for all $x\in X$.
By definition, $\ip=\mathcal{F}(X)/\langle R\rangle$
where $\langle R\rangle$ is the ideal of $\mathcal{F}(X)$ generated by $R$.
Since $X$ in $\ip[\epsilon]$ satisfies all relations in $R$,
it induces an operadic morphism from $f: \ip\to \ip[\epsilon]$.
Let $\pi$ be the canonical operadic morphism from $\ip[\epsilon]
\to \ip$ provided by the definition of an infinitesimal deformation.
Then $F:=\pi\circ f:\ip\to \ip$ is an operadic morphism sending
$x\mapsto x$ for all $x\in X$. Since $\ip$ is generated by $X$, we have
that $F$ is the identity of $\ip$. This implies that
for every $\mu\in \ip$, $f(\mu)=\mu+\epsilon \partial(\mu)$
for some $\partial(\mu)\in \ip$.
Let $\{\mu_i\}_{i\in I}$ be a $\Bbbk$-linear basis of $\ip$.
Then $\{\epsilon f(\mu_i)\}_{i\in I}$ is a basis of $\epsilon \ip$
and hence $\{f(\mu_i)\}_{i\in I}\bigcup \{\epsilon f(\mu_i)\}_{i\in I}$
is a $\Bbbk$-linear basis of $\ip[\epsilon]$. Thus we can define
a $\Bbbk$-linear isomorphism $G: \ip \otimes \Bbbk[\epsilon]\to
\ip[\epsilon]$ by
$$G(\mu_i)=f(\mu_i), \qquad {\text{and}} \qquad
G(\epsilon \mu_i)=\epsilon f(\mu_i)=\epsilon \mu_i$$
for all $i\in I$. (Here $\ip \otimes \Bbbk[\epsilon]$ is the
trivial infinitesimal deformation of $\ip$.)
It is clear that $G$ is $\Bbbk[\epsilon]$-linear.
It remains to show that $G$ is an operadic morphism over $\Bbbk$.
We use the $\Bbbk$-linear basis $\{\mu_i\}_{i\in I}\bigcup
\{\epsilon \mu_i\}_{i\in I}$ of the trivial deformation
$\ip \otimes \Bbbk[\epsilon]$. For all $\mu_m,\mu_n$ where $m,n\in I$
and all $i$, using the fact that $f$ is an operadic morphism and
the definition of $G$, we have
\[G(\mu_m \ucr{i} \mu_n)
=f(\mu_m \ucr{i} \mu_n)
=f(\mu_m) \ucr{i}^\ep f(\mu_n)
=G(\mu_m) \ucr{i}^\ep G(\mu_n).
\]
Similarly,
\[G(\epsilon \mu_m \ucr{i} \mu_n)
=\epsilon (\mu_m \ucr{i} \mu_n)
=(\epsilon \mu_m) \ucr{i}^\ep \mu_n
=G(\epsilon\mu_m) \ucr{i}^\ep G(\mu_n).
\]
By symmetry, we have
\[G(\mu_m \ucr{i} \epsilon \mu_n)
=G(\mu_m) \ucr{i}^\ep G(\epsilon \mu_n).
\]
Finally,
\[G(\epsilon\mu_m \ucr{i} \epsilon \mu_n)=0
=G(\epsilon \mu_m) \ucr{i}^\ep G(\epsilon \mu_n).
\]
Therefore $G$ is an operadic morphism as desired.
\end{proof}

We have an immediate consequence.

\begin{proposition}
\label{yypro5.2}
Let $\mathcal{F}(X)$ be a free operad over a field $\Bbbk$ generated
by a set $X$. Then every infinitesimal deformation of $\ip$
is equivalent to the trivial one. As a consequence,
$H^2_{\ast}(\ip)=0$.
\end{proposition}

\begin{proof} Let $(\mathcal{F}(X)[\epsilon],
\ucr{i}^\ep, \ast^{\epsilon})$ be an
infinitesimal deformation of $\mathcal{F}(X)$. Since $\mathcal{F}(X)$ is free
generated by $X$, there is an operadic morphism
from $\mathcal{F}(X)\to \mathcal{F}(X)[\epsilon]$ sending $x\mapsto x$ for all
$x\in X$. The hypotheses in Lemma \ref{yylem5.1} hold
since $R$ is the empty set. By Lemma \ref{yylem5.1},
$\mathcal{F}(X)[\epsilon]$ is equivalent to the trivial one. The
consequence is clear by Theorem \ref{yythm2.10}(3).
\end{proof}

\begin{theorem}
\label{yythm5.3}
Let $\ip$ be a locally finite operad over a field $\Bbbk$.
Suppose that $\ip$ is generated by a finite set and subject
to a finite set of relations. Then $H^2_{\ast}(\ip)$ {\rm{(}}and
hence $H^2(\ip)${\rm{)}} is finite dimensional over $\Bbbk$.
\end{theorem}

\begin{proof} Suppose $\ip$ is generated by a finite set $X$ and
subject to a finite set $R$ of relations. For a large $N$ (that
is bigger than the degrees of elements in $X$), all operadic
operations (either a partial composition or an $\S$-action)
involved in any relations in $R$ are one of the following:
\[-\ucr{i} -\colon \ip(m)\otimes \ip(n)\to \ip(m+n-1)\]
for some $m,n\leq N$, or
\[-\ast -\colon \ip(m)\otimes \Bbbk \S_m\to \ip(m)\]
for some $m\leq N$.

Let $\wv$ be a 2-cocycle of $\ip$. Let $Res(\wv)$ be
a collection of morphisms
\[\wp_i\colon \ip(m)\otimes \ip(n)\to \ip(m+n-1)\]
for all $m,n\leq N$, and
\[\varsigma\colon \ip(m)\otimes \Bbbk \S_m\to \ip(m)\]
for all $m\leq N$. Let
\[V=\bigoplus_{m,n\leq N}\Hom_{\Bbbk}
(\ip(m)\otimes \ip(n), \ip(m+n-1))\oplus
\bigoplus_{m\leq N} \Hom_{\Bbbk}(\ip(m)\otimes \Bbbk\S_m,\ip(m)).\]
Since $\ip$ is locally finite, $V$ is finite dimensional.
We define a $\Bbbk$-linear map
$$Res\colon Z^2_{\ast}(\ip)\to V$$
sending $\wv$ to $Res(\wv)$. Let $K$ be the
kernel of $Res$.
If $\wv$ is in $K$, then
all maps
\[\wp_i\colon \ip(m)\otimes \ip(n)\to \ip(m+n-1)\]
for all $m,n\leq N$ and
\[\varsigma\colon \ip(m)\otimes \Bbbk \S_m\to \ip(m)\]
for all $m\leq N$ are zero.
Let $\ip_{[\wv]}=(\ip[\ep], \ucr{i}^\epsilon, \ast^\epsilon)$
be the infinitesimal deformation
of $\ip$ associated to $\wv=(\wp_i, \varsigma)\in K$.
By \eqref{E2.9.1} and \eqref{E2.9.2},
$\ucr{i}^{\epsilon}$ and $\ast^{\epsilon}$
equal to $\ucr{i}$ and $\ast$ when restricted
to $\ip(m)\otimes \ip(n)$ (for all $m, n\leq n$)
and to $\ip(m)\otimes \Bbbk \S_m$ (for all
$m\leq N$) respectively. This shows that the relations
in $R$ still hold for elements in $X$ considered
in $\ip_{[\wv]}$.
 By Lemma \ref{yylem5.1}, $\wv$ is a 2-coboundary.
 Therefore $K\subseteq B^2_{\ast}(\ip)$. So
\[\dim H^2_{\ast}(\ip)=\dim Z^2_{\ast}(\ip)/B^2_{\ast}(\ip)
\leq \dim Z^2_{\ast}(\ip)/K\leq \dim V<\infty.\]
\end{proof}

\begin{proof}[Proof of Theorem \ref{yythm0.8}]
(1) This is Theorem \ref{yythm1.7}(1).

(2). If $i=0$, $\dim H^0(\ip)\leq \dim \ip(1)<\infty$.
If $i=1$, it follows from Lemma \ref{yylem3.1}(2).
If $i=2$, this is Theorem \ref{yythm5.3}.
\end{proof}

We prove another lemma that is needed in the next section.

\begin{lemma}
\label{yylem5.4}
\begin{enumerate}
\item[(1)]
Suppose $\{a_{m, n}\}_{m\geq 1, n\geq 0}$ is a family of scalars in $\Bbbk$.
Then it satisfies
\begin{equation}
\label{E5.4.1}\tag{E5.4.1}
a_{l+m-1, n}+a_{l,m}=a_{l, m+n-1}+a_{m, n}
\end{equation}
for all $l\geq 1, m\geq 1, n\geq 0$ if and only if there is a sequence of
scalars $\{c_m\}_{m\geq 0}$ such that
\begin{equation}
\label{E5.4.2}\tag{E5.4.2}
a_{m,n}=(c_{m+n-1}-c_{m}-c_n)
\end{equation}
for all $m\geq 1$ and $n\geq 0$.
\item[(2)]
Suppose $\{a_{m, n}\}_{m, n\geq 1}$ is a family of scalars in $\Bbbk$.
Then it satisfies \eqref{E5.4.1}
for all $l, m, n\geq 1$ if and only if there is a sequence of
scalars $\{c_m\}_{m\geq 1}$ such that
\eqref{E5.4.2} holds for all $m, n\geq 1$.
\end{enumerate}
\end{lemma}

\begin{proof} (1) First of all, it is straightforward to check that
\eqref{E5.4.2}
implies that \eqref{E5.4.1}. It remains to show the other implication, namely,
prove \eqref{E5.4.2} by assuming \eqref{E5.4.1}.

Taking $m=1$ in \eqref{E5.4.1}, we get $a_{l, 1}=a_{1, n}$
for all $l\geq 1$ and $n\geq 0$. Set
\begin{equation}
\label{E5.4.3}\tag{E5.4.3}
c_n=\begin{cases}
 -a_{1,0}-a_{2,0}, &n=0,\\
-a_{1,1}=-a_{1,0}, & n=1,\\
0, & n=2,\\
\sum\limits_{i=2}^{n-1} a_{2, i}, &n\ge 3.
\end{cases}
\end{equation}
Clearly, $a_{1, n}=a_{l, 1}=-c_1$ for all $l\geq 1$ and
$n\geq 0$ which implies that
\[a_{1,n}=-c_1=c_{1+n-1}-c_n-c_1\]
and
\[a_{l,1}=-c_1=c_{l+1-1}-c_{l}-c_1.\]
So \eqref{E5.4.2} holds for $(m, n)=(1, n)$ and $(l, 1)$
for all $n\geq 0$ and $l\geq 1$. Again by the definition
of $c_n$ in \eqref{E5.4.3}, one sees that
\[a_{2, n}=-c_n+c_{n+1}=c_{2+n-1}-c_{n}-c_2\]
for all $n \ge 0$. So \eqref{E5.4.2} holds for
$(m, n)=(2, n)$ for all $n\geq 0$. Now we show \eqref{E5.4.2}
by induction on $m$ for $m\geq 3$. By induction hypothesis,
we assume that \eqref{E5.4.2} holds for all smaller $m$
(including for $m=2$). Taking $l=2$ in \eqref{E5.4.1}, for all
$n\geq 0$, we have
\begin{align*}
a_{m+1, n}=& a_{m, n}+a_{2, m+n-1}-a_{2, m}\\
=& (-c_m-c_n+c_{m+n-1})+(-c_{m+n-1}+c_{m+n})-(-c_m+c_{m+1})\\
=& -c_{m+1}-c_n+c_{m+n}.
\end{align*}
Therefore \eqref{E5.4.2} holds for $(m+1,n)$. We finish the
proof by induction.

(2) The proof is similar after we replace \eqref{E5.4.3}
by
\begin{equation}
\label{E5.4.4}\tag{E5.4.4}
c_n=\begin{cases}
-a_{1,1}, & n=1,\\
0, & n=2,\\
\sum\limits_{i=2}^{n-1} a_{2, i}, & n\ge 3.
\end{cases}
\end{equation}
Details are omitted.
\end{proof}

Finally we define a special kind of 2-coboundaries, namely,
superfluous 2-coboundary in a slightly different way
from that of Remark \ref{yyrem1.3}(1).

\begin{definition}
\label{yydef5.5}
Let $\wv$ be a 2-coboundary of $\ip$. We call $\wv$ a
{\it superfluous 2-coboundary} if there is a sequence of
scalars $\{c_n\}_{n\ge 0}$ in $\Bbbk$ such that, for all
$1\leq i\leq m$, $\mu\in \ip(m)$, $\nu\in \ip(n)$
and $\sigma\in \S_m$,
$$\begin{aligned}
\wp_i(\mu,\nu)&=(c_{m+n-1}-c_{n}-c_m) \mu\ui \nu,\\
\varsigma(\mu, \sigma)&=0,
\end{aligned}
$$
or equivalently, there is a superfluous map $\partial:\ip\to\ip$
with $\partial(\mu)=c_{m} \mu$ for all $\mu\in \ip(m)$
such that
$$\begin{aligned}
\wp_i(\mu,\nu)&=\partial(\mu\ui \nu)
-[\partial(\mu)\ui \nu+\mu \ui \partial(\nu)],\\
\varsigma(\mu, \sigma)&=\partial(\mu\ast \sigma)-\partial(\mu)\ast \sigma.
\end{aligned}
$$
\end{definition}

Lemma \ref{yylem5.4} is related to superfluous 2-coboundaries.

\section{Calculation of $H^2_{\ast}$, part 1}
\label{yysec6}
Throughout this section we assume that $\Bbbk$ is a field.
The goal of this section is to work out $H^2_{\ast}$ and
$H^2$ for operads $\As$, $\Com$ and $\Lie$. We divide this section into
several subsections. Recall that $\Bbbk[\epsilon]=\Bbbk[t]/(t^2)$
and that we will use both $t$ and $\epsilon$ for the
variable $\epsilon$ in $\Bbbk[\epsilon]$.

\subsection{$H^2_{\ast}(\Com)$}
\label{yysec6.1}

The main result of this subsection is the following.

\begin{theorem}
\label{yythm6.1}
Suppose that ${\rm{char}}\; \Bbbk\neq 2$. Let $\ip=\Com$.
Then the following hold.
\begin{enumerate}
\item[(1)]
Every infinitesimal deformation of $\ip$ is trivial.
\item[(2)]
$\idf(\ip)=\{0\}$ and $H^2(\ip)=H^2_{\ast}(\ip)=0$.
\item[(3)]
Every 2-cocycle of $\ip$ is a 2-coboundary.
\end{enumerate}
\end{theorem}

The assertion also holds when ${\rm{char}}\; \Bbbk=2$,
see Theorem \ref{yythm6.8}. We first prove two lemmas.

\begin{lemma}
\label{yylem6.2}
Let $\Bbbk$ also denote the trivial $\Bbbk \S_n$-module.
\begin{enumerate}
\item[(1)]
Suppose that ${\rm{char}}\; \Bbbk \neq 2$.
Then $\Ext^1_{\Bbbk \S_n}(\Bbbk,\Bbbk)=0$ for all $n$.
\item[(2)]
Suppose that ${\rm{char}}\; \Bbbk=2$.
Then $\Ext^1_{\Bbbk \S_n}(\Bbbk,\Bbbk)=
\begin{cases}
0, & n\leq 1,\\
\Bbbk, & n\geq 2.
\end{cases}$
\end{enumerate}
\end{lemma}

\begin{proof}
(1) If ${\rm{char}}\; \Bbbk=0$, then $\Bbbk \S_n$ is semisimple.
The assertion follows. Now assume ${\rm{char}}\; \Bbbk\geq 3$.
Then $\Bbbk \S_n$ is semisimple for $n\leq 2$. So the assertion
holds for $n\leq 2$. Assume that we have a short exact sequence
$$0\to \Bbbk\to E\to \Bbbk\to 0$$
as modules over $\Bbbk \S_n$ for $n\geq 3$. Let $(ij)_n$ be the
permutation in $\S_n$ that switches $i$ and $j$, for $1\leq i
<j \leq n$. Since
$$\Ext^1_{\Bbbk \langle (ij)_n \rangle}(\Bbbk, \Bbbk)
=\Ext^1_{\Bbbk \S_2}(\Bbbk,\Bbbk)=0,$$
$((ij)_n-1_n)$ becomes zero when it acts on $E$.
Note that $\S_n$ is generated by $(ij)_n$ for different $i,j$.
Then $((ij)_n-1_n)$ is zero when acting on $E$ for all $(ij)_n\in \S_n$.
This means that $E$ is a direct sum of two trivial module.
Therefore the assertion holds.

(2) The assertion is clear for $n< 2$. It is also
easy to show that $\Ext^1_{\Bbbk \S_2}(\Bbbk,\Bbbk)=\Bbbk$.
When $n\geq 3$, consider a short exact sequence
\begin{equation}
\label{E6.2.1}\tag{E6.2.1}
0\to \Bbbk\to E\to \Bbbk\to 0
\end{equation}
as modules over $\Bbbk \S_n$.
Observe that $\sigma=(123)_n$ is an even permutation.
Then the action of $(1-\sigma)$ on $\Bbbk$ is zero.
This implies that the action of $(1-\sigma)(\sigma-\sigma^2)$ on $E$ is zero.
Since ${\rm{char}}\; \Bbbk=2$,
\[(1-\sigma)(\sigma-\sigma^2)=\sigma^2+\sigma+\sigma^2+1=1-\sigma.\]
Then the action of $(1-\sigma)$ on $E$ is zero. Since the alternating
group ${\mathbb A}_n(\subset \S_n)$ is generated by elements of the form similar
to $\sigma$, we have that the action of ${\mathbb A}_n$ on $E$ is
trivial. So we can consider the short exact sequence \eqref{E6.2.1}
as that of $\Bbbk \S_2$-module (viewing $\S_2$ as $\S_n/{\mathbb A}_n$).
Then the assertion follows from the fact
$\Ext^1_{\Bbbk \S_2}(\Bbbk,\Bbbk)=\Bbbk$.
\end{proof}

\begin{lemma}
\label{yylem6.3}
Let $\ip=\Com$.
\begin{enumerate}
\item[(1)]
Every 2-coboundary is of the form $(\wp_i, 0)$.
\item[(2)]
Suppose $(\widetilde{\wp}_i, \widetilde{\varsigma})$ is
equivalent to $(\wp_i,\varsigma)$. Then
$\widetilde{\varsigma}=\varsigma$.
\item[(3)]
$H^2(\ip)=0$.
\item[(4)]
If ${\rm{char}}\; \Bbbk\neq 2$, then $H^2_{\ast}(\ip)=0$.
\end{enumerate}
\end{lemma}

\begin{proof} (1) Since the $\S_m$-action on $\ip(m)$ is trivial,
the assertion follows from Definition \ref{yydef2.5}(1b).

(2) The assertion follows from part (1).

(3) Let $(\wp_i,\varsigma)$ be an $\S$-2-cocycle. By
definition, we have $\varsigma=0$. Since
$\ip(n)=\Bbbk\cdot \1_n$ for all $n\geq 0$, we obtain that, for all
$\sigma\in \S_m$,
\begin{equation}
\label{E6.3.1}\tag{E6.3.1}
\wp_i(\1_m, \1_n)= a_{m, n, i}\1_{m+n-1} \ {\rm and}\ \vs(\1_m, \sigma)=0
\end{equation}
for some $a_{m, n, i}\in \Bbbk$. By \eqref{E2.4.4} (with $\varsigma=0$),
we have
\[
a_{m, n, i}= a_{m, n, \phi(i)}
\]
for all $m\geq 1, n\ge 0$ and $\phi\in \S_m$. So we can suppose
$a_{m, n, i}=a_{m, n}$ for all $m, n, i$.
By \eqref{E2.4.1}, we have
\[
a_{l+m-1, n}+a_{l,m}= a_{l, m+n-1}+a_{m, n},
\]
which agrees with \eqref{E5.4.1}.
By Lemma \ref{yylem5.4}, there is a sequence of scalars
$\{c_m\}_{m\geq 0}$ such that
\begin{equation}
\label{E6.3.2}\tag{E6.3.2}
a_{m,n}=(c_{m+n-1}-c_{n}-c_{m})
\end{equation}
for all $m\geq 1, n\geq 0$. Combining \eqref{E6.3.1}
with \eqref{E6.3.2}, it follows from Definition \ref{yydef5.5}
that the 2-cocycle $(\wp_i, \varsigma)$ is a 2-coboundary.

(4) By Lemma \ref{yylem6.2}(1),
$\Ext^1_{\Bbbk \S_n}(\Bbbk,\Bbbk)=0$. By Theorem \ref{yythm2.10}(3)
and part (3), we have
$H^2_{\ast}(\ip)=H^2(\ip)=0$.
\end{proof}

\begin{proof}[Proof of Theorem \ref{yythm6.1}]
Note that (1), (2) and (3) are equivalent. So we only need to show
(3). Note that part (3) is equivalent to Lemma \ref{yylem6.3}(4).
We are done.
\end{proof}

\subsection{$H^2_{\ast}(\As)$}
\label{yysec6.2}
The main result of this subsection is the following.

\begin{theorem}
\label{yythm6.4}
Let $\ip=\As$. Then $H^2(\ip)=H^2_{\ast}(\ip)=0$.
\end{theorem}

We need the following lemmas.

\begin{lemma}
\label{yylem6.5}
Let $\Bbbk$ be a field of any characteristic.
\begin{enumerate}
\item[(1)]
For every $n\geq 0$,
$\Ext^1_{\Bbbk \S_n}(\Bbbk \S_n, \Bbbk \S_n)=0$.
\item[(2)]
If $\ip$ is $\As$ or $\Pois$, then every 2-cocycle
of $\ip$ is equivalent to an $\S$-2-cocycle.
\end{enumerate}
\end{lemma}

\begin{proof}
(1) It follows from the fact that $\Bbbk \S_n$ is
a free module over itself.

(2) Since $\ip(n)=\Bbbk \S_n$ for all $n\geq 0$
when $\ip=\As$ or $\Pois$. The assertion follows
from Theorem \ref{yythm2.10}(2).
\end{proof}

The following lemma is a generalized version of
Lemma \ref{yylem5.4}.

\begin{lemma}
\label{yylem6.6}
Let $\ip$ be an operad and
let $\{\theta_n\in \ip(n)\mid n\geq 0\}$ be a sequence elements
in $\ip$ with $\theta_1=\1$ such that $\theta_{m}\underset{1}{\circ} \theta_{n}=\theta_{m+n-1}$
for all $m\geq 1$ and $n\geq 0$. Let $\{K(\theta_{m}, \theta_n)\in
\ip(m+n-1)\mid m\geq 1, n\geq 0\}$ be a set of elements.
\begin{enumerate}
\item[(1)]
Suppose that $\ip(1)=\Bbbk \1$. Then, for all $l,m\geq 1$ and
$n\geq 0$,
\begin{equation}
\label{E6.6.1}\tag{E6.6.1}
K(\theta_{l+m-1},\theta_{n})+K(\theta_{l},\theta_{m})
\ucr{1} \theta_{n}=K(\theta_{l},\theta_{m+n-1})+
\theta_{l}\ucr{1} K(\theta_{m},\theta_{n})
\end{equation}
holds if and only if there is a sequence of elements
$\{\bar\theta_m\in \ip(m)\}_{m\geq 0}$  such that
\begin{equation}
\label{E6.6.2}\tag{E6.6.2}
K(\theta_m, \theta_n)=\bar\theta_{m+n-1}-\bar\theta_{m}\ucr{1} \theta_{n}
-\theta_{m}\underset{1}{\circ} \bar\theta_{n}.
\end{equation}
\item[(2)]
Suppose that $\ip(0)=0$. Then \eqref{E6.6.1} holds
for all $l, m, n\geq 1$ if and only if \eqref{E6.6.2}
holds for all $m, n\geq 1$.
\end{enumerate}
\end{lemma}

\begin{proof}
(1) Using the fact that $\theta_{m}\ucr{1} \theta_{n}=\theta_{m+n-1}$,
one can easily check that \eqref{E6.6.2} implies
\eqref{E6.6.1}. It remains to show the other implication.

Now we assume \eqref{E6.6.1} holds. Since $\ip(1)=\Bbbk \1$, we can write
$K(\theta_1,\theta_1)=a_{11} \1$ and $K(\theta_2,\theta_0)=a_{20} \1$ for some
scalars $a_{11}$ and $a_{20}$ in $\Bbbk$.

Taking $m=1$ in \eqref{E6.6.1}, we get
$K(\theta_l, \theta_1)\underset{1}{\circ} \theta_n=\theta_{l}
\underset{1}{\circ} K(\theta_1,\theta_n)$ for all $l\geq 1,n\geq 0$.
By taking $l=1$ and $n=1$ respectively, we have
\begin{equation}
\label{E6.6.3}\tag{E6.6.3}
K(\theta_1,\theta_n) = \theta_1\underset{1}{\circ}K(\theta_1,\theta_n)
=K(\theta_1, \theta_1)\underset{1}{\circ} \theta_n=a_{11} \theta_n
\end{equation}
for all $n\geq 0$, and
\begin{equation}
\label{E6.6.4}\tag{E6.6.4}
K(\theta_l, \theta_1)=\theta_{l}
\underset{1}{\circ} K(\theta_1,\theta_1)
=a_{11} \theta_l
\end{equation}
for all $l\geq 1$. Set
\begin{equation}
\label{E6.6.5}\tag{E6.6.5}
\bar\theta_n=
\begin{cases}
-K(\theta_1,\theta_0)-K(\theta_2, \theta_0)\underset{1}{\circ}
        \theta_0 =-(a_{11}+a_{20})\theta_0, &n=0,\\
-K(\theta_1,\theta_1) =-a_{11} \theta_1=-a_{11} \1, & n=1,\\
0, & n=2,\\
\sum\limits_{i=1}^{n-2} K(\theta_{n-i},\theta_2) \underset{1}{\circ} \theta_{i}, &n\ge 3.
\end{cases}
\end{equation}
We now prove \eqref{E6.6.2} by induction on $m$. Let $RHS(m)$ and
$LHS(m)$ be the right-hand side and the left-hand side of
\eqref{E6.6.2} for $m$. For each fixed $m$,  we will show that
$RHS(m)=LHS(m)$ for all $n\geq 0$. The initial step is when $m=1$:
$$\begin{aligned}
RHS(1)&=\bar\theta_{1+n-1}-\bar\theta_1\underset{1}{\circ} \theta_n-\theta_1
       \underset{1}{\circ} \bar\theta_n=\bar\theta_{n}-\bar\theta_1\underset{1}{\circ} \theta_n-\bar\theta_n\\
&=-\bar\theta_1\underset{1}{\circ} \theta_n=K(\theta_1,\theta_1)\underset{1}{\circ} \theta_n\\
&=K(\theta_1,\theta_n) \qquad\qquad\qquad \qquad {\text{by}}\;\; \eqref{E6.6.3}\\
&=LHS(1).
\end{aligned}
$$
So \eqref{E6.6.2} holds for $m=1$. When $m=2$, we have
$$\begin{aligned}
RHS(2)&=\bar\theta_{2+n-1}-\bar\theta_2\underset{1}{\circ} \theta_n-\theta_2 \underset{1}{\circ} \bar\theta_n
          =\bar\theta_{n+1}-\theta_2 \underset{1}{\circ} \bar\theta_n\\
&=\begin{cases}
\bar\theta_1 -\theta_2 \underset{1}{\circ} \bar\theta_0, & n=0\\
\bar\theta_2-\theta_2\underset{1}{\circ} \bar\theta_1, & n=1\\
\sum\limits_{i=1}^{n-1} K(\theta_{n+1-i},\theta_2) \underset{1}{\circ} \theta_{i}
-\theta_2 \underset{1}{\circ} [\sum\limits_{i=1}^{n-2} K(\theta_{n-i},\theta_2)
\underset{1}{\circ} \theta_{i}], &n\geq 2.
\end{cases}
\end{aligned}
$$
When $n=0$,
$$RHS(2)=\bar\theta_1 -\theta_2 \underset{1}{\circ} \bar\theta_0=
-a_{11}\1+\theta_2 \underset{1}{\circ} (a_{11}+a_{20})\theta_0=
a_{20} \theta_1=K(\theta_2, \theta_0)=LHS(2).$$
When $n=1$,
$$RHS(2)=\bar\theta_2-\theta_2\underset{1}{\circ} \bar\theta_1 =\theta_2\underset{1}{\circ}K(\theta_1,\theta_1)
=a_{11} \theta_2=K(\theta_2, \theta_1)=LHS(2)$$
by \eqref{E6.6.4}. If $n\geq 2$, using \eqref{E6.6.1} and \eqref{E6.6.5},
we have
$$\begin{aligned}
RHS(2)&=\sum\limits_{i=1}^{n-1} K(\theta_{n+1-i},\theta_2) \underset{1}{\circ} \theta_{i}
-\theta_2 \underset{1}{\circ} [\sum\limits_{i=1}^{n-2} K(\theta_{n-i},\theta_2)
\underset{1}{\circ} \theta_{i}]\\
&=\sum\limits_{i=1}^{n-1} K(\theta_{n+1-i},\theta_2) \underset{1}{\circ} \theta_{i}
-[\sum\limits_{i=1}^{n-2} \theta_2 \underset{1}{\circ} K(\theta_{n-i},\theta_2)]
\underset{1}{\circ} \theta_{i}\\
&=\sum\limits_{i=1}^{n-1} K(\theta_{n+1-i},\theta_2) \underset{1}{\circ} \theta_{i}
-\sum\limits_{i=1}^{n-2} [K(\theta_{n-i+1},\theta_2)+K(\theta_2,\theta_{n-i})
\underset{1}{\circ} \theta_2-K(\theta_2,\theta_{n-i+1})]\underset{1}{\circ} \theta_i\\
&=K(\theta_2,\theta_2)\underset{1}{\circ} \theta_{n-1}-K(\theta_2,\theta_2)\underset{1}{\circ} \theta_{n-1}
+K(\theta_2,\theta_n)\underset{1}{\circ} \theta_{1}\\
&=K(\theta_2,\theta_n)=LHS(2).
\end{aligned}
$$
Up to this point, we have proved \eqref{E6.6.2} for $m=1,2$. Next
we use induction on $m$. Let $l=2$ in \eqref{E6.6.1} and continue
with induction hypothesis, we have, for all $n\geq 0$,
$$\begin{aligned}
LHS(m+1)&=K(\theta_{m+1},\theta_n)=-K(\theta_2,\theta_m)\underset{1}{\circ} \theta_n+K(\theta_2, \theta_{m+n-1})
+\theta_2\underset{1}{\circ} K(\theta_m,\theta_n)\\
&=-[\bar\theta_{2+m-1}-\bar\theta_2\underset{1}{\circ} \theta_m -\theta_2 \underset{1}{\circ} \bar\theta_m]
\underset{1}{\circ} \theta_n
+[\bar\theta_{m+n}-\bar\theta_2 \underset{1}{\circ} \theta_{m+n-1}-\theta_2 \underset{1}{\circ}  \bar\theta_{m+n-1}]\\
&\qquad \qquad\quad +\theta_2 \underset{1}{\circ}
[\bar\theta_{m+n-1}-\theta_{m}\underset{1}{\circ} \bar\theta_{n}-\bar\theta_{m}\underset{1}{\circ} \theta_n]\\
&=\bar\theta_{m+n}-\bar\theta_{2+m-1}\underset{1}{\circ} \theta_n
-\theta_{m+1} \underset{1}{\circ}  \bar\theta_{n}\\
&=RHS(m+1).
\end{aligned}
$$
The assertion follows by induction.

(2) We only sketch the proof since it is similar to the proof of part (1).
Using the fact that $\theta_{m}\underset{1}{\circ} \theta_{n}=\theta_{m+n-1}$,
one can easily check that \eqref{E6.6.2} implies that
\eqref{E6.6.1}. It remains to show the other implication.

In part (2) we have $\ip(0)=0$ (but we do not assume
$\ip(1)=\Bbbk \1$). Note that \eqref{E6.6.3}-\eqref{E6.6.4}
hold for $n,l\geq 1$ without the last equation. Set
\begin{equation}
\notag
\bar\theta_n=
\begin{cases}
-K(\theta_1,\theta_1), & n=1,\\
0, & n=2,\\
\sum\limits_{i=1}^{n-2} K(\theta_{n-i},\theta_2) \underset{1}{\circ} \theta_{i}, &n\ge 3.
\end{cases}
\end{equation}
The rest of the proof is similar without worrying
the case of $n=0$ for proving $RHS(2)=LHS(2)$.
\end{proof}

In $\As$ we identify $\1_n$ with $1_n\in\S_n$ for all $n$.

\begin{lemma}
\label{yylem6.7}
Let $\ip=\As$ and $\wv$ be an
$\S$-2-cocycle of $\ip$ such that $\wp_1(1_m,1_n)=0$ for all
$m\geq 1, n\geq 0$. Then $\wv$ is a 2-coboundary.
\end{lemma}

\begin{proof}
Let $(\ip_{[\wv]}, \ucr{i}^{\epsilon})$
denote the infinitesimal deformation of $\ip$ associated to $\wv$, namely,
$$\mu \ucr{i}^{\epsilon} \nu=\mu\ucr{i} \nu +\wp_i(\mu,\nu) \epsilon$$
for all $\mu,\nu\in \ip$.

By hypothesis, $\wp_1(1_2,1_2)=0$ and $\wp_1(1_2,1_0)=0$, or equivalently,
$1_2\ucr{1}^{\epsilon}1_2=1_3$ and $1_2 \ucr{1}^{\epsilon} 1_0=1_1$.
Let $1_2 \ucr{2}^{\epsilon} 1_0=a1_1$ where $a\in \Bbbk[\epsilon]$.
Then, by operadic axioms,
\[a 1_0= a1_1 \ucr{1}^{\epsilon} 1_0=(1_2 \ucr{2}^{\epsilon} 1_0)
\ucr{1}^{\epsilon} 1_0=(1_2 \ucr{1}^{\epsilon} 1_0)
\ucr{1}^{\epsilon} 1_0=1_1\ucr{1}^{\epsilon} 1_0=1_0,\]
which implies that $a=1$ and $1_2 \ucr{2}^{\epsilon} 1_0=1_1$.
Using the fact that $1_1$ is the identity
$\ip_{[\wv]}$,  one can easily show that
\[(1_2\ucr{2}^{\epsilon} 1_2)\ucr{i}^{\epsilon} 1_0
=1_2\]
for $i=1, 2, 3$.
Write $1_2 \ucr{2}^{\epsilon} 1_2=1_3+ \alpha \epsilon$.
Then the above equations implies that
\[\alpha \epsilon \ucr{i}^{\epsilon} 1_0=(\alpha \ucr{i} 1_0)\epsilon=0\]
for $i=1,2,3$.
This means that $\alpha\in {^3 \iu}(3)$, see \eqref{E3.0.1}.
Recall from \eqref{E3.2.1} and \eqref{E3.2.2} that ${^3 \iu}(3)$
is 2-dimensional with a $\Bbbk$-linear basis $\{\xi_1, \xi_2\}$,
where
\begin{align*}
\xi_1& =(1) - (12)+ (13) - (123),\\
\xi_2& =(23) - (12) - (123) + (132).
\end{align*}
By definition, this implies that $\wp_2(1_2, 1_2)=\alpha=b_1\xi_1+b_2\xi_2$
for some $b_1, b_2\in \Bbbk$. We claim that $b_1=0$.
In fact, taking $\la=\mu=\nu=1_2$ and $i=1, k=2$ in \eqref{E2.4.2}, we get
\begin{equation}\label{E6.7.1}
\tag{E6.7.1}
\wp_{3}(1_3, 1_2)=\wp_2(1_2, 1_2)\ucr{1} 1_2.
\end{equation}
Taking $\la=\mu=\nu=1_2$ and $i=2, j=1$ in \eqref{E2.4.1}, we get
\begin{equation}\label{E6.7.2}
\tag{E6.7.2}
\wp_2(1_3, 1_2)+\wp_2(1_2, 1_2)\ucr{2} 1_2=\wp_2(1_2, 1_3).
\end{equation}
Taking $\la=\mu=\nu=1_2$ and $i=1, j=2$ in \eqref{E2.4.1}, we get
\begin{equation}
\label{E6.7.3} \tag{E6.7.3}
\wp_2(1_3, 1_2)=1_2 \ucr{1} \wp_2(1_2, 1_2).
\end{equation}
For $\la=\mu=\nu=1_2$ and $i=j=2$ in \eqref{E2.4.1}, we have
\begin{equation}
\label{E6.7.4}\tag{E6.7.4}
\wp_3(1_3, 1_2)+\wp_2(1_2, 1_2)\ucr{3} 1_2
=\wp_2(1_2, 1_3)+1_2\ucr{2} \wp_2(1_2, 1_2).
\end{equation}
By \eqref{E6.7.1}-\eqref{E6.7.4}, we have
\[\wp_2(1_2, 1_2)\ucr{1} 1_2+\wp_2(1_2, 1_2)
\ucr{3} 1_2=1_2\ucr{1}\wp_2(1_2, 1_2)+
\wp_2(1_2, 1_2)\ucr{2} 1_2+1_2\ucr{2} \wp_2(1_2, 1_2).\]
Comparing the coefficient of $1_4$ in the above equation, we
obtain
\begin{equation}
\label{E6.7.5}\tag{E6.7.5}
2 b_1=3 b_1,
\end{equation}
which implies $b_1=0$ (and that this is the only equation).
From this point, we have
\[\wp_2(1_2, 1_2)=b_2\xi_2=b_2((12)\ucr{2} 1_2 +1_2\ucr{2}(12)-(12) \ucr{1} 1_2-1_2\ucr{1} (12)).\]
Choosing $\bar{1}_2=1_2+\epsilon b(1_2-(12))$, we have
\begin{align*}
&\bar{1}_2\ucr{2}^{\epsilon} \bar{1}_2-\bar{1}_2\ucr{1}^{\epsilon} \bar{1}_2\\
=& [1_2+\epsilon b_2(1_2-(12))]\ucr{2}^{\epsilon}[1_2+\epsilon b_2(1_2-(12))]
   -[1_2+\epsilon b_2(1_2-(12))]\ucr{1}^{\epsilon}[1_2+\epsilon b_2(1_2-(12))]\\
=& \epsilon b_2(\xi_2-((12)\ucr{2}1_2+1_2\ucr{2}(12)-(12)\ucr{1} 1_2-1_2\ucr{1}(12)))\\
=& 0,
\end{align*}
and therefore $\bar{1}_2\ucr{1}^{\epsilon} 1_0=\bar{1}_2\ucr{2}^{\epsilon} 1_0=1_1$.
It follows that $\bar{1}_2$ is an associative and unital binary operation.
Define $\bar{1}_n=\bar{1}_2\ucr{1} \bar{1}_{n-1}$ for $n\ge 3$, inductively.

Now we can define a $\Bbbk[\epsilon]$-linear operadic morphism $f$ from
$\ip_{[0]}\to \ip_{[\wv]}$,
where $\ip_{[0]}$ is the trivial infinitesimal deformation of $\ip$, by sending
$f: 1_n\in \ip_{[0]}(n)\mapsto \bar{1}_n\in \ip_{[\wv]}(n)$, or equivalently,
\[f: \sigma_0+\epsilon \sigma_1\mapsto \bar{1}_n \ast (\sigma_0+\epsilon \sigma_1)\]
for all $\sigma_0, \sigma_1\in \S_n$. Note that $f$ being an operadic morphism follows
from the fact that $f$ preserves the following only relations in $\ip_{[0]}$:
\[1_2 \ucr{1} 1_2=1_2\ucr{2} 1_2, \quad {\text{and}}\quad
1_2 \ucr{1} 1_0=1_1=1_2 \ucr{2} 1_0.\]
By definition, $\bar{1}_n=1_n+ \epsilon \nu_n$
for some $\nu_n\in \Bbbk\S_n$. Since $\wv$ is an $\S$-2-cocycle, $1_n$ generates
$\ip[\epsilon](n)$ as an $\S_n[\epsilon]$-module. Then $\bar{1}_n$ generates
$\ip[\epsilon](n)$ as an $\S_n[\epsilon]$-module as well. Therefore, $f$ is
surjective, whence injective. The assertion follows.
\end{proof}

\begin{proof}[Proof of Theorem \ref{yythm6.4}]
Since $H^2(\ip)$ is a subspace of $H^2_{\ast}(\ip)$, we only need
to show that $H^2_{\ast}(\ip)=0$. Let $\wv$ be a 2-cocycle of $\ip$
and it remains to show that it is trivial.
By Lemma \ref{yylem6.5}(2), we can assume that
$\wv=(\wp_i,\varsigma)=(\wp_i,0)$, namely, $\wp$ is an
$\S$-2-cocycle.

Consider $\theta_n=1_n$ for all $n\ge 0$ in Lemma \ref{yylem6.6},
and $K(1_m, 1_n)=\wp_1(1_m, 1_n)$.
Then $1_{m} \underset{1}{\circ} 1_{n}=1_{m+n-1}$ for all $m\geq 1, n\geq 0$
and $\{K(1_m, 1_n)\mid m\geq 1, n\geq 0\}$ satisfies \eqref{E6.6.1}
(that follows from \eqref{E2.4.1} by setting $i=j=1$).
By Lemma \ref{yylem6.6},
there is a sequence of elements $\{\bar{1}_n\in \ip(n)\mid n\geq 0\}$
such that \eqref{E6.6.2} holds:
\begin{equation}
\label{E6.7.6}\tag{E6.7.6}
\wp_1(1_m, 1_n)=K(1_m, 1_n)=\bar{1}_{m+n-1}-\bar{1}_{m}\ucr{1} 1_{n}
-1_{m} \ucr{1} \bar{1}_{n}
\end{equation}
for all $m\geq 1$ and $n\geq 0$.

For each $n$, define a map $\partial_n: \ip(n)\to \ip(n)$ by
\[\partial_n (\sigma)=\bar{1}_n\ast \sigma \qquad {\text{for all }}\; \sigma\in \ip(n).\]
Since $\ip(n)$ is a free right $\S_n$-module generated by $1_n$,
the map $\partial_n\colon \ip(n) \to \ip(n)$ is an endomorphism of the right $\Bbbk\S_n$-module for all $n$.
Note that $\partial_n(1_n)=\bar{1}_n$ by definition.
Let $\partial=\{\partial_n\}_{n\geq 0}$ and
let $\wv(\partial)=(\wp(\partial)_i, \vs(\partial))$ be the 2-cocycle associated to $\partial$,
namely,
\begin{align*}
\wp(\partial)_i(\mu, \nu)=& \partial(\mu\ucr{i} \nu)-[\partial(\mu)\ucr{i}\nu-\mu\ucr{i}\partial(\nu)],\\
\vs(\partial)(\mu, \sigma)=& \partial(\mu\ast \sigma)-\partial(\mu)\ast \sigma
\end{align*}
for all $\mu\in \ip(m), \nu\in \ip(n), \sigma\in \S_n$.
We now consider the 2-cocycle $\widetilde{\wv}:=\wv-\wv(\partial)$
which is equivalent to $\wv$.

First of all, since $\partial$ is an  $\S$-module map,
by Definition \ref{yydef2.5}(1), $\widetilde{\wv}$ is still an $\S$-2-cocycle,
namely, $\widetilde{\varsigma}=0$.
By definitions of $\partial$ and $K(1_m, 1_n)$, and \eqref{E6.7.6} and
Definition \ref{yydef2.5}(1), we have
\begin{align*}
\widetilde{\wp}_1(1_n, 1_m)
&=\wp_1(1_n, 1_m)-\wp(\partial)_1(1_n, 1_m)\\
&=K(1_m, 1_n)-[\bar{1}_{m+n-1}-\bar{1}_{m}\ucr{1} 1_{n}
-1_{m}\ucr{1} \bar{1}_{n}]\\
&=0
\end{align*}
for all $m\geq 1, n\geq 0$. By Lemma \ref{yylem6.7},
$\widetilde{\wv}$ is trivial as desired.
\end{proof}

Recall that the definition of ${^k \iu}$ is given in \eqref{E3.0.1}.
The next result is similar to Theorem \ref{yythm6.4}.

\begin{theorem}
\label{yythm6.8}
Let $\ip=\As/{^k \iu}$ for $k\geq 1$. Suppose that $k\neq 4$.
Then $H^2(\ip)=H^2_{\ast}(\ip)=0$.
\end{theorem}

\begin{proof}[Sketch of Proof]
We will repeat some ideas in the proof of Theorem \ref{yythm6.4} and skip many details.

Let both $\ip_{[\wv]}$ and $(\ip[\epsilon], \ucr{i}^\epsilon, \ast^\epsilon)$
denote the infinitesimal deformation of $\ip$ corresponding to a 2-cocycle
$\wv$. It remains to show that $\wv$ is trivial. We can assume that
$\1_2 \ucr{1}^\epsilon \1_0=\1$ after replacing
$\1_0$ by $(1+a\epsilon)\1_0$ if necessary. Write
$$\1_2 \ucr{2}^\epsilon \1_0=(1+c\epsilon) \1.$$
Then
$$\begin{aligned}
(1+c\epsilon) \1_0&= (1+c\epsilon) \1\ucr{1}^\epsilon \1_0\\
&=(\1_2 \ucr{2}^\epsilon \1_0)\ucr{1}^\epsilon \1_0\\
&=(\1_2 \ucr{1}^\epsilon \1_0)\ucr{1}^\epsilon \1_0\\
&=\1 \ucr{1}^\epsilon \1_0=\1_0.
\end{aligned}
$$
Hence $c=0$ and $\1_2 \ucr{2}^\epsilon \1_0=\1$. Using this,
one can show that
$$(\1_2 \ucr{1}^\epsilon \1_2-\1_2 \ucr{2}^\epsilon \1_2)
\ucr{i}^\epsilon \1_0=0$$
for $i=1,2,3$ (some computations are omitted). Therefore
$\1_2 \ucr{1}^\epsilon \1_2-\1_2 \ucr{2}^\epsilon \1_2
\in \epsilon \; {^3 \iu}(3)$. If $k\leq 3$, then
${^3 \iu}(3)=0$ and consequently,
$\1_2 \ucr{1}^\epsilon \1_2=\1_2 \ucr{2}^\epsilon \1_2$.
If $k\geq 5$, similar to the proof of Lemma
\ref{yylem6.7}, one can show that
$\1_2 \ucr{1}^\epsilon \1_2=\1_2 \ucr{2}^\epsilon \1_2$
after replacing $1_2$ by $\bar{1}_2:=1_2-\epsilon b(\tau-1_2)$
for some $b\in \Bbbk$,
or equivalently, replacing $\wv$ by an equivalent 2-cocycle.
(A lot of computations are omitted, see the proof of
Lemma \ref{yylem6.7} and Theorem \ref{yythm6.4}).

Up to this point, we have proved that
$$\1_2 \ucr{1}^\epsilon \1_2=\1_2 \ucr{2}^\epsilon \1_2.$$
Recall from \cite[Figure 2, p.1707]{Sa} that
$\As$ is generated by $\{\1_0,\1_2\}$ subject to the relations
$$\begin{aligned}
\1_2 \underset{1}{\circ} \1_0&=\1,\\
\1_2 \underset{2}{\circ} \1_0&=\1,\\
\1_2 \underset{1}{\circ} \1_2&=\1_2 \underset{2}{\circ} \1_2.
\end{aligned}
$$
Therefore there is an operadic morphism from $f: \As\to \ip_{[\wv]}$
sending $\1_0\mapsto \1_0$ and $\1_2\mapsto \1_2$. By \cite[Theorem 0.1(2)]{BYZ},
$\GKdim \ip=k$. Considering $\ip_{[\wv]}$ as an operad over $\Bbbk$, we
have have $\GKdim \ip_{[\wv]}=k$. This implies that $\GKdim \As/\ker f
\leq k$ where $\ker f$ is the kernel of $f$. By \cite[Theorem 0.1(2)]{BYZ},
$\ker f$ contains ${^k\iu}$ of $\As$. Therefore $f$ induces an operadic
morphism from $\ip$ to $\ip_{[\wv]}$ sending $\1_0\mapsto \1_0$ and $\1_2\mapsto \1_2$,
still denoted by $f$. By Lemma \ref{yylem5.1}, $\ip_{[\wv]}$
is trivial as desired.
\end{proof}

\begin{remark}
\label{yyrem6.9}
When $k=4$, then we can not show that $\1_2 \ucr{1}^\epsilon \1_2
=\1_2 \ucr{2}^\epsilon \1_2$ even after replacing $1_2$ by
$\bar{1}_2:=1_2+\epsilon b(1_2-(12))$, this is because we can
not obtain \eqref{E6.7.5} in this case. It would be interesting to
work out $H^2_{\ast}(\As/{^4 \iu})$ and $H^2(\As/{^4 \iu})$.

When $k=1$, Theorem \ref{yythm6.8} recovers Theorem
\ref{yythm6.1} in any characteristic. These two proof are slightly
different.
\end{remark}

\subsection{$H^2_{\ast}(\Lie)$}
\label{yysec6.3}
In this subsection we deal with the Lie algebra operad $\Lie$.
Throughout this subsection let $\ip$ denote $\Lie$. Write
$\ip(2)=\Bbbk \fb$ where
\begin{equation}
\label{E6.9.1}\tag{E6.9.1}
\fb \ast (12)_2=-\fb .
\end{equation}
Then $\ip(3)=\Bbbk \fb\ucr{1}\fb +\Bbbk \fb\ucr{2}\fb$ and the
Jacobi identity of a Lie algebra is equivalent to the relation
\begin{equation}
\label{E6.9.2}\tag{E6.9.2}
\fb\ucr{2}\fb=\fb\ucr{1} \fb +(\fb\ucr{2} \fb)\ast (12)_3.
\end{equation}
As a consequence, $\ip(3)$ is generated by $\fb\ucr{1}\fb$
as an $\S_3$-module. It is well-known that $\ip$ is generated
by a single element $\{\fb\}$ subject to relations
\eqref{E6.9.1} and \eqref{E6.9.2}. The main result of this
subsection is the following.

\begin{theorem}
\label{yythm6.10}
Let $\Bbbk$ be a field as a general setup in this subsection.
\begin{enumerate}
\item[(1)]
$H^2(\ip)=0$.
\item[(2)]
Suppose ${\rm{char}}\; \Bbbk\neq 2,3$. Then $H^2(\ip)=H^2_{\ast}(\ip)=0$.
\end{enumerate}
\end{theorem}

\begin{proof} If ${\rm{char}}\; \Bbbk=0$, then part (2) follows
from part (1) and Theorem \ref{yythm2.10}(3). For simplicity, we
only prove part (1) and make some comments when ${\rm{char}}\; \Bbbk\geq 5$.

Let $\ip_{[\wv]}=(\ip[\epsilon], \ucr{i}^\epsilon, \ast^\epsilon)$
denote the infinitesimal deformation of $\ip$ associated to
an $\S$-2-cocycle $\wv$.
Since $\wv$ is an $\S$-2-cocycle, we have $\varsigma=0$ and
that $\ast^{\epsilon}=\ast$. (When ${\rm{char}}\; \Bbbk\neq 2,3$,
we can assume that $\ast^{\epsilon}=\ast$ for the right
$\S_2$-action and the right $\S_3$-action on $\ip(2)$ and $\ip(3)$
respectively. This is enough to continue the proof.)

Write
$$\begin{aligned}
\fb\ucr{1}^{\epsilon} \fb &= \fb\ucr{1} \fb +\epsilon \wp_1(\fb, \fb),\\
\fb\ucr{2}^{\epsilon} \fb &= \fb\ucr{2} \fb+\epsilon \wp_2(\fb, \fb),
\end{aligned}
$$
where
\begin{align*}
\wp_1(\fb, \fb)= a_1\fb\ucr{1} \fb+a_2 \fb\ucr{2}\fb, \ \ \
\wp_2(\fb, \fb)= b_1\fb\ucr{1} \fb+b_2 \fb\ucr{2}\fb,
\end{align*}
for some $a_1, a_2, b_1, b_2\in \Bbbk$.

Taking $\mu=\nu=\fb$, $\sigma=(2,1)$ and $i=1$ in \eqref{E2.4.3}, we get
\begin{equation}\label{E6.10.1}\tag{E6.10.1}
-\wp_1(\fb, \fb)=\wp_1(\fb, \fb)\ast (12)_3,
\end{equation}
which is equivalent to
$$-(a_1\fb\ucr{1} \fb+a_2 \fb\ucr{2}\fb)
=-a_1 \fb\ucr{1} \fb+a_2 (\fb\ucr{2}\fb-\fb\ucr{1} \fb).$$
Hence $a_2=0$. Similarly, we have
$-\wp_2(\fb, \fb)=\wp_2(\fb, \fb)\ast (23)_3$ which implies
that $b_1=0$. Taking $\mu=\nu=\fb$, $\phi=(12)_2$ and $i=1$ in \eqref{E2.4.4}, we get
\[
-\wp_1(\fb, \fb)=\wp_2(\fb, \fb)\ast (123)_3
\]
which implies that $a_2=b_1=: a$. In this case,
$$\fb \ucr{i}^\epsilon \fb =(1+a\epsilon)
\fb \ucr{i} \fb$$
for $i=1, 2$. This implies that
\[\fb\ucr{2}^{\epsilon}\fb=\fb\ucr{1}^{\epsilon} \fb +(\fb\ucr{2}^{\epsilon} \fb)\ast (12)_3\]
or \eqref{E6.9.2} holds for $\fb$ in $\ip_{[\wv]}$.
Since \eqref{E6.9.1} trivially holds for any $\S$-2-cocycles,
by Lemma \ref{yylem5.1}, $\ip_{[\wv]}$ is trivial as required.
\end{proof}

We conclude this section with two remarks.

\begin{remark}
\label{yyrem6.11}
By Theorems \ref{yythm6.1}, \ref{yythm6.4} and
\ref{yythm6.10}, $\As, \Com$ and $\Lie$ (when
${\rm{char}}\; \Bbbk\neq 2,3$) are $\idf$-rigid.
So, together with Theorem \ref{yythm8.2},
we give an explicit proof of the claim made by 
Kontsevich-Soibelman in \cite[Section 5.7.7.]{KS}.
We don't know whether or not $\Lie$ is $\idf$-rigid when
${\rm{char}}\; \Bbbk$ is 2 or 3. In view of Example
\ref{yyex2.8}, it will not be surprising if $H^2_{\ast}(\Lie)=\Bbbk$
when ${\rm{char}}\; \Bbbk=2$. But we don't have a strong evidence.
\end{remark}

\begin{remark}
\label{yyrem6.12}
Let $\ip$ be the operad ${\mathcal D}(A)$ given in Example
\ref{yyex4.1}. Some ideas in the proof of Theorem \ref{yythm6.8}
indicates that $H^2_{\ast}(\ip)$ is isomorphic to the second
Hochschild cohomology $\HH^2(\bar{A})$ of the algebra $\bar{A}$.
It is interesting to work out both $H^2_{\ast}(\ip)$ and
$H^2(\ip)$ in this case.
\end{remark}

\section{Calculation of $H^2_{\ast}$, part 2}
\label{yysec7}

In this section we will calculate $H^2_{\ast}(\Pois)$ and
essentially prove Theorem \ref{yythm0.6}(4) . The
proof is very complicated and we need to break it
into several steps. The relations in $\Pois$ are given in
\eqref{E3.5.1}-\eqref{E3.5.7}. For the rest of this
section, let $\ip$ denote either $\Pois$ or $\Pois/{^k \iu}$
for some $k\geq 5$.

\begin{lemma}
\label{yylem7.1}
The following hold.
\begin{enumerate}
\item[(1)]
$\ip(3)$ has a $\Bbbk$-linear basis

\begin{center}
\begin{tabular}{lll}
$\1_3=\1_2 \ucr{1} \1_2$, & $\theta_1=\1_2\ucr{1}\fb$, & $\theta_2=(\1_2\ucr{1}\fb)\ast (23)_3$, \\
$\theta_3=\1_2 \ucr{2} \fb$, &
$\theta_4=\fb\ucr{1} \fb$, &
$\theta_5=\fb\ucr{2}\fb$.
\end{tabular}
\end{center}

\item[(2)]
%
%
Every 2-cocycle of $\ip$ is equivalent to a 2-cocycle with $\varsigma_m=0$ for
$m=0,1,2,3,4$.
\end{enumerate}
\end{lemma}

\begin{proof} (1) Easy computations are omitted here.

(2) Since $\ip(m)=\Bbbk\S_m$ for $m=0,1,2,3,4$, the assertion follows
from the proof of Theorem \ref{yythm2.10}(2).
\end{proof}

When $\ip=\Pois/{^k \iu}$, we need $k\geq 5$ in the proof of
Lemma \ref{yylem7.1}(3). Note that, when $\varsigma_m=0$,
$\ast^{\epsilon}=\ast$ when applied to $\ip_{[\wv]}(m)$.

\begin{lemma}
\label{yylem7.2}
Let $\wv$ be a 2-cocycle of $\ip$ such that $\varsigma_m=0$
for all $m\leq 3$ {\rm{(}}for example, $\wv$ is an $\S$-2-cocycle{\rm{)}}.
Let $\ip_{[\wv]}=(\ip[\epsilon], \ucr{i}^\epsilon, \ast^\epsilon)$
denote the infinitesimal deformation of $\ip$ associated to $\wv$.
\begin{enumerate}
\item[(1)]
Up to a change of basis element $\1_0$ we have
\begin{equation}
\label{E7.2.1}\tag{E7.2.1}
\1_2 \ucr{1}^{\epsilon} \1_0
= \1=\1_2 \ucr{2}^{\epsilon} \1_0, \quad and \quad \wp_1(\1_2, \1_0)=0=\wp_2(\1_2, \1_0)
\end{equation}
which is \eqref{E3.5.1} for $\ucr{2}^{\epsilon}$.
\item[(2)]
We have
\begin{equation}
\label{E7.2.2}\tag{E7.2.2}
\1_2\ast^{\epsilon} (12)=\1_2\ast (12)=\1_2,
\end{equation}
and
\begin{equation}
\label{E7.2.3}\tag{E7.2.3}
\fb\ast^{\epsilon} (12)=\fb\ast (12)=-\fb,
\end{equation}
which are \eqref{E3.5.3} and \eqref{E3.5.4} respectively.
\item[(3)]
\begin{align}
\label{E7.2.4}\tag{E7.2.4}
\1_2 \ucr{1}^{\epsilon} \1_2
&= \1_3 +\ep a_4(\theta_4-2\theta_5),\\
\label{E7.2.5}\tag{E7.2.5}
\1_2 \ucr{2}^{\epsilon} \1_2
&=\1_3 +\ep a_4(\theta_5-2\theta_4)
\end{align}
for some $a_0, a_4\in \Bbbk$.
\item[(4)]
\begin{equation}
\label{E7.2.6}\tag{E7.2.6}
\1_2 \ucr{1}^{\epsilon} \1_2-\1_2 \ucr{2}^{\epsilon} \1_2
=(3a_4\epsilon) (\theta_4-\theta_5),
\end{equation}
or equivalently,
\[\wp_1(\1_2,\1_2)-\wp_2(\1_2,\1_2)=3a_4(\theta_4-\theta_5),\]
where $a_4$ is given by part (3).
\end{enumerate}
\end{lemma}

\begin{proof}
(1) First we work with $\1_0$ and $\1_2$.
Let $\1_2 \ucr{1}^{\epsilon} \1_0=(1+a\epsilon) \1$ for some $a\in \Bbbk$.
After replacing $\1_0$ by $(1-a\epsilon) \1_0$, we can assume that $a=0$.
Let $\1_2 \ucr{2}^{\epsilon} \1_0 =(1+b\epsilon) \1$ for some $b\in \Bbbk$.
Using the associativity,
\[(1+b\epsilon) \1_0=
(\1_2 \ucr{2}^{\epsilon} \1_0)\ucr{1}^{\epsilon} \1_0
=(\1_2 \ucr{1}^{\epsilon} \1_0)\ucr{1}^{\epsilon} \1_0
=\1\ucr{1}^{\epsilon} \1_0=\1_0,\]
which implies that $b=0$. Therefore the assertions hold.

(2) Since $\varsigma_{2}=0$, the assertion follows.

(3) Using the basis given
in Lemma \ref{yylem7.1}(1), write
\begin{align*}
\1_2 \ucr{1}^{\epsilon} \1_2
&=\1_3+\ep(a_0 \1_3 +a_1\theta_1+a_2\theta_2+a_3\theta_3+a_4\theta_4+a_5\theta_5)
\end{align*}
where $a_i\in \Bbbk, i=0, 1, \cdots, 5$.
By $\1_2\ucr{1}^\ep \1_2=\1_2\ucr{1}^\ep(\1_2\ast (12)_2)=(\1_2\ucr{1}^\ep \1_2)\ast (12)_3$, we get
\begin{align*}
a_0\1_3+a_1\theta_1+a_2\theta_2+ &a_3\theta_3+a_4\theta_4+a_5\theta_5\\
=& a_0\1_3-a_1\theta_1+a_2\theta_3+a_3\theta_2-a_4\theta_4+a_5(\theta_5-\theta_4)
\end{align*}
and $a_1=0, a_2=a_3, a_5=-2a_4$.
Replacing $(1+ a_0\ep)\1_3$ by $\1_3$,  we can assume that $a_0=0$
(this does not effect the proof).
Therefore,
\[\1_2\ucr{1}^\ep \1_2=\1_3+\ep[a_2(\theta_2+\theta_3)+a_4(\theta_4-2\theta_5)],\]
and
\begin{align*}
\1_2\ucr{2}^\ep \1_2=&(\1_2\ast (12))\ucr{1}^\ep \1_2=(\1_2\ucr{1}^\ep \1_2)\ast (132)\\
=&\1_3+\ep[-a_2(\theta_1+\theta_2)+a_4(\theta_5-2\theta_4)].
\end{align*}
Using \eqref{E7.2.1}, we have
\[(\1_2\ucr{1}^{\epsilon}\1_2-\1_2\ucr{2}^{\epsilon}\1_2)
\ucr{i}^{\epsilon} \1_0=0\]
for all $i=1, 2, 3$, which implies that $a_2=0$.
 The assertion follows.

(4) This follows from part (3).
\end{proof}

\begin{lemma}
\label{yylem7.3}
Assume the hypothesis of Lemma \ref{yylem7.2}.
\begin{enumerate}
\item[(1)]
\begin{equation}
\label{E7.3.1}\tag{E7.3.1}
\fb\ucr{1}^{\epsilon}\1_0=\fb\ucr{2}^{\epsilon}\1_0=0
\qquad {\text{and}}\qquad
\wp_1(\fb,\1_0)=\wp_2(\fb,\1_0)=0
\end{equation}
which is \eqref{E3.5.2} for $\ucr{i}^{\epsilon}$.
\item[(2)]
\begin{align}
\label{E7.3.2}\tag{E7.3.2}
\fb \ucr{1}^{\epsilon} \fb
&=\fb\ucr{1} \fb+\ep b_2(2\theta_1+\theta_2-\theta_3),\\
\label{E7.3.3}\tag{E7.3.3}
\fb \ucr{2}^{\epsilon} \fb
&=\fb\ucr{2} \fb+\ep b_2(\theta_1-\theta_2-2\theta_3)
\end{align}
for some $b_2\in \Bbbk$.
\item[(3)]
\begin{equation}
\label{E7.3.4}\tag{E7.3.4}
\fb \ucr{2}^{\epsilon} \fb-\fb \ucr{1}^{\epsilon} \fb
-(\fb \ucr{2}^{\epsilon} \fb)\ast (12)_3=0
\end{equation}
which is \eqref{E3.5.7} for $\ucr{i}^{\epsilon}$.
\item[(4)]
\begin{equation}
\label{E7.3.5}\tag{E7.3.5}
(\fb\ucr{1} \fb)\ucr{j}^{\epsilon} \1_0=
\begin{cases}
-b_2\epsilon \fb, & j=1,\\
b_2\epsilon \fb,  & j=2,\\
2b_2\epsilon \fb, & j=3,
\end{cases}
\quad\qquad
(\fb\ucr{2} \fb)\ucr{j}^{\epsilon} \1_0=
\begin{cases}
-2b_2\epsilon \fb, & j=1,\\
-b_2\epsilon \fb,  & j=2,\\
b_2\epsilon \fb,   & j=3,
\end{cases}
\end{equation}
where $b_2$ is given in part (2).
\end{enumerate}
\end{lemma}

\begin{proof} Since $\varsigma_m=0$ for all $m\leq 3$,
we have $\ast^{\epsilon}=\ast$ when applied to $\ip_{[\wv]}(m)$
for $m\leq 3$.

(1) Suppose $\fb \ucr{1}^{\epsilon} \1_0=\fb\ucr{1}\1_0+\ep x \1=\ep x\1$ for some $x\in \Bbbk$.
Then
\begin{align*}
\ep x \1_0&=(\ep x\1)\ucr{1}^{\epsilon}\1_0=(\fb\ucr{1}^{\epsilon} \1_0)\ucr{1}^{\epsilon} \1_0\\
&=-((\fb\ast (12))\ucr{1}^{\epsilon} \1_0)\ucr{1}^{\epsilon} \1_0
 =-(\fb\ucr{2}^{\epsilon} \1_0)\ucr{1}^{\epsilon} \1_0\\
&=-(\fb\ucr{1}^{\epsilon} \1_0)\ucr{1}^{\epsilon} \1_0=-\ep x\1_0,
\end{align*}

which implies that $x=0$ and $\fb\ucr{1}^{\epsilon}\1_0=0$. Similarly,
$\fb\ucr{2}^{\epsilon}\1_0=0$. The assertion follows.

(2, 3) Using the basis given in Lemma \ref{yylem7.1}(1),
write
\begin{align*}
\fb \ucr{1}^{\epsilon} \fb
&= \fb \ucr{1} \fb+\ep(b_0\1_3+b_1\theta_1+b_2\theta_2+b_3\theta_3+b_4\theta_4+b_5\theta_5),
\end{align*}
where $b_i\in \Bbbk, i=0, 1, \cdots, 5$.
By $-\fb\ucr{1}^\ep \fb=\fb\ucr{1}^\ep (\fb\ast (12)_2)=(\fb\ucr{1}^\ep\fb)\ast (12)_3$, we get
\begin{align*}
& -(b_0\1_3+b_1\theta_1+b_2\theta_2+b_3\theta_3+b_4\theta_4+b_5\theta_5)\\
=& b_0\1_3-b_1\theta_1+b_2\theta_3+b_3\theta_2-b_4\theta_4+b_5(\theta_5-\theta_4),
\end{align*}
and $b_0=0, b_2=-b_3, b_5=0$.
Replacing $(1+ b_4\ep)\theta_4$ by $\theta_4$,  we can assume that $b_4=0$
(this does not effect the proof). Therefore,
\begin{align}\label{E7.3.6}
\tag{E7.3.6}
\fb \ucr{1}^\ep \fb= \fb\ucr{1}\fb+\ep[b_1\theta_1+b_2(\theta_2-\theta_3)],
\end{align}
and
\begin{align}\label{E7.3.7}
\fb\ucr{2}^\ep \fb=&-(\fb\ast (12)_2)\ucr{2}^\ep \fb=-(\fb\ucr{1}^\ep \fb)\ast (132)_3 \notag\\
=& \fb\ucr{2}\fb+\ep[-b_1\theta_3+b_2(\theta_1-\theta_2)]. \tag{E7.3.7}
\end{align}
Using \eqref{E7.3.1}, one sees that
\[(\fb \ucr{2}^{\epsilon} \fb-\fb \ucr{1}^{\epsilon} \fb
-(\fb \ucr{2}^{\epsilon} \fb)\ast(12)_3) \ucr{i}^{\epsilon}
\1_0=0\]
for $i=1,2,3$.  An easy computation following by
\eqref{E7.3.6}-\eqref{E7.3.7} shows that
$b_1=2b_2$. Then we have
\begin{align*}
\fb \ucr{1}^\ep \fb=& \fb\ucr{1}\fb+\ep b_2(2\theta_1+\theta_2-\theta_3),\\
\fb\ucr{2}^\ep \fb=& \fb\ucr{2}\fb+\ep b_2(\theta_1-\theta_2-2\theta_3),
\end{align*}
and
\begin{align*}
\fb \ucr{2}^{\epsilon} \fb-\fb \ucr{1}^{\epsilon} \fb
-(\fb \ucr{2}^{\epsilon} \fb)\ast(12)_3=\fb \ucr{2} \fb-\fb \ucr{1} \fb
-(\fb \ucr{2} \fb)\ast (12)_3=0.
\end{align*}

(4)
It follows from part (1) that
$(\fb\ucr{i}^{\epsilon} \fb)\ucr{j}^{\epsilon} \1_0=0$
for $i=1, 2$ and $j=1, 2, 3$. Now the assertions follow from
\eqref{E7.3.2}-\eqref{E7.3.3}.
\end{proof}

\begin{lemma}
\label{yylem7.4}
Assume the hypothesis of Lemma \ref{yylem7.2}.
\begin{enumerate}
\item[(1)]
\begin{align}
\label{E7.4.1}\tag{E7.4.1}
1_2\ucr{1}^{\epsilon} \fb
&= \1_2\ucr{1}\fb+\ep[c_2(\theta_2-\theta_3)+c_4\theta_4],\\
\label{E7.4.2}\tag{E7.4.2}
1_2 \ucr{2}^{\epsilon} \fb
&= \1_2\ucr{2}\fb+\ep[c_2(\theta_2-\theta_1)-c_4\theta_5]
\end{align}
for some $c_2, c_4\in \Bbbk$.
\item[(2)]
\begin{align}
\label{E7.4.3}\tag{E7.4.3}
\fb\ucr{1}^{\epsilon} \1_2
=& \fb\ucr{1} \1_2+\ep[c_2(\theta_2+\theta_3)+d_4(\theta_4-2\theta_5)]\\
\label{E7.4.4}\tag{E7.4.4}
\fb\ucr{2}^{\epsilon} \1_2
=& \fb\ucr{2} \1_2+\ep[c_2(\theta_2+\theta_1)-d_4(\theta_5-2\theta_4)].
\end{align}
for some $d_4\in \Bbbk$, where $c_2$ is given in part $(1)$.
\item[(3)]
\begin{align}
\fb\ucr{1}^{\epsilon} \1_2-\1_2\ucr{2}^{\epsilon}\fb -(\1_2\ucr{2}^{\epsilon}\fb)\ast (12)_3
=&(d_4-c_4)\epsilon (\theta_4 -2 \theta_5), \label{E7.4.5}\tag{E7.4.5}\\
\fb\ucr{2}^{\epsilon} \1_2-\1_2\ucr{1}^{\epsilon}\fb
-(\1_2\ucr{1}^{\epsilon}\fb)\ast (23)_3=&(d_4-c_4)\epsilon (2\theta_4
-\theta_5), \label{E7.4.6}\tag{E7.4.6}
\end{align}
where $c_4, d_4$ are given in part (1) and (2), respectively.
\end{enumerate}
\end{lemma}

\begin{proof}
(1) Using the basis given in Lemma \ref{yylem7.1}(1),
write
\begin{align*}
1_2\ucr{1}^{\epsilon} \fb
&=1_2\ucr{1}\fb+\ep(c_0\1_3+c_1\theta_1+c_2\theta_2+c_3\theta_3+c_4\theta_4+c_5\theta_5),
\end{align*}
where $c_i\in \Bbbk, i=0, 1, \cdots, 5$.
By $-\1_2\ucr{1}^{\epsilon} \fb=\1_2\ucr{1}^{\epsilon} 
(\fb\ast (12)_2)=(\1_2\ucr{1}^{\epsilon} \fb)\ast (12)_3$, we get
\begin{align*}
-(c_0\1_3+c_1\theta_1+c_2\theta_2+ & c_3\theta_3+c_4\theta_4+c_5\theta_5)\\
=& c_0\1_3-c_1\theta_1+c_2\theta_3+c_3\theta_2-c_4\theta_4+c_5(\theta_5-\theta_4),
\end{align*}
and $c_0=0, c_3=-c_2, c_5=0$. Replacing $(1+ c_1\ep)\theta_1$ by $\theta_1$,  we can assume that $c_1=0$
(this does not effect the proof). Therefore,
\begin{align*}
\1_2\ucr{1}^\ep \fb=\1_2\ucr{1}\fb+\ep[c_2(\theta_2-\theta_3)+c_4\theta_4],
\end{align*}
and
\begin{align*}
\1_2\ucr{2}^\ep \fb=& (\1_2\ast (12)_2)\ucr{2}^\ep \fb=(\1_2\ucr{1}^\ep \fb)\ast (132)\\
=& \1_2\ucr{2}\fb+\ep(c_2(\theta_2-\theta_1)-c_4\theta_5).
\end{align*}

(2, 3) Write
\begin{align*}
\fb\ucr{1}^{\epsilon} \1_2
&= \fb\ucr{1} \1_2+\ep(d_0\1_3+d_1\theta_1+d_2\theta_2+d_3\theta_3+d_4\theta_4+d_5\theta_5),
\end{align*}
where $d_i\in \Bbbk$, $i=0, 1, \cdots, 5$.
By  $\fb\ucr{1}^{\epsilon} \1_2=\fb\ucr{1}^{\epsilon} 
(\1_2\ast (12)_2)=(\fb\ucr{1}^{\epsilon} \1_2)\ast (12)_3$, we get
\begin{align*}
d_0\1_3+d_1\theta_1+d_2\theta_2+ & d_3\theta_3+d_4\theta_4+d_5\theta_5\\
=& d_0\1_3-d_1\theta_1+d_2\theta_3+d_3\theta_2-d_4\theta_4+d_5(\theta_5-\theta_4),
\end{align*}
and $d_1=0, d_2=d_3, d_5=-2d_4$. Therefore, we have
\begin{align*}
\fb \ucr{1}^\ep \1_2=& \fb\ucr{1}\1_2+
 \ep[d_0\1_3+d_2(\theta_2+\theta_3)+d_4(\theta_4-2\theta_5)], \\
\fb \ucr{2}^\ep \1_2=& -(\fb\ast(12)_2) \ucr{2}^\ep \1_2=-(\fb\ucr{1}^\ep\1_2)\ast (132)\\
=& \fb \ucr{2}\1_2+\ep[-d_0\1_3+d_2(\theta_1+\theta_2)-d_4(\theta_5-2\theta_4)].
\end{align*}
Set $\mu=\fb\ucr{1}^\ep \1_2-\1_2\ucr{2}^\ep \fb-(\1_2\ucr{2}^\ep \fb)\ast (12)_3$. 
Then by \eqref{E7.2.1} and \eqref{E7.3.1}, we have $\mu\ucr{i}^\e \1_0=0$ 
for $i=1, 2, 3$, which implies that
$d_0=0$ and $d_2=c_2$. Therefore, we get
\begin{align*}
\fb\ucr{1}^{\epsilon} \1_2
=& \fb\ucr{1} \1_2+\ep[c_2(\theta_2+\theta_3)+d_4(\theta_4-2\theta_5)],\\
\fb\ucr{2}^{\epsilon} \1_2=&-(\fb\ast (12)_2)\ucr{2}^\ep \1_2=-(\fb \ucr{1}^\ep \1_2)\ast (132)_3\\
=& \fb\ucr{2} \1_2+\ep[c_2(\theta_2+\theta_1)-d_4(\theta_5-2\theta_4)].
\end{align*}
It follows that
\begin{align*}
\fb\ucr{1}^{\epsilon} \1_2&=\1_2\ucr{2}^{\epsilon}\fb +(\1_2\ucr{2}^{\epsilon}\fb)\ast (12)_3+(d_4-c_4)\ep(\theta_4-2\theta_5),\\
\fb\ucr{2}^{\epsilon} \1_2&=-(\fb\ast (12)_2)\ucr{2}^\ep \1_2=-(\fb\ucr{1}^{\epsilon} \1_2) \ast (132)_3\\
&=\1_2 \ucr{1}^{\epsilon}\fb+ (\1_2\ucr{1}\fb)\ast(23)_3+
(d_4-c_4)\ep(2\theta_4-\theta_5).
\end{align*}
The assertions follow.
\end{proof}

\begin{lemma}
\label{yylem7.5}
Assume the hypothesis of Lemma \ref{yylem7.2}.
In addition we suppose that $\varsigma_4=0$
and $\1_2\ucr{1}^{\epsilon} \1_2=\1_2\ucr{2}^{\epsilon}\1_2$. Then
\begin{align*}
\fb\ucr{1}^{\epsilon} \1_2=\1_2\ucr{2}^{\epsilon}\fb +(\1_2\ucr{2}^{\epsilon}\fb)\ast (12)_3.
\end{align*}
which is \eqref{E3.5.6} for $\ucr{i}^\ep$.
\end{lemma}

\begin{proof} By \eqref{E7.4.5}, we can assume that
\[\fb\ucr{1}^{\epsilon} \1_2=\1_2\ucr{2}^{\epsilon}\fb +(\1_2\ucr{2}^{\epsilon}\fb)\ast (12)_3
+a\epsilon (\fb\ucr{1}\fb -2 \fb\ucr{2}\fb)\]
for some $a\in \Bbbk$.
By our assumption, we have
\begin{align*}
\fb\ucr{1}^\e (\1_2 \ucr{2}^\e \1_2)=&(\fb\ucr{1}^\e \1_2) \ucr{2}^\e \1_2\\
= & [\1_2\ucr{2}^\e \fb+(\1_2\ucr{2}^\e \fb)\ast (12)_3+a\e ( \fb \ucr{1} \fb -2 \fb \ucr{2} \fb )]\ucr{2}^\e \1_2\\
= & (\1_2\ucr{2}^\e \fb)\ucr{2}^\e \1_2 +[(\1_2\ucr{2}^\e \fb)\ast (12)_3]\ucr{2}^\e \1_2 +a\e ( \fb \ucr{1} \fb -2 \fb \ucr{2} \fb )\ucr{2}\1_2\\
= & \1_2 \ucr{2}^\e (\fb \ucr{1}^\e \1_2)+[(\1_2\ucr{2}^\e \fb)\ucr{1}^\e \1_2]\ast (132)_4+a\e ( \fb \ucr{1} \fb -2 \fb \ucr{2} \fb )\ucr{2}\1_2\\
= & \1_2\ucr{2}^\e [\1_2\ucr{2}^\e \fb+(\1_2\ucr{2}^\e \fb)\ast (12)_3+a\e ( \fb \ucr{1} \fb -2 \fb \ucr{2} \fb )]\\
& +((\1_2\ucr{1}^\e \1_2)\ucr{3}^\e \fb)\ast (132)_4+a\e ( \fb \ucr{1} \fb -2 \fb \ucr{2} \fb )\ucr{2} \1_2\\
= & (\1_2\ucr{2}^\e \1_2)\ucr{3}^\e \fb+((\1_2\ucr{2}^\e \1_2)\ucr{3}^\e \fb)\ast (23)_4+((\1_2\ucr{1}^\e \1_2)\ucr{3}^\e \fb)\ast (132)_4\\
& +a\e [\1_2\ucr{2} (\fb \ucr{1} \fb-2\fb\ucr{2} \fb)+(\fb \ucr{1} \fb-2\fb\ucr{2} \fb)\ucr{2} \1_2].
\end{align*}
Similarly,
\begin{align*}
\fb\ucr{1}^\e (\1_2 \ucr{1}^\e \1_2)=&(\fb\ucr{1}^\e \1_2) \ucr{1}^\e \1_2\\
= &  (\1_2\ucr{1}^\e \1_2)\ucr{3}^\e \fb +((\1_2\ucr{2}^\e \1_2)\ucr{3}^\e \fb)\ast (123)_4+((\1_2\ucr{2}^\e \1_2)\ucr{3}^\e \fb)\ast (13)_4\\
&+a\e [(\1_2\ucr{2} (\fb \ucr{1} \fb-2\fb\ucr{2} \fb))\ast (123)_4 + (\fb \ucr{1} \fb-2\fb\ucr{2} \fb)\ucr{1}\1_2]\\
\end{align*}
By $\1_2\ucr{1}^\e \1_2=\1_2\ucr{2}^\e \1_2$ and $\1_2\ast (12)_2=\1_2$, we have
\begin{align*}
(\1_2\ucr{2}^\e \1_2)\ucr{3}^\e \fb
=& (\1_2\ucr{1}^\e (\1_2\ast (12)_2))\ucr{3}^\e \fb\\
=& ((\1_2 \ucr{1}^\e \1_2)\ast (12)_3)\ucr{3}^\e \fb\\
=& ((\1_2 \ucr{1}^\e \1_2)\ucr{3}^\e \fb)\ast (12)_4,
\end{align*}
and
\begin{align*}
((\1_2\ucr{2}^\e \1_2)\ucr{3}^\e \fb)\ast (23)_4=
  & ((\1_2 \ucr{1}^\e \1_2)\ucr{3}^\e \fb)\ast ((12)_4(23)_4)\\
=& ((\1_2 \ucr{1}^\e \1_2)\ucr{3}^\e \fb)\ast (123)_4,\\
((\1_2\ucr{2}^\e \1_2)\ucr{3}^\e \fb)\ast (132)_4=
  & ((\1_2 \ucr{1}^\e \1_2)\ucr{3}^\e \fb)\ast ((12)_4(132)_4)\\
=& ((\1_2 \ucr{1}^\e \1_2)\ucr{3}^\e \fb)\ast (13)_4.
\end{align*}
From 
$\fb\ucr{1}^\e (\1_2\ucr{2}^\e \1_2)=\fb\ucr{1}^\e (\1_2\ucr{1}^\e \1_2)$, 
it follows that
\[a[\1_2\ucr{2} (\fb \ucr{1} \fb-2\fb\ucr{2} \fb)
  +(\fb \ucr{1} \fb-2\fb\ucr{2} \fb)\ucr{2} \1_2]
=a[(\1_2\ucr{2} (\fb \ucr{1} \fb-2\fb\ucr{2} \fb))\ast (3,1,2,4) 
+ (\fb \ucr{1} \fb-2\fb\ucr{2} \fb)\ucr{1}\1_2].\]
On the other hand, we consider the free Poisson algebra $\Pois(V)$
associated to the $\Bbbk$-vector space $V=\oplus_{i=1}^4 \Bbbk x_i$. Then
\begin{align*}
\; [\1_2\ucr{2} (\fb \ucr{1} \fb&-2\fb\ucr{2} \fb)
+(\fb \ucr{1} \fb-2\fb\ucr{2} \fb)\ucr{2} \1_2](x_1, x_2, x_3, x_4)\\
=& x_1\{\{x_2, x_3\}, x_4\}-2x_1\{x_2, \{x_3, x_4\}\}+\{\{x_1, x_2x_3\}, x_4\}-2\{x_1, \{x_2x_3, x_4\}\}\\
=& -\{x_1, x_2\}\{x_3, x_4\}-\{x_1, x_3\}\{x_2, x_4\} -x_1\{x_2, \{x_3, x_4\}\}-x_1\{x_3, \{x_2, x_4\}\}\\
& -x_2\{x_1, \{x_3, x_4\}\}-x_2\{x_3, \{x_1, x_4\}\} -x_3\{x_1, \{x_2, x_4\}\}-x_3\{x_2, \{x_1, x_4\}\},
\end{align*}
\begin{align*}
\; [(\1_2\ucr{2} (\fb \ucr{1} \fb&-2\fb\ucr{2} \fb))\ast (123)_4 + (\fb \ucr{1} \fb-2\fb\ucr{2} \fb)\ucr{1}\1_2](x_1, x_2, x_3, x_4)\\
=& x_3\{\{x_1, x_2\}, x_4\}-2x_3\{x_1, \{x_2, x_4\}\}
 +\{\{x_1x_2, x_3\}, x_4\}-2\{x_1x_2, \{x_3, x_4\}\}\\
=& \{x_1, x_4\}\{x_2, x_3\}+\{x_1, x_3\}\{x_2, x_4\} 
 -x_1\{x_2, \{x_3, x_4\}\}-x_1\{x_3, \{x_2, x_4\}\}\\
& -x_2\{x_1, \{x_3, x_4\}\}-x_2\{x_3, \{x_1, x_4\}\}
 -x_3\{x_1, \{x_2, x_4\}\}-x_3\{x_2, \{x_1, x_4\}\}.
\end{align*}
It follows that $a=0$.

\end{proof}

\begin{theorem}
\label{yythm7.6}
Let $\ip$ denote either $\Pois$ or $\Pois/{^k \iu}$ for some $k\geq 5$.
Then $\dim H^2_{\ast}(\ip)\leq 1$. As a consequence,
$\ip$ is $\idf$-$($semi$)$rigid.
\end{theorem}

By a result in the next section (Corollary \ref{yycor8.5}), if
${\rm{char}}\;\ \Bbbk=0$, the $H^2_{\ast}(\ip)\cong \Bbbk$.

\begin{proof}[Proof of Theorem \ref{yythm7.6}]
First we introduce some temporary notation.
Let $Z^2_4(\ip)$ (respectively, $B^2_4(\ip)$) denote the $\Bbbk$-vector space
of all 2-cocycles (respectively, 2-coboundary) with $\varsigma_{m=0}$ for all $m\leq 4$.
By Lemma \ref{yylem7.1}(3), every 2-cocycle of $\ip$
is equivalent to a 2-cocycle with $\varsigma_{m=0}$ for all $m\leq 4$.
This implies that $H^2_{\ast}(\ip)\cong Z^2_{4}(\ip)/B^2_{4}(\ip)$.

Now we can start with a 2-cocycle $\wv$ in $Z^2_4(\ip)$.
Let $\ip_{[\wv]}=(\ip[\epsilon], \ucr{i}^\epsilon, \ast^\epsilon)$
denote the infinitesimal deformation of $\ip$ corresponding to $\wv$.
Since $\varsigma_m=0$ for $m\leq 4$,
we have $\ast^{\epsilon}=\ast$
when applied on elements in $\ip_{[\wv]}(m)$ for $m\leq 4$.

Recall that $\Pois$ is generated by $\{\1_0, \1_2, \fb\}$ subject
to the relations given in \eqref{E3.5.1}-\eqref{E3.5.7}. We need to
analyze the compositions of these generators in $\ip=\Pois$ or $\Pois/{^k \iu}$.

Define a $\Bbbk$-linear map from $f: Z^2_4(\ip)\to \ip(3)$
by
$$f(\wv):=\wp_1(\1_2, \1_2)-\wp_2(\1_2, \1_2).$$
By Lemma \ref{yylem7.2}(4), $f(\wv)$ is the 1-dimensional
space $\Bbbk (\fb\ucr{1}\fb -\fb\ucr{2}\fb)$. Let $K$ be the kernel
of $f$. We claim that $K$ is a subspace of $B^2_4(\ip)$.
If the claim holds, then
$$\dim H^2_{\ast}(\ip)=\dim Z^2_4(\ip)/B^2_4(\ip)
\leq \dim Z^2_4(\ip)/K\leq \dim \Bbbk (\fb\ucr{1}\fb -\fb\ucr{2}\fb)=1.$$
Therefore it is enough to prove the claim.

Suppose $\wv$ is in $K$. Then  $\wp_1(\1_2,\1_2)-\wp_2(\1_2,\1_2)=0$
or $\1_2 \ucr{1}^{\epsilon} \1_2=\1_2 \ucr{2}^{\epsilon} \1_2$ by
Lemma \ref{yylem7.2}(4). This yields equation \eqref{E3.5.5} for
$\ucr{i}^{\epsilon}$. By Lemma \ref{yylem7.5},
we have
\[\fb \ucr{1}^{\epsilon} \1_2=\1_2\ucr{2}^{\epsilon}\fb
+(\1_2\ucr{2}^{\epsilon}\fb)\ast (12)_3\]
which is equation \eqref{E3.5.6} for $\ucr{i}^{\epsilon}$.
By Lemma \ref{yylem7.3}(3), Equation \eqref{E3.5.7} holds for $\ucr{i}^{\epsilon}$.
Now other equations \eqref{E3.5.1}-\eqref{E3.5.4} hold by Lemmas
\ref{yylem7.2}(1, 2) and \ref{yylem7.3}(1). Up to this point, we have
shown that \eqref{E3.5.1}-\eqref{E3.5.7} hold for the elements
$\{\1_0, \1_1, \1_2, \fb\}$ in $\ip_{[\wv]}$. Therefore there is an operadic
morphism $\Phi$ from $\Pois$ to $\ip_{[\wv]}$ sending $\1_0\mapsto \1_0,
\1_2\mapsto \1_2, \fb\mapsto \fb$. If $\ip=\Pois$, by Lemma \ref{yylem5.1},
$\wv$ is a 2-coboundary as desired. Now suppose that $\ip=\Pois/{^k \iu}$
for some $k\geq 5$. Then $\GKdim \ip=k$ by \cite[Theorem 0.1(2)]{BYZ}.
This implies that $\GKdim \Pois/\ker \Phi\leq k$ where $\ker\Phi$ is the
kernel of the operadic morphism $\Phi$. By \cite[Theorem 0.1(2)]{BYZ} again,
$\ker \Phi \supseteq {^k \iu}$ of $\Pois$. As a consequence,
$\Phi$ induces an operadic morphism $\Phi': \ip:=\Pois/{^k \iu}\to
\ip_{[\wv]}$. Since $\Phi'$ sends $\1_0\mapsto \1_0, \1_2\mapsto \1_2, \fb\mapsto \fb$,
it follows from Lemma \ref{yylem5.1} that $\wv$ is a 2-coboundary as
desired.
\end{proof}

\begin{remark}
\label{yyrem7.7}
Let $\ip=\Pois/{^k \iu}$ for all $\k\geq 1$.
\begin{enumerate}
\item[(1)]
If $k=1$ or $2$, $\ip\cong \Com$ and by Theorem
\ref{yythm6.8}, $H^2_{\ast}(\ip)=0$.
\item[(2)]
If $k=3$ and ${\rm{char}}\; \Bbbk\neq 2$, then
$\ip\cong \As/{^3 \iu}$ [Theorem \ref{yythm3.8}(2)].
By Theorem \ref{yythm6.8}, $H^2_{\ast}(\ip)=0$.
\item[(3)]
By Theorem \ref{yythm7.6} and Corollary \ref{yycor8.5},
if ${\rm{char}}\;\Bbbk=0$, then $H^2_{\ast}(\Pois)
=H^2(\Pois)=\Bbbk$. We expect the statement holds
in positive characteristic. Is there a direct
way of showing
$H^2(\ip)\neq 0$ for $k\geq 5$ and any fields $\Bbbk$?
\item[(4)]
We don't know the exact statement for
$H^2_{\ast}(\Pois/{^4 \iu})$ (and
$H^2_{\ast}(\Pois/{^3 \iu})$ when ${\rm{char}}\; \Bbbk=2$).
\item[(5)]
The computation of $H^2_{\ast}$ for a general $\ip$ could be very
complicated. It would be useful to develop effective methods.
\end{enumerate}
\end{remark}

\section{Formal deformations of operads}
\label{yysec8}

In this section we recall the definition of a formal deformation of an operad.
Most of this section is either well known or a folklore, see \cite{NN, Re}.
But a formal deformation in this paper does not necessarily satisfy the equivariance property.
This freedom is convenient when we are working with some examples.

Here we assume that $\Bbbk$ is a base commutative ring.
When we are working with examples such as $\As$, $\Com$, $\Lie$, $\Pois$,
we usually assume that $\Bbbk$ is a field.
Even when $\Bbbk$ is a field, the formal power series ring $\Bbbk[[t]]$ is not a field.
(But the Laurent power series ring $\Bbbk[[t^{\pm 1}]]$ is again a field.)
Let $\ip$ be an operad over $\Bbbk$.
Then we use $\ip[[t]]$ (respectively, $\ip[[t^{\pm 1}]]$)
to denote the operad $\ip\otimes_{\Bbbk} \Bbbk[[t]]$ over $\Bbbk[[t]]$
(respectively, $\ip\otimes_{\Bbbk} \Bbbk[[t^{\pm 1}]]$ over $\Bbbk[[t^{\pm 1}]]$).

Recall that, for every $n\geq 0$,
\[\ip[[t]](n):=(\ip(n))[[t]]=\ip(n)\otimes_{\Bbbk} \Bbbk[[t]].\]

\begin{definition}
\label{yydef8.1}
Let $\ip$ be an operad over $\Bbbk$.
\begin{enumerate}
\item[(1)]
A {\it formal deformation} of an operad $\ip$ is a
$\Bbbk[[t]]$-linear operadic structure on the $\Bbbk[[t]]$-module
$\ip[[t]]$, with partial composition $-\ucr{i}^t -$
and right $\S$-action $\ast^t$ that satisfy the following conditions:
\begin{enumerate}
\item[(1a)]
for all $m, n, 1\leq i\leq m$,
\begin{align*}
-\underset{i}{\circ^t} -\colon \ip[[t]](m) &\otimes_{\Bbbk[[t]]}\ip[[t]](n) \to \ip[[t]](m+n-1)\\
&\mu \underset{i}{\circ^t} \nu=\mu \ucr{i} \nu +\sum_{j=1}^{\infty} \wp_i^{j}(\mu,\nu) t^j
\end{align*}
is $\Bbbk[[t]]$-linear, where
\[\wp_i^{j}(-,-)\colon \ip(m) \otimes \ip(n) \to \ip(m+n-1)\]
is a $\Bbbk$-linear map for all $j\ge 1$.
\item[(1b)]
for each $m\geq 0$,
\begin{align*}
\ast^t\colon & \ip[[t]](m) \otimes_{\Bbbk[[t]]} \Bbbk[[t]] \S_m\to \ip[[t]](m)\\
&\mu \ast^t \sigma =\mu \ast \sigma +\sum_{j=1}^{\infty} \varsigma^j(\mu, \sigma) t^j
\end{align*}
is $\Bbbk[[t]]$-linear, where
\begin{align*}
\varsigma^j: \ip(m)\otimes \Bbbk \S_m \to \ip(m), \; \forall \; m\geq 0
\end{align*}
is $\Bbbk$-linear for all $j\ge 1$.

A formal deformation of $\ip$ is denoted by any of the following
notation
$$\wv^{\bullet} \quad {\text{or}}\quad
(\wp_i^{\bullet},\varsigma^{\bullet}) \quad {\text{or}}\quad
\ip_{[\wv^{\bullet}]} \quad {\text{or}}\quad
(\ip[[t]], \ucr{i}^{t},\ast^t).$$
\end{enumerate}
\item[(2)]
A formal deformation of $\ip$ is called an {\it $\S$-formal deformation} if
$\varsigma^j=0$ for all $j\geq 1$.
\item[(3)]
Two infinitesimal deformations
$(\ip[[t]], \ucr{i}^{t}, \ast^{t})$
and $(\ip[[t]], \underset{i}{\widetilde{\circ}}^{t}, \widetilde{\ast}^{t})$ are called {\it equivalent} if there is an
automorphism $\Phi$ of the $\Bbbk[[t]]$-module $\ip[[t]]$ such that
\begin{enumerate}
\item[(3a)]
$\Phi$ is an isomorphism of operads from
$(\ip[[t]], \underset{i}{\circ^{t}},
\ast^{t})$ to $(\ip[[t]], \underset{i}{\widetilde{\circ}}^{t}, 
\widetilde{\ast}^{t})$, and
\item[(3b)]
there is a sequence of $\Bbbk$-linear maps $\partial^j:\ip\to\ip$ such that
$\Phi(\mu)=\mu+\sum_{j=1}^{\infty}\partial^j (\mu) t^j$ for all $\mu\in \ip(m)$.
\end{enumerate}
\item[(4)]
A formal deformation is called {\it trivial} if it is equivalent to
the trivial formal deformation defined by zero $\wv^{\bullet}$-collection.
\item[(5)]
The set of formal deformations of $\ip$ modulo the equivalent relation
defined in part (3) is denoted by $\fdf (\ip)$.
\item[(6)]
We say $\ip$ is {\it $\fdf$-rigid} if every formal deformation of
$\ip$ is trivial.
\end{enumerate}
\end{definition}

One of the main results in this section is to show that
$\As,\Com, \Lie$ are $\fdf$-rigid, which is a consequence
of the following theorem together with Theorems \ref{yythm6.1},
\ref{yythm6.4} and \ref{yythm6.10}.

\begin{theorem}
\label{yythm8.2}
Suppose that $H^2_{\ast}(\ip)=0$, or equivalently, $\ip$ is
$\idf$-rigid. Then $\ip$ is $\fdf$-rigid.
\end{theorem}

The proof of Theorem \ref{yythm8.2} follows from
the standard argument in deformation theory. We say that
a formal deformation $\wv^{\bullet}$ is {\it $n$-away}
if $\wp^j=0$ and $\varsigma^j=0$ for all $j<n$.
We start with some lemmas.

\begin{lemma}
\label{yylem8.3}
Let $\ip$ be an operad over $\Bbbk$.
\begin{enumerate}
\item[(1)]
If $\wv^{\bullet}$ is $n$-away for some positive
integer $n$, then $(\wv^n,\varsigma^n)$ is a 2-cocycle.
\item[(2)]
Suppose that $\wv^{\bullet}$ is $n$-away with
corresponding
formal deformation $\ip^n_{[\wv^{\bullet}]}$
and that $(\wv^n,\varsigma^n)$ is a 2-coboundary.
Then $\ip^n_{[\wv^{\bullet}]}$ is equivalent to
an $(n+1)$-away formal deformation
$\ip^{n+1}_{[\wv^{\bullet}]}$ with
an isomorphism $\Phi^n: \ip^n_{[\wv^{\bullet}]}\to
\ip^{n+1}_{[\wv^{\bullet}]}$ such that
$$\Phi^n(\mu)=\mu+ \partial^n(\mu) t^n$$
where $\partial^n$ provides the
2-coboundary $(\wv^n,\varsigma^n)$, namely,
$(\wv^n,\varsigma^n)=\wv(\partial^n)$.
\end{enumerate}
\end{lemma}

\begin{proof} Basically the proof of Theorem
\ref{yythm2.10} can be adapted here. Details
are omitted.
\end{proof}

\begin{proof}[Proof of Theorem \ref{yythm8.2}]
Every formal deformation is obviously $1$-away,
which is denoted by $\ip^1_{[\wv^{\bullet}]}$.
By Lemma \ref{yylem8.3}(1), $(\wp^1, \varsigma^1)$
is a 2-cocycle. Since $H^2_{\ast}(\ip)=0$,
it is also a 2-coboundary. By Lemma \ref{yylem8.3}(2),
$\ip^1_{[\wv^{\bullet}]}$ is isomorphic to
$\ip^2_{[\wv^{\bullet}]}$ via isomorphism
$\Phi^1$. By induction and Lemma \ref{yylem8.3},
there are a sequence of isomorphisms
$\Phi^n$ from $\ip^n_{[\wv^{\bullet}]}$ to
Let $\Bbbk$
$\ip^{n+1}_{[\wv^{\bullet}]}$. It follows
from the form of $\Phi^n$, the infinite
product
$$\cdots\cdots \Phi^{n}\Phi^{n-1}\cdots \Phi^1$$
is an isomorphism from $\ip^1_{[\wv^{\bullet}]}$
to $\ip^{\infty}_{[\wv^{\bullet}]}$ where
$\ip^{\infty}_{[\wv^{\bullet}]}$ is the trivial
deformation of $\ip$. Therefore $\ip^1_{[\wv^{\bullet}]}$
is equivalent to the trivial one. Thus the
assertion follows.
\end{proof}

The following operad is also well known, first due to
Livernet-Loday \cite{LL}, is now called the Livernet-Loday
operad. See \cite[Example 3]{MaR} and \cite[Definition 4]{Do}
for some discussions. Below we are using the presentation
given in \cite[Example 3]{MaR}.

\begin{example}
\label{yyex8.4}
In \cite[Example 3]{MaR}, $\Bbbk$ is the complex field ${\mathbb C}$,
but it is clear that the construction works for any field $\Bbbk$ of characteristic zero.
It is less clear if this works for a field of
positive characteristic. The Livernet-Loday operad,
denoted by $\mathcal{LL}_{t}$, is a quadratic
operad over $\Bbbk[[t]]$ generated by a commutative operation
$(a, b) \mapsto a \cdot b$ and a skew-commutative
operation $(a, b) \mapsto [a, b]$ satisfying the identities
\begin{enumerate}
\item[(i)]
$[a \cdot b, c] = a \cdot [b, c] + [a, c] \cdot b,$
\item[(ii)]
$[a, [b, c]] + [b, [c, a]] + [c, [a, b]] = 0,$
\item[(iii)]
$(a \cdot b) \cdot c - a \cdot (b \cdot c) = t
[b, [a, c]].$
\end{enumerate}
By \cite[Example 3]{MaR}, when $t=0$, $\mathcal{LL}_t$ becomes
$\Pois$. Or equivalently, $\mathcal{LL}_t/t\mathcal{LL}_t \cong
\Pois$. By \cite[Example 3]{MaR}, when $t\neq 0$, $\mathcal{LL}_t$ is
isomorphic to $\As$. This means that there is a $\Bbbk[[t^{\pm 1}]]$-linear
operadic isomorphism from $\mathcal{LL}_t\otimes_{\Bbbk[[t]]}\Bbbk[[t^{\pm 1}]]$
to $\As\otimes_{\Bbbk} \Bbbk[[t^{-1}]]$. Using the relations given in (i)-(iii),
one can check that $\mathcal{LL}_t(n)$, for each $n$, is $t$-torsionfree
and of finite rank over $\Bbbk[[t]]$. Since $\Bbbk[[t]]$ is a DVR,
we obtain that $\mathcal{LL}_t$ is isomorphic to $\As\otimes_{\Bbbk} \Bbbk[[t]]$ (and then
to $\Pois\otimes_{\Bbbk} \Bbbk[[t]]$) as $\Bbbk[[t]]$-modules. Then $\mathcal{LL}_t$
is a formal deformation of $\Pois$.

We claim that $\mathcal{LL}_t$ is not a trivial formal deformation.
By \cite[Example 3]{MaR}, $\mathcal{LL}_t\otimes_{\Bbbk[[t]]}\Bbbk[[t^{\pm 1}]]$
is isomorphic to $\As\otimes_{\Bbbk} \Bbbk[[t^{-1}]]$ as an operad
over $\Bbbk[[t^{-1}]]$. This implies that
$\mathcal{LL}_t\otimes_{\Bbbk[[t]]}\Bbbk[[t^{\pm 1}]]$
is not isomorphic to $\Pois\otimes_{\Bbbk}\Bbbk[[t^{\pm 1}]]$ as an
operad over $\Bbbk[[t^{\pm 1}]]$.
As a consequence, $\mathcal{LL}_t$
is not isomorphic to $\Pois\otimes_{\Bbbk}\Bbbk[[t]]$ as an
operad over $\Bbbk[[t]]$.
This says that $\mathcal{LL}_t$ is not isomorphic to a
trivial formal deformation, or equivalently,
$\mathcal{LL}_t$ is not a trivial formal deformation.
\end{example}

\begin{corollary}
\label{yycor8.5}
Suppose that ${\rm{char}}\; \Bbbk=0$.
\begin{enumerate}
\item[(1)]
$\Pois$ has a non-trivial formal deformation.
\item[(2)]
$H^2_{\ast}(\Pois)=H^2(\Pois)=\Bbbk$.
\end{enumerate}
\end{corollary}

\begin{proof}
(1) This is Example \ref{yyex8.4}.

(2) This follows from part (1), Theorems \ref{yythm2.10},
\ref{yythm7.6} and \ref{yythm8.2}.
\end{proof}



\subsection*{Acknowledgments}

Y.-H. Bao and Y. Ye were partially supported by the Natural 
Science Foundation of China (Grant No. 11871071,
11971449). Y.-H. Wang was partially supported by
the Natural Science Foundation of China (Grant No. 11971289) and
Foundation of China Scholarship Council (Grant No. [2016]3009).
J.J. Zhang was partially supported by the US National Science
Foundation (Grant No. DMS-1700825).

\end{document}